\documentclass[10pt]{amsart}

\usepackage[latin2]{inputenc}
\usepackage{amsmath}
\usepackage{graphicx}
\usepackage{amssymb}
\usepackage{esint}
\usepackage[dvipsnames]{xcolor}
\usepackage{tikz}
\usepackage{xxcolor}
\usepackage{floatrow}
\usepackage{color}
\usepackage{amsthm}
\usepackage{epsfig}
\usepackage[english]{babel}

\newtheorem{theorem}{Theorem}

\newtheorem{proposition}[theorem]{Proposition}
\newtheorem{lemma}[theorem]{Lemma}

\newtheorem{definition}[theorem]{Definition}

\newtheorem*{theorem*}{Theorem}

\def\Xint#1{\mathchoice
{\XXint\displaystyle\textstyle{#1}}%
{\XXint\textstyle\scriptstyle{#1}}%
{\XXint\scriptstyle\scriptscriptstyle{#1}}%
{\XXint\scriptscriptstyle\scriptscriptstyle{#1}}%
\!\int}
\def\XXint#1#2#3{{\setbox0=\hbox{$#1{#2#3}{\int}$ }
\vcenter{\hbox{$#2#3$ }}\kern-.6\wd0}}

\def\dashint{\Xint-}

\allowdisplaybreaks[2]

\newcommand{\Id}{\operatorname{Id}} 
\newcommand{\supp}{\operatorname{supp}}

\newcommand{\dist}{\operatorname{dist}}

\newcommand{\RVE}{{\operatorname{RVE}}}

\newcommand{\Var}{{\operatorname{Var}~}}
\newcommand{\Cov}{\operatorname{Cov}}
\newcommand{\osc}{{\operatorname{osc}}}

\newcommand{\shom}{{\mathsf{hom}}}
\newcommand{\per}{{\mathsf{per}}}

\definecolor{Yellow}{rgb}{0.95,0.9,0.0} 
\definecolor{Red}{rgb}{0.8,0.1,0.1}
\definecolor{Green}{rgb}{0.1,0.65,0.2}
\definecolor{Blue}{rgb}{0.1,0.1,0.8}
\definecolor{Purple}{rgb}{0.7,0.1,0.7}
\definecolor{Grey}{rgb}{0.6,0.6,0.6}

\begin{document}

\title[Normal approximation for sums with multilevel dependence]{Quantitative normal approximation for sums of random variables with multilevel local dependence structure}

\author{Julian Fischer}
\address{Institute of Science and Technology Austria (IST Austria),
Am Campus 1, 3400 Klosterneuburg, Austria, E-Mail:
julian.fischer@ist.ac.at}

\begin{abstract}
We establish a quantitative normal approximation result for sums of random variables with multilevel local dependencies. As a corollary, we obtain a quantitative normal approximation result for linear functionals of random fields which may be approximated well by random fields with finite dependency range. Such random fields occur for example in the homogenization of linear elliptic equations with random coefficient fields. In particular, our result allows for the derivation of a quantitative normal approximation result for the approximation of effective coefficients in stochastic homogenization in the setting of coefficient fields with finite range of dependence. The proof of our normal approximation theorem is based on a suitable adaption of Stein's method and requires a different estimation strategy for terms originating from long-range dependencies as opposed to terms stemming from short-range dependencies.
\end{abstract}


\thanks{This project was initiated while the author enjoyed the hospitality of the Hausdorff Research Institute for Mathematics, Bonn, as a participant of the Trimester Program ``Multiscale Problems: Algorithms, Numerical Analysis and Computation''. The author would like to thank Mitia Duerinckx and Antoine Gloria for interesting discussions on the manuscript.}

\maketitle

\section{Introduction}

Stein's method of normal approximation \cite{SteinOriginal,Stein} -- as well as Chatterjee's variant \cite{Chatterjee,Chatterjee2} -- are among the most widely used methods for the derivation of quantitative estimates on the deviation of a probability distribution from a Gaussian.
Stein's method is remarkably flexible, being applicable to the multivariate setting \cite{Goetze,GoldsteinRinott,RinottRotar,ChatterjeeMeckes,ReinertRoellin,
ChenGoldsteinShao} and to the setting of sums of ``locally dependent'' random variables
\cite{Stein,BaldiRinott,BaldiRinottStein,Rinott,ChenShao,RinottRotar3,
ChenGoldsteinShao}. A survey of quantitative normal approximation by Stein's method may be found e.\,g.\ in \cite{ChenGoldsteinShao,RossSurvey}.

In recent years, Chatterjee's variant \cite{Chatterjee,Chatterjee2} of Stein's method in the form of second-order Poincar\'e inequalities has found widespread use: For Gaussian random fields second-order Poincar\'e inequalities have been established in \cite{Chatterjee2,NourdinPeccatiReinert}, while for discrete probability distributions with product structure (like Poisson point processes) this has been accomplished in \cite{Chatterjee,LachiezeReyPeccati}. Random geometries like random sequential adsorption processes or Voronoi and Delaunay tesselations for random (e.\,g.\ Poisson) point distributions have been considered in \cite{PenroseYukich,SchreiberPenroseYukich,LastPeccatiSchulte,
DuerinckxGloriaPoincare}.

In the present work, we derive a result on normal approximation for random variables that arise from random fields with finite range of dependence, featuring multilevel dependencies (with strong local dependencies and decaying dependencies on larger dependency ranges). The lack of a (hidden) product structure -- as present in the abovementioned applications of Chatterjee's variant of Stein's method -- prevents us from using Chatterjee's variant of Stein's method. Instead, our approach is closer to the original version of Stein's method respectively the results on sums of random variables with local dependence structure \cite{Rinott,ChenShao,RinottRotar}.

The quantitative normal approximation results obtained in the present contribution have an interesting consequence for the homogenization of elliptic equations with random coefficient field: While so far the quantitative description of fluctuations had been limited to the setting of probability distributions of coefficient fields subject to a second-order Poincar\'e inequality \cite{Nolen,GloriaNolen,GuMourrat,MourratNolen,DuerinckxGloriaOtto}, our result enables us to develop a first quantitative description of certain fluctuations in stochastic homogenization -- namely, the fluctuations of the \emph{effective conductivity} of random periodic materials -- under the assumption of finite range of dependence \cite{FischerVarianceReduction}. It has also an interesting consequence for the numerical analysis of an algorithm by Le~Bris, Legoll, and Minvielle \cite{LeBrisLegollMinvielle} capable of increasing the accuracy of approximations for the effective conductivity, see \cite{FischerVarianceReduction}.
In fact, the results of the present work will enable us to develop a complete quantitative description of fluctuations in linear elliptic PDEs with random coefficient field in the case of finite range of dependence, see the forthcoming paper \cite{DuerinckxFischerGloriaOtto}.

\subsection{Random fields in stochastic homogenization}

Let us give a few remarks on random fields in stochastic homogenization, from which the present work draws its main motivation.
The solution $u\in H^1_0(\mathbb{R}^d)$ to a linear elliptic equation with a random coefficient field $a:\mathbb{R}^d \times \Omega \rightarrow \mathbb{R}^{d\times d}$
\begin{align}
\label{EllipticPDE}
-\nabla \cdot (a\nabla u)=f
\end{align}
(with a deterministic right-hand side $f\in L^2(\mathbb{R}^d)$) may -- provided that the random coefficient field $a$ is subject to e.\,g.\ the assumptions of uniform ellipticity and boundedness, stationarity, and finite range of dependence $\varepsilon\leq \frac{1}{2}$ -- be approximated by the solution to an \emph{effective equation} of the form
\begin{align*}
-\nabla \cdot (a_\shom \nabla u_\shom) = f,
\end{align*}
where $a_\shom \in \mathbb{R}^{d\times d}$ is a constant (elliptic) effective coefficient which depends on the law of $a$ (but is invariant under spatial rescaling).
Namely, one has the following quantitative estimate \cite{ArmstrongKuusiMourrat,GloriaNeukammOttoTwoScale,GloriaOttoNew}: The difference of the solutions $u$ and $u_\shom$ is bounded for a suitable $p(d)\geq 2$ by
\begin{align*}
||u-u_\shom||_{L^p(\mathbb{R}^d)} \leq 
\begin{cases}
\mathcal{C}(a,f) ||f||_{L^2} \varepsilon |\log \varepsilon|&\text{for }d=2,
\\
\mathcal{C}(a,f) ||f||_{L^2} \varepsilon &\text{for }d\geq 3,
\end{cases}
\end{align*}
where $\mathcal{C}(a,f)$ denotes a random constant with almost Gaussian stochastic moments
\begin{align*}
\mathbb{E}\bigg[\exp\bigg(\frac{\mathcal{C}(a,f)^{2-\delta}}{C(d,\lambda,\delta)}\bigg)\bigg]\leq 2.
\end{align*}
Note that the fluctuations of functionals of the solution $u$ like $\int u \eta \,dx$ for a smooth compactly supported $\eta\in C^\infty_{cpt}(\mathbb{R}^d)$ are expected to be of the order of the central limit theorem scaling $\varepsilon^{d/2}$. This suggests that for $d\geq 3$ it should be possible to achieve a higher-order deterministic approximation for $u$. As shown in \cite{BellaFehrmanFischerOtto,BenoitGloria,Gu}, in weaker norms one may in fact obtain such a higher-order approximation for $u$ in terms of solutions to deterministic equations, based on the concept of higher-order homogenization correctors (which also plays an important role in higher-order regularity theory for elliptic equations with periodic or random coefficient field, see \cite{AvellanedaLin,FischerOtto}).

However, to achieve a description of the fluctuations of the solution $u$ a different approach is required. A description of the correlation structure of fluctuations of the homogenization corrector $\phi_i$ (see below for a definition) had been obtained in \cite{MourratOtto}; subsequently, in \cite{GuMourrat} a description of the correlation structure of fluctuations in the solutions $u$ to the equation \eqref{EllipticPDE} was given. The first quantitative normal approximation result in the context of stochastic homogenization was obtained for the effective conductivity on the torus in \cite{Nolen} and subsequently \cite{GloriaNolen}, following earlier qualitative results in \cite{BiskupSalviWolff,Rossignol}; for linear functionals of the solution $u$ of the form $\int u \eta \,dx$ it was obtained in \cite{MourratNolen}.

A fundamental object in homogenization is the \emph{homogenization corrector} $\phi_i$, defined as the unique (up to addition of constants) solution with sublinear growth to the PDE
\begin{align*}
-\nabla \cdot (a(e_i+\nabla \phi_i))=0.
\end{align*}
The effective coefficient $a_\shom$ is determined by the homogenization corrector as
\begin{align}
\label{ahom}
a_\shom e_i:=\mathbb{E}[a(e_i+\nabla \phi_i)].
\end{align}

To faciliate the description of fluctuations in stochastic homogenization, in \cite{DuerinckxGloriaOtto} the concept of the \emph{homogenization commutator}\begin{align*}
\Xi_{ij} := (a-a_\shom)(e_i+\nabla \phi_i) \cdot e_j
\end{align*}
has been introduced (bearing its name due to the similarity of the expression with a commutator); see also \cite{ArmstrongKuusiMourrat} for a related energy quantity. The homogenization commutator $\Xi$ has strong locality properties: Under a suitable rescaling it converges to essentially white noise. Furthermore, it allows for the description of fluctuations of solutions to equations of the form $-\nabla \cdot (a\nabla u)=\nabla \cdot g$: Introducing the solutions $u_\shom$ to the homogenized equation $-\nabla \cdot (a_\shom\nabla u_\shom)=\nabla \cdot g$ and the solution $v_\shom$ to the equation $-\nabla \cdot (a_\shom^*\nabla v_\shom)=\nabla \cdot \eta$, one has
\begin{align*}
\int \eta \cdot (\nabla u-\mathbb{E}[\nabla u]) \,dx \approx \int \Xi\nabla u_\shom \cdot \nabla v_\shom \,dx
\end{align*}
up to an error of order $\varepsilon^{d/2+1}$ (which amounts to a relative error of order $\varepsilon$).
As shown in \cite{DuerinckxGloriaOtto} under the assumption of a second-order Poincar\'e inequality for the probability distribution of the coefficient field $a$, functionals of the homogenization commutator like $\int \Xi \eta \,dx$ display fluctuations on scale $\varepsilon^{d/2}$ while the distance of their probability distribution to a Gaussian is basically of order $\varepsilon^{d}$ (e.\,g.\ in the $1$-Wasserstein distance), amounting to a relative error of order $\varepsilon^{d/2}$. The results of our present work will allow for an analogous result in the setting of coefficient fields with finite range of dependence, see the upcoming work \cite{DuerinckxFischerGloriaOtto}.

To obtain an approximation for the effective coefficient $a_\shom$ which is given analytically by \eqref{ahom}, one typically considers a \emph{periodization} of the random coefficient field, that is an $L\varepsilon$-periodic random coefficient field $a_{\per,L}$ whose law on each cube of diameter $\frac{L\varepsilon}{2}$ coincides with the law of the original coefficient field $a$ on the same cube\footnote{Note that the existence of such a periodization is a priori unclear; even if it exists, it must be constructed on a case-by-case basis depending on the probability distribution of $a$.}.
For such a periodization, the homogenization corrector $\phi_{\per,L,i}$ (defined as the $L\varepsilon$-periodic solution to the PDE $-\nabla \cdot (a_{\per,L}(e_i+\nabla \phi_{\per,L,i}))=0$) is an $L\varepsilon$-periodic function and one may for a single realization compute an approximation for the effective coefficient $a_\shom$ according to
\begin{align}
\label{aRVE}
a^\RVE e_i := \dashint_{[0,L\varepsilon]^d} a_{\per,L} (e_i+\nabla \phi_{\per,L,i}) \,dx.
\end{align}
The error of this approximation is dominated by the fluctuations of $a^\RVE$, which are of the order
\begin{align*}
\mathbb{E}\big[\big|a^\RVE-\mathbb{E}[a^\RVE]\big|^2\big]^{1/2} \leq C L^{-d/2}
\end{align*}
(see \cite{GloriaOtto,GloriaOtto2,GloriaOttoNew}),
while the systematic error is of higher order
\begin{align*}
\big|\mathbb{E}[a^\RVE]-a_\shom\big| \leq C L^{-d} |\log L|^d.
\end{align*}
As shown in \cite{FischerVarianceReduction}, the quantitative result on normal approximation established in the present work has an interesting consequence: It facilitates a rigorous mathematical analysis of the \emph{selection approach for representative volumes} introduces by Le~Bris, Legoll, and Minvielle \cite{LeBrisLegollMinvielle}. The selection approach of Le~Bris, Legoll, and Minvielle is a remarkably successful numerical algorithm for increasing the accuracy of the approximations for effective coefficients: It basically proceeds by selecting not a random sample of the coefficient field $a_\per$ for the computation of the approximation $a^\RVE$ for the effective coefficient \eqref{aRVE}, but a sample of the coefficient field that is ``particularly representative'' in the sense that it captures certain statistical properties of the random coefficient field -- like the spatial average $\dashint_{[0,L\varepsilon)^d} a \,dx$ -- exceptionally well. It has been observed numerically by Le~Bris, Legoll, and Minvielle that their selection approach achieves its gain in accuracy by reducing the fluctuations of the approximations \cite{LeBrisLegollMinvielle}. In \cite{FischerVarianceReduction}, it is shown that the approximate multivariate normality of the joint probability distribution of $a^\RVE$ and $\dashint_{[0,L\varepsilon)^d} a \,dx$ -- which is established using Theorem~\ref{TheoremNormalApproximationMultilevelLocalDependence} below -- allows for a rigorous analysis of the selection approach. In fact, it is this mathematical application that has dictated our choice of the distance between probability distributions in Definition~\ref{DefinitionDistance}.

{\bf Notation.}
For a vector $v\in \mathbb{R}^m$ we denote by $|v|$ its Euclidean norm; the vectors of the standard basis are denoted by $e_i$, $1\leq i\leq m$. We denote the identity matrix in $\mathbb{R}^{N\times N}$ by $\Id$ or $\Id_N$. For a matrix $A\in \mathbb{R}^{m\times m}$ we shall denote by $|A|$ its natural norm $|A|:=\max_{v,w\in \mathbb{R}^m,|v|=|w|=1} |v\cdot A w|$. Similarly, on the space of tensors $B\in \mathbb{R}^{m\times m\times m}$ we shall use the norm given by $|B|:=\max_{u,v,w\in \mathbb{R}^m,|u|=|v|=|w|=1} \sum_{i,j,k=1}^m B_{ijk}u_i v_j w_k$. For $x\in \mathbb{R}^d$ we denote by $|x|_\infty=\max_i |x_i|$ its supremum norm. By $|x-y|_\per:=\inf_{k\in \mathbb{Z}^d} |x-y-Lk|$ respectively (for sets) $\dist^\per(U,V):=\inf_{k\in \mathbb{Z}^d} \dist(U,k+V)$, we denote the periodicity-adjusted distance in the context of the torus $[0,L]^d$. By $|x-y|_{\per,\infty}$ and $\dist^\per_\infty$, we denote the corresponding distances associated with the maximum norm.

Given a positive definite symmetric matrix $\Lambda\in \mathbb{R}^{N\times N}$, we denote the Gaussian with covariance matrix $\Lambda$ by
\begin{align*}
\mathcal{N}_{\Lambda}(x):=\frac{1}{(2\pi)^{N/2}\sqrt{\det \Lambda}}
\exp\bigg(-\frac{1}{2}\Lambda^{-1} x \cdot x\bigg).
\end{align*}
For $\gamma>0$, we equip the space of random variables $X$ with stretched exponential moment $\mathbb{E}[\exp(|X|^\gamma/a)]<\infty$ for some $a=a(X)>0$ with the norm $||X||_{\exp^\gamma}:=\sup_{p\geq 1} p^{-1/\gamma} \mathbb{E}[|X|^p]^{1/p}$. For a discussion of this choice of norm, see Appendix~\ref{AppendixStretchedExponential}.

For a map $f:\mathbb{R}^N\rightarrow V$ into a normed vector space $V$, we denote for any $r>0$ by $\osc_r f(x_0):=\sup_{x,y\in \{|x-x_0|\leq r\}} |f(x)-f(y)|_V$ its oscillation in the ball of radius $r$ around $x_0$.

For two random variables $X$ and $Y$ (possibly defined on different probability spaces), we denote equality in law by $X\stackrel{d}{=}Y$. For a vector-valued random variable $X$, we denote by $\Var X$ the covariance matrix of its entries. Similarly, for two vectors $X$ and $Y$ we denote by $\Cov[X,Y]$ the matrix of covariances of the entries of $X$ and the entries of $Y$. Given a condition like $f\leq b$ for some random variable $f$ and some $b\in \mathbb{R}$, we denote by $\chi_{f\leq b}$ the characteristic function of the set $\{\omega\in \Omega:f(\omega)\leq b\}$.

On the space of bounded fields $v:\mathbb{R}^d \rightarrow \mathbb{R}^N$ we shall use the $L^p_{loc}(\mathbb{R}^d)$ topology for any $1<p<\infty$. By slight abuse of notation, by $\mathcal{W}_1(X,Y)$ respectively $\mathcal{W}_1(X,\mathcal{N}_\Lambda)$ we denote the $1$-Wasserstein distance between the law of the two random variables $X$ and $Y$ respectively between the law of the random variable $X$ and the Gaussian $\mathcal{N}_\Lambda$.

For two subsets $U$ and $V$ of $\mathbb{R}^k$, we denote as usual by $U+V$ the Minkowski sum $U+V:=\{y+z:y\in U,z\in V\}$. Similarly, for a subset $U\subset \mathbb{R}^k$, a vector $x\in \mathbb{R}^k$, and a scalar $\lambda>0$, we denote by $x+U$ the translation of $U$ by $x$ and by $\lambda U$ the set $\lambda U:=\{\lambda y:y\in U\}$. By $B_r$ we denote the ball of radius $r$ around $0$.

\section{Main Results}

Throughout our present work, the following assumption of \emph{finite range of dependence} is the central assumption on the random fields.
\begin{itemize}
\item[(A)] Let $(\Omega,\mathcal{F},\mathbb{P})$ be a probability space and let $a:\mathbb{R}^d \times \Omega\rightarrow \mathbb{R}^n$ (with $d,n\in\mathbb{N}$) be a random field. We say that $a$ has range of dependence $1$ if for any two measurable sets $U,V\subset\mathbb{R}^d$ with $\dist(U,V)>1$ the restrictions $a|_U$ and $a|_V$ are stochastically independent.
\item[(A')] Let $(\Omega,\mathcal{F},\mathbb{P})$ be a probability space and let $a:\mathbb{R}^d \times \Omega\rightarrow \mathbb{R}^n$ (with $d,n\in\mathbb{N}$) be an almost surely $L$-periodic random field for some $L\geq 1$. We say that $a$ has range of dependence $1$ if for any two measurable sets $U,V\subset\mathbb{R}^d$ with $\dist_\per(U,V)>1$ the restrictions $a|_U$ and $a|_V$ are stochastically independent.
\end{itemize}

We will establish two main results: A quantitative multivariate normal approximation result for linear functionals of random fields that admit a good approximation in terms of finite-range random fields on the one hand (see Theorem~\ref{TheoremRandomField}), as well as a quantitative multivariate normal approximation result for sums of random variables with multilevel local dependence on the other hand (see Theorem~\ref{TheoremNormalApproximationMultilevelLocalDependence}).

The distance of these probability distributions to a multivariate Gaussian will be quantified through the following notion of distance between probability measures. Note that this distance is a standard choice in the theory of multivariate normal approximation, see e.\,g.\ \cite{ChenGoldsteinShao} and the references therein. Note also that our choice of distance $\mathcal{D}$ dominates the $1$-Wasserstein distance.

\begin{definition}
\label{DefinitionDistance}
For a symmetric positive definite matrix $\Lambda\in \mathbb{R}^{N\times N}$ and $\bar L<\infty$, we consider the classes $\Phi_{\Lambda}^{\bar L}$ of functions $\phi:\mathbb{R}^N\rightarrow \mathbb{R}$ with the following properties:
\begin{itemize}
\item $\phi$ is smooth and its first derivative is bounded in the sense $|\nabla \phi(x)| \leq \bar L$ for all $x\in \mathbb{R}^N$.
\item For any $r>0$ and any $x_0\in \mathbb{R}^N$, we have
\begin{align}
\label{DefinitionClassPhi}
\int_{\mathbb{R}^N} \osc_r \phi(x) ~\mathcal{N}_{\Lambda}(x-x_0) \,dx \leq r,
\end{align}
where $\osc_r \phi (x)$ is the oscillation of $\phi$ defined as
\begin{align*}
\osc_r\phi(x):=\sup_{|z|\leq r}\phi(x+z)-\inf_{|z|\leq r} \phi(x+z)
\end{align*}
and where 
\begin{align*}
\mathcal{N}_{\Lambda}(x):=\frac{1}{(2\pi)^{N/2}\sqrt{\det \Lambda}}
\exp\bigg(-\frac{1}{2}\Lambda^{-1} x \cdot x\bigg).
\end{align*}
\end{itemize}
The class $\Phi_\Lambda$ is defined as
\begin{align*}
\Phi_\Lambda:=\bigcup_{\bar L>0} \Phi_\Lambda^{\bar L}.
\end{align*}

Furthermore, we introduce the distance $\mathcal{D}$ and the regularized distance $\mathcal{D}^{\bar L}$ between the law of an $\mathbb{R}^N$-valued random variable $X$ and the $N$-variate Gaussian $\mathcal{N}_\Lambda$ as
\begin{align}
\label{DefinitionD}
\mathcal{D}(X,\mathcal{N}_\Lambda) := \sup_{\phi\in \Phi_\Lambda} \bigg(\mathbb{E}[\phi(X)]-\int_{\mathbb{R}^N} \phi(x) \mathcal{N}_\Lambda(x)\,dx \bigg)
\end{align}
and
\begin{align}
\label{DefinitionDL}
\mathcal{D}^{\bar L}(X,\mathcal{N}_\Lambda) := \sup_{\phi\in \Phi_\Lambda^{\bar L}} \bigg(\mathbb{E}[\phi(X)]-\int_{\mathbb{R}^N} \phi(x) \mathcal{N}_\Lambda(x)\,dx \bigg).
\end{align}
\end{definition}
By the definition of our distance $\mathcal{D}$, the regularization $|\nabla\phi|\leq \bar L$ in the definition of the test functions for the distance $\mathcal{D}^{\bar L}$ may be removed by letting $\bar L\rightarrow\infty$. We shall prove most of our statements first in the regularized setting $\bar L<\infty$ -- note that for random variables $X$ with finite first moment the distance $\mathcal{D}^{\bar L}$ is guaranteed to be finite -- and then extend them to $\mathcal{D}$ by passing to the limit $\bar L\rightarrow\infty$.

Note that for $\bar L>1$ the distance $\mathcal{D}^{\bar L}$ is a stronger distance than the $1$-Wasserstein distance (while for $\bar L=1$ it coincides with the $1$-Wasserstein distance).

Let us also remark that \eqref{DefinitionClassPhi} entails (by letting $r\rightarrow 0$) the bound
\begin{align}
\label{ClassPhiDifferentialBound}
\int_{\mathbb{R}^N} |\nabla \phi|(x) \mathcal{N}_{\Lambda}(x-x_0) \,dx \leq 1
\end{align}
for any $x_0\in \mathbb{R}^N$.

For linear functionals of random fields $v$ which may be approximated well by random fields $v_r$ with finite range of dependence $r$ we establish the following normal approximation result.
\begin{theorem}
\label{TheoremRandomField}
Let $a:\mathbb{R}^d\times \Omega\rightarrow \mathbb{R}^n$, $d,n\in \mathbb{N}$, be a random field with finite range of dependence $1$ in the sense of (A). Let $v:\mathbb{R}^d\times \Omega\rightarrow \mathbb{R}^N$, $N\in \mathbb{N}$, be a random field with
\begin{align*}
\dashint_{\{|x-x_0|\leq 1\}} |v| \,dx \leq \mathcal{C}(a,x_0)
\end{align*}
for some random constant $\mathcal{C}(a,x_0)$ with stretched exponential stochastic moments $||\mathcal{C}(a,x_0)||_{\exp^\gamma}\leq 1$ for some $\gamma>0$ and any $x_0\in \mathbb{R}^d$.

Let $L\geq 1$ and $K\in \mathbb{N}$.
Suppose that there exists a family of random fields $v_r:\mathbb{R}^d\times \Omega\rightarrow \mathbb{R}^N$, $1\leq r\leq L$, such that $v_r$ is an $r$-local function of $a$ and an approximation for $v$ in the following sense:
\begin{itemize}
\item For any measurable $U\subset \mathbb{R}^d$ the restriction $v_r|_{U}$ is a measurable function of $a|_{U+B_r}$.
\item There exist random constants $\mathcal{C}(a,\psi)$ with stretched exponential stochastic moments $||\mathcal{C}(a,\psi)||_{\exp^\gamma}\leq 1$ such that for any $\psi \in W^{K,\infty}(\mathbb{R}^d)$ the estimate
%subject to the bound $|\nabla^k \psi(x)| \leq r^{-k} (1+|x|/L)^{-d-1}$ for all $0 \leq k\leq n$ and all $x\in \mathbb{R}^d$ the estimate
\begin{align}
\label{AssumptionApproximationByFiniteRange}
\bigg|\int_{\mathbb{R}^d} (v-v_r) \psi \,dx\bigg|
\leq \mathcal{C}(a,\psi) r^{-d} \int_{\mathbb{R}^d} \sup_{y\in B_r(x)} \sum_{k=0}^K r^{k} |\nabla^k \psi|(y) \,dx
\end{align}
holds.
\end{itemize}
Then for any $\xi\in W^{K,\infty}(\mathbb{R}^d)$ subject to the estimate
$|\nabla^k \xi|\leq L^{-k} (1+|x|/L)^{-d-1-k}$ for all $0\leq k\leq K$
the random variable
\begin{align*}
X:=L^{-d} \int_{\mathbb{R}^d} \xi v \,dx
\end{align*}
admits the normal approximation
\begin{align*}
&\mathcal{D}(X-\mathbb{E}[X],\mathcal{N}_{\Var X})
\leq 
C(d,\gamma) (\log L)^{C(d,\gamma)} \big(L^{-d} |\Var X^{1/2}| |\Var X^{-1/2}|^3\big) L^{-d}.
\end{align*}
\end{theorem}
Heuristically speaking, our theorem asserts the following: Suppose that $v$ is a random field with strongly localized dependencies in the sense that for any scale $r\geq 1$ it may be approximated by a random field $v_r$ with finite dependency range $r$ up to an error $v-v_r$ of order $r^{-d}$. Then any smoothly weighted average of the random field $v$ on a scale $L\geq 1$ will be approximately Gaussian, up to an error of the order $L^{-d} |\log L|^C$ in e.\,g.\ the $1$-Wasserstein distance, provided that the variance of this weighted average does not degenerate. Note that the fluctuations of such weighted averages $L^{-d} \int_{\mathbb{R}^d} \xi v \,dx$ are expected to be of the order of the CLT scaling $L^{-d/2}$; in other words, our theorem provides a relative error of order $L^{-d/2} |\log L|^C$.

In our second main result, we will make use of the following notion of ``multilevel local dependence decomposition''. An illustration of this decomposition is provided in Figure~\ref{FigureMultilevel}.

\begin{definition}[Sums of random variables with multilevel local dependence structure]
\label{ConditionRandomVariable}
Let $d\geq 1$ and $L\geq 2$. Consider a random field $a$ on $\mathbb{R}^d$ subject to the assumption of finite range of dependence (A) or an $L$-periodic random field subject to the assumption of finite range of dependence (A'). Let $X=X(a)$ be an $\mathbb{R}^N$-valued random variable depending on the random field.

We then say that $X$ is a sum of random variables with multilevel local dependence if there exist random variables $X_y^m=X_y^m(a)$, $0\leq m\leq 1+\log_2 L$ and $y\in 2^m \mathbb{Z}^d\cap [0,L)^d$, and constants $K\geq 2$, $\gamma\in (0,2]$, and $B\geq 1$ with the following properties:
\begin{itemize}
\item The random variable $X_y^m(a)$ is a measurable function of $a|_{y+K \log L \, [-2^m,2^m]^d}$.
\item We have
\begin{align*}
X=\sum_{m=0}^{1+\log_2 L} \sum_{y\in 2^m \mathbb{Z}^d\cap [0,L)^d} X_y^m.
\end{align*}
\item The random variables $X_y^m$ satisfy the bound
\begin{align}
\label{BoundMultilevelDependenceStructure}
||X_y^m||_{\exp^\gamma} \leq B L^{-d}.
\end{align}
\end{itemize}
\end{definition}

It is well-known that Stein's method of normal approximation allows to establish a quantitative result on normal approximation for sums of random variables with local dependence structure, see e.\,g.\ \cite{ChenGoldsteinShao,ChenShao,RinottRotar} and the references therein. However, in many applications global dependencies arise naturally: For example, the approximation of the effective coefficient in the homogenization of linear elliptic equations with random coefficient field -- that is, the random variable $a^\RVE$ as defined by \eqref{aRVE} -- features global dependencies.
It is shown in the companion article \cite{FischerVarianceReduction} that $a^\RVE$ may nevertheless be written as such a sum of random variables with a \emph{multilevel local dependence} structure.

\begin{figure}
\begin{tikzpicture}[scale=0.75]
\foreach \x in {0}
   \draw (16*\x+0.1,4*1.5+0.1) rectangle (16*\x+15.9,4*1.5+1.4);
\foreach \x in {0,1}
   \draw (8*\x+0.1,3*1.5+0.1) rectangle (8*\x+7.9,3*1.5+1.4);
\foreach \x in {0,1,2,3}
   \draw (4*\x+0.1,2*1.5+0.1) rectangle (4*\x+3.9,2*1.5+1.4);
\foreach \x in {0,1,2,3,4,5,6,7}
   \draw (2*\x+0.1,1*1.5+0.1) rectangle (2*\x+1.9,1*1.5+1.4);
\foreach \x in {0,1,2,3,4,5,6,7,8,9,10,11,12,13,14,15}
   \draw (\x+0.1,0*1.5+0.1) rectangle (\x+0.9,0*1.5+1.4);
\foreach \x in {0,1,2,3,4,5,6,7,8,9,10,11,12,13,14,15}
   \draw[draw=blue,fill=blue] (\x+0.1,-0.5+0.1) rectangle (\x+0.9,-0.5+0.2);
\foreach \x in {0,4,7,11,13,14}
   \draw[draw=red,fill=red] (\x+0.1,-0.5+0.1) rectangle (\x+0.9,-0.5+0.2);
%\node[scale=0.8] at (-4.5,0*1.5+0.75){$X_k^0:= L^{-d}\int_{4^0}^{4^1} \int_{B_k^4} |u_i(x,s)|^2 \,dx \,ds$};
%\node[scale=0.8] at (-4.5,1*1.5+0.75){$X_k^1:= L^{-d}\int_{4^1}^{4^2} \int_{B_k^3} |u_i(x,s)|^2 \,dx \,ds$};
%\node[scale=0.8] at (-4.5,2*1.5+0.75){$X_k^2:= L^{-d}\int_{4^2}^{4^3} \int_{B_k^2} |u_i(x,s)|^2\,dx \,ds$};
%\node[scale=0.8] at (-4.5,3*1.5+0.75){$X_k^3:= L^{-d}\int_{4^3}^{4^4} \int_{B_k^1} |u_i(x,s)|^2 \,dx \,ds$};
%\node[scale=0.8] at (-4.5,4*1.5+0.75){$X_k^4:= L^{-d}\int_{4^4}^{4^5} \int_{B_k^0} |u_i(x,s)|^2 \,dx \,ds$};
\foreach \x in {0,1,2,3,4,5,6,7,8,9,10,11,12,13,14,15}
		\node[scale=0.7] at (\x+0.5,0*1.5+0.75){$X_{\x}^0$};
\foreach \x in {0,1,2,3,4,5,6,7}
		\node[scale=0.7] at (2*\x+1.0,1*1.5+0.75){$X_{\x}^1$};
\foreach \x in {0,1,2,3}
		\node[scale=0.7] at (4*\x+2.0,2*1.5+0.75){$X_{\x}^2$};
\foreach \x in {0,1}
		\node[scale=0.7] at (8*\x+4.0,3*1.5+0.75){$X_{\x}^3$};
\foreach \x in {0}
		\node[scale=0.7] at (16*\x+8.0,4*1.5+0.75){$X_{\x}^4$};
\end{tikzpicture}
\caption{An illustration of the ``multilevel local dependence structure'' as introduced in Definition~\ref{ConditionRandomVariable} (in a one-dimensional setting). At the bottom, a sample of the random field $a$ is depicted; the $X_y^k$ may depend not only on the values of the random field directly below their box, but on the random field in a region that is wider by a factor of $K \log L$.\label{FigureMultilevel}}
\end{figure}
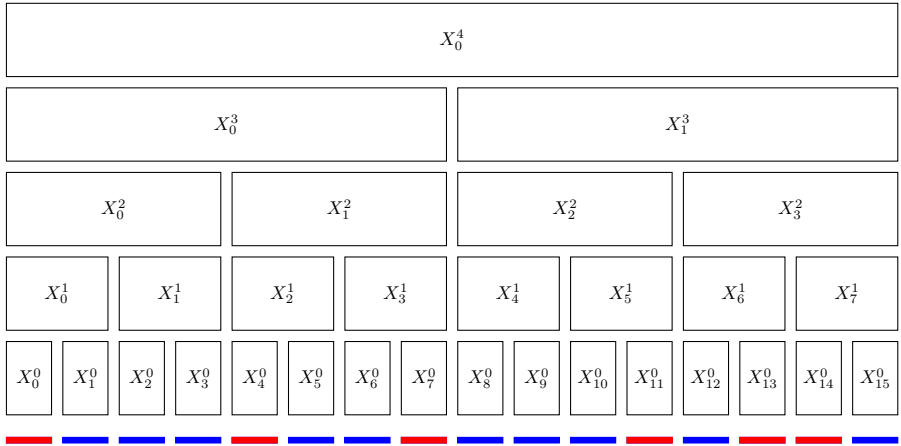

The (second) main result of the present work is the following quantitative central limit theorem for sums of vector-valued random variables with a multilevel local dependence structure, which is not covered by the typical normal approximation results for sums of random variables with a given dependency graph.

\begin{theorem}
\label{TheoremNormalApproximationMultilevelLocalDependence}
Let $d\geq 1$ and $L\geq 2$. Consider a random field $a$ on $\mathbb{R}^d$ subject to the assumption of finite range of dependence (A) or an $L$-periodic random field subject to the assumption of finite range of dependence (A'). Let $X=X(a)$ be an $\mathbb{R}^N$-valued random variable that may be written as a sum of random variables with multilevel local dependence in the sense of Definition~\ref{ConditionRandomVariable}. Then the law of the random variable $X$ is close to a multivariate Gaussian in the sense
\begin{align}
\label{NormalApproximationMultilevel}
&\mathcal{D}(X-\mathbb{E}[X],\mathcal{N}_\Lambda)
\leq
C(d,\gamma,N,K) B^3 (\log L)^{C(d,\gamma)} \big(L^{-d} |\Lambda^{1/2}| |\Lambda^{-1/2}|^3\big) L^{-d},
\end{align}
where $\Lambda:=\Var X$ and where the constant $C(d,\gamma,N,K)$ depends in a polynomial way on $d$, $N$ and $K$.

Furthermore, we have for any symmetric positive definite $\Lambda\in \mathbb{R}^{N\times N}$ with $\Lambda\geq \Var X$ and $|\Lambda-\Var X|\leq L^{-d}$
\begin{align}
\label{NormalApproximationMultilevelDegenerate}
\mathcal{D}(X-\mathbb{E}[X],\mathcal{N}_\Lambda)
\leq&
C(d,\gamma,N,K) B^3 (\log L)^{C(d,\gamma)} \big(L^{-d} |\Lambda^{1/2}| |\Lambda^{-1/2}|^3\big) L^{-d}
\\&
\nonumber
+C(d,N) (\log L)^{C(d,\gamma)} |\Lambda-\Var X|^{1/2},
\end{align}
providing a better bound in the case of degenerate covariance matrices $\Var X$.
\end{theorem}
For sums of random variables with multilevel local dependence structure, we also prove the following simple (and far from optimal) result on moderate deviations. Its proof makes use of the previous Theorem~\ref{TheoremNormalApproximationMultilevelLocalDependence} and an auxiliary concentration estimate provided in Lemma~\ref{iidSumTailBound} which is a consequence of Bennett's inequality.
\begin{theorem}
\label{TheoremModerateDeviations}
Let $d\geq 1$ and $L\geq 2$. Consider a random field $a$ on $\mathbb{R}^d$ subject to the assumption of finite range of dependence (A) or an $L$-periodic random field subject to the assumption of finite range of dependence (A').  Let $X=X(a)$ be an $\mathbb{R}^N$-valued random variable that may be written as a sum of random variables with multilevel local dependence structure $X=\sum_{m=0}^{1+\log_2 L} \sum_{i\in 2^m \mathbb{Z}^d \cap [0,L)^d} X_i^m$ in the sense of Definition~\ref{ConditionRandomVariable}.

Then there exists $\beta=\beta(d,\gamma)>0$ and a positive definite symmetric matrix $\Lambda\in \mathbb{R}^{N\times N}$ with $|\Lambda-\Var X|\leq C(d,\gamma,N,K) B^2 L^{-2\beta} L^{-d}$ such that for any measurable $A\subset \mathbb{R}^N$ we have the estimate
\begin{align*}
&\mathbb{P}\big[X-\mathbb{E}[X]\in A\big]
\\&
\leq \int_{\{z\in \mathbb{R}^N:\dist(z,A)\leq L^{-\beta} L^{-d/2}\}} \mathcal{N}_{\Lambda}(z) \,dz + C(d,\gamma,N,K) \exp\Big(-\frac{c}{B^C} L^{2\beta}\Big).
\end{align*}
\end{theorem}

\section{Normal Approximation with an Abstract Multilevel Dependency Structure}

\label{SectionNormalApproximationMain}

We now establish a result on quantitative normal approximation for a sum of random variables with a more abstract dependence structure allowing for multiple dependency ranges. More precisely, for a finite index set $I$ we consider a sum of random variables $X_i$
\begin{align*}
X:=\sum_{i\in I} X_i
\end{align*}
to each of which a ``dependency level'' $m(i)\in \mathbb{N}$ is assigned. In the application of our next result in the proof of Theorem~\ref{TheoremNormalApproximationMultilevelLocalDependence}, the dependency level $m(i)$ will correspond to a ``range of dependence'' of about $\sim 2^{m(i)}$ (up to factors of order $\log L$), i.\,e.\ each random variable $X_y^m$ from Definition~\ref{ConditionRandomVariable} will be assigned the dependency level $m$. However, the dependence structure we introduce now is more general than Definition~\ref{ConditionRandomVariable} (and in particular does not include any explicit reference to an underlying random field or even a spatial dimension).

The potential dependencies of the random variables $X_i$ shall be encoded by a matrix $\chi_{ij}\in \{0,1\}$ and a tensor $\chi_{ijk}\in \{0,1\}$ with the following property:
\begin{itemize}
\item For all $i\in I$, the random variable $X_i$ is independent from the collection of all random variables $X_j$ with $\chi_{ij}=0$, $j\in I$.
\item The matrix $\chi_{ij}$ is symmetric, i.\,e.\ $\chi_{ij}=\chi_{ji}$.
\item For all $i,j\in I$, the pair of random variables $(X_i,X_j)$ is independent from the collection of all random variables $X_k$ with $\chi_{ijk}=0$, $k\in I$.
\item The tensor $\chi_{ijk}$ is symmetric in its first two indices, i.\,e.\ $\chi_{ijk}=\chi_{jik}$.
\end{itemize}
Furthermore, we suppose that for any $j\in I$ and any $n\in \mathbb{N}$ with $n>m(j)$ there exists an assignment $i^n(j)$ which assigns the random variable $X_j$ of level $m(j)$ to another random variable $X_{i^n(j)}$ of the (higher) level $n$ with more possible dependencies:
\begin{itemize}
\item We have $m(i^n(j))=n$.
\item It must hold that $\chi_{i i^n(j)}\geq \chi_{ij}$, in other words $\chi_{ij}=1$ implies $\chi_{i i^n(j)}=1$.
\item For all $i$ and $k$ the inequality $\chi_{i i^n(j)k}\geq \chi_{ijk}$ must be true, in other words $\chi_{ijk}=1$ implies $\chi_{i i^n(j)k}=1$.
\end{itemize}
Note that this dependence structure is reminiscent of the local dependence structure in the results of \cite{Stein,BaldiRinott,BaldiRinottStein,Rinott,ChenShao,RinottRotar3,
ChenGoldsteinShao}, but in addition features multiple dependency ranges.

\begin{theorem}
\label{NormalApproximation}
Let $I$ be a finite index set, let $N\in \mathbb{N}$, and let $(\Omega,\mathcal{F},\mathbb{P})$ be a probability space. For any $i\in I$, let $X_i$ be an $\mathbb{R}^N$-valued random variable with vanishing expectation $\mathbb{E}[X_i]=0$ and finite third moment. To each of the $X_i$, let a number $m(i)\in \mathbb{N}$ be assigned.

Let $\chi_{ij}\in \{0,1\}$, $i,j\in I$, and $\chi_{ijk}\in \{0,1\}$, $i,j,k\in I$, be indicator functions for possible dependencies subject to the assumptions preceding the theorem. Let $m(i)$, $i\in I$, and $i^n(j)$, $i,j\in I$, $m(j)<n$, be as above.

Then the probability distribution of the sum
\begin{align}
\label{DefinitionX}
X:=\sum_{i\in I} X_i
\end{align}
may be approximated by an $N$-dimensional Gaussian with covariance matrix
\begin{align}
\label{DefinitionLambda}
\Lambda:=
\sum_{i\in I}\sum_{j\in I} \chi_{ij} \mathbb{E}[X_i\otimes X_j]
\end{align}
in the distance $\mathcal{D}(X,\mathcal{N}_\Lambda)$ in the following sense:

Introduce the abbreviations
\begin{subequations}
\label{NormalApproximationAbbreviations1}
\begin{align}
Z_{ij}&:=\sum_{k\in I:\chi_{ijk}=1} X_k,
\\
Z_{i}&:=\sum_{j\in I:\chi_{ij}=1} X_j,
\\
Y_{il}&:=\sum_{j\in I:m(j)<m(i),i^{m(i)}(j)=l,\chi_{ij}=1} \Big(X_i \otimes X_j-\mathbb{E}[X_i \otimes X_j]\Big),
\\
W_{ij}&:=X_i \otimes X_j-\mathbb{E}[X_i \otimes X_j],
\end{align}
\end{subequations}
and let $\bar X_i, \bar Z_{ij},\bar Z_i,\bar Y_{il}, \bar W_{ij}>0$ be arbitrary positive real numbers.

Let $\ell \in \mathbb{N}$ and $\varepsilon\in (0,\frac{1}{2}]$ be such that the condition
\begin{align}
\nonumber
&\sum_{i\in I,m(i)\leq\ell} \frac{N^{9/2} |\Lambda^{-1/2}|^3}{\varepsilon} \bigg(\sum_{j\in I:m(j)=m(i),\chi_{ij}=1} \big(\bar W_{ij} \bar Z_{ij}
+2 \bar Y_{ij}\, \bar Z_{ij}\big)
+\bar X_i \bar Z_{i}^2\bigg)
\\&~~~
\label{ConditionForNormalApproximation}
+\sum_{i\in I,m(i)>\ell} N^{4} |\Lambda^{-1}| |\log \varepsilon|
\bigg(\sum_{j\in I:m(j)=m(i),\chi_{ij}=1} \big(\bar W_{ij}+2\bar Y_{ij}\big)
+\bar X_i \bar Z_{i} 
\bigg)
\\&~~~
\nonumber
\leq \frac{1}{C}
\end{align}
is satisfied, where the universal constant $C=C(N)$ is given by the proof below.

Then the normal approximation result
\begin{align}
\label{NormalApproximationEstimate}
\mathcal{D} (X,\mathcal{N}_\Lambda)
\leq
C \sqrt{N} |\Lambda^{1/2}| \varepsilon
+\mathcal{R}_{lowlevel}
+\mathcal{R}_{alllevel}
+\mathcal{R}_{tail}
\end{align}
holds true, where
\begin{subequations}
\label{NormalApproximationAbbreviations2}
\begin{align}
\mathcal{R}_{lowlevel}&:=
\frac{CN^{9/2} |\Lambda^{-1/2}|^3}{\varepsilon}\sum_{i\in I,m(i)\leq\ell} \bigg(\sum_{j\in I:m(j)=m(i), \chi_{ij}=1} \big(\bar W_{ij} \bar Z_{ij}^2 
+2 \bar Y_{ij}\, \bar Z_{ij}^2 \big)
+\bar X_i \bar Z_{i}^3\bigg),
\\
\mathcal{R}_{alllevel}&:=
CN^{9/2} |\Lambda^{-1/2}|^2 
|\log \varepsilon|
\sum_{i\in I} \bigg(\sum_{j\in I:m(j)=m(i),\chi_{ij}=1} \big(\bar W_{ij} \bar Z_{ij}
+2 \bar Y_{ij}\, \bar Z_{ij}\big)
+\bar X_i \bar Z_{i}^2\bigg),
\\
\mathcal{R}_{tail}:=&C |\Lambda^{-1}| N^{3/2} \varepsilon^{-N}
\sum_{i\in I} \Bigg(
\sum_{j\in I:m(j)=m(i),\chi_{ij}=1} \mathbb{E}\Big[|W_{ij}| |Z_{ij}| (\chi_{|Z_{ij}|>\bar Z_{ij}}+\chi_{|W_{ij}|>\bar W_{ij}})\Big]
\\&~~~~~~~~~~~~~~~~~~~~~~~~~~
\nonumber
+\sum_{l\in I:m(l)=m(i), \chi_{il}=1}
\mathbb{E}\Big[|Y_{il}|\,
|Z_{il}|(\chi_{|Z_{il}|>\bar Z_{il}}+\chi_{|Y_{il}|>\bar Y_{il}}) \Big]
\\&~~~~~~~~~~~~~~~~~~~~~~~~~~
\nonumber
+\mathbb{E}\Big[|X_i| |Z_{i}|^2 (\chi_{|Z_i|>\bar Z_i}+\chi_{|X_i|>\bar X_i})\Big]
\Bigg).
\end{align}
\end{subequations}
\end{theorem}
Concerning our theorem, a few remarks are in order:
\begin{itemize}
\item In our theorem $\varepsilon\in (0,\frac{1}{2}]$ is a free parameter that by \eqref{NormalApproximationEstimate} one would typically choose in such a way that $|\Lambda^{1/2}| \varepsilon$ is of the order of the normal approximation error.
\item Note that our theorem is tailored towards an application to random variables $X_k$ with stretched exponential moments: For typical dependence structures like our multilevel local dependence structure from Definition~\ref{ConditionRandomVariable}, by concentration estimates like in Lemma~\ref{ConcentrationStretchedExponential} the random variables $Z_{ij}$, $Z_i$, and $Y_{il}$ also satisfy stretched exponential moment bounds. Choosing $\bar Z_{ij}\sim C||Z_{ij}||_{\exp^{\gamma_0}} (\log L)^{1/\gamma_0}$ and making similar choices for the other variables, the smallness of the terms $\mathcal{R}_{tail}$ is then a consequence Lemma~\ref{CalculusStretchedExponential}b.
\item The dependence of our estimates on the dimension $N$ is not optimal and we make no attempt to get close to optimality.
\end{itemize}

Before proving our theorem, let us state two auxiliary results that are required for the proof.

The majority of the following auxiliary results on the existence of solutions to the (smoothed) Stein's equation are standard, see e.\,g.\ \cite{ChenGoldsteinShao}; however, in order to keep the paper self-contained and in order to keep the dependence on the dimension $N$ explicit, we provide the detailed argument in the appendix.

\begin{proposition}
\label{PropositionSolutionSteinEquation}
Let $N\in \mathbb{N}$ and let $\Lambda\in \mathbb{R}^{N\times N}$ be a symmetric positive definite matrix. Let $\bar L>0$.

For any $\phi\in \Phi_{\Lambda}^{\bar L}$ and any $\varepsilon>0$, there exists a function $f_\varepsilon$ with the following properties:
\begin{itemize}
\item[a)] The functions $\phi$ and $f_\varepsilon$ are related through the ``mollified'' Stein equation
\begin{align}
\label{MollifiedSteinEquation}
&-(\nabla \cdot \Lambda \nabla f_\varepsilon)(x)+(x\cdot \nabla f_\varepsilon)(x)
=\phi_\varepsilon(x)
-\int_{\mathbb{R}^N} \phi_\varepsilon(z) \mathcal{N}_\Lambda(z) \,dz
\end{align}
with
\begin{align}
\label{DefinitionPhiVarepsilon}
\phi_\varepsilon(x)
:=\int_{\mathbb{R}^N} \phi(\sqrt{1-\varepsilon^2}x-\varepsilon z) \mathcal{N}_\Lambda(z) \,dz.
\end{align}
\item[b)] The third derivative of $f_\varepsilon$ is subject to the uniform bound
\begin{align}
\label{BoundThirdDerivativeLinfty}
|\nabla^3 f_\varepsilon(x)|
\leq 15 |\Lambda^{-1}| \varepsilon^{-N}
\end{align}
for all $x\in \mathbb{R}^N$.
\item[c)] For any $\delta>0$ and $K:= 2\sqrt{N}+1$, the function
\begin{align*}
H_\delta^\varepsilon(x):=2(\mathcal{N}_{\delta^2 \Id_N} \ast \osc_{K\delta} \nabla^2 f_\varepsilon)(x)
\end{align*}
is an upper bound for the oscillation of $\nabla^2 f_\varepsilon$
\begin{align}
\label{UpperBoundByH}
(\osc_{\delta} \nabla^2 f_\varepsilon) (x) = \sup_{|z_a|\leq \delta,|z_b|\leq \delta} |\nabla^2 f_\varepsilon(x+z_a)-\nabla^2 f_\varepsilon(x+z_b)| \leq H_\delta^\varepsilon(x)
\end{align}
and satisfies the estimates
\begin{align}
\label{BoundMollifiedSecondDerivativeOscLimitDistribution}
&\int_{\mathbb{R}^N} H_\delta^\varepsilon (x) \mathcal{N}_{\Lambda}(x) \,dx
\leq 10^2 N^{3/2} |\Lambda^{-1}| \, |\log \varepsilon| \, \delta
\end{align}
and
\begin{align}
\label{BoundMollifiedSecondDerivativeOscInFunctionClass}
\frac{1}{3\cdot 10^4 N^{5/2} |\Lambda^{-1}| |\log \varepsilon|} H_\delta^\varepsilon \in \Phi_\Lambda^{\tilde L}
\end{align}
for any $\tilde L\geq 2\cdot 4^N (|\Lambda^{1/2}|^N \delta^{-N} +1)$.
\item[d)] For any $\delta>0$, the function
\begin{align*}
H_{\varepsilon,\delta}'(x):=|\nabla^3 f_\varepsilon(x)|+2(\mathcal{N}_{\delta^2 \Id_N} \ast \osc_{K\delta} \nabla^3 f_\varepsilon)(x)
\end{align*}
with $K:=2\sqrt{N}+1$
is an upper bound for the supremum of $\nabla^3 f_\varepsilon$ in a $\delta$-neighborhood
\begin{align}
\label{UpperBoundByH2}
\sup_{|z|\leq \delta} |\nabla^3 f_\varepsilon| (x+z) \leq H_{\varepsilon,\delta}'(x)
\end{align}
and satisfies the estimates
\begin{align}
\label{BoundMollifiedThirdDerivativeLimitDistribution}
&\int_{\mathbb{R}^N} H_{\varepsilon,\delta}' (x) \mathcal{N}_{\Lambda}(x) \,dx
\leq 10^2 N^3 |\Lambda^{-1/2}|^2 \Big(|\log \varepsilon|+\frac{|\Lambda^{-1/2}|\delta}{\varepsilon}\Big)
\end{align}
and
\begin{align}
\label{BoundMollifiedThirdDerivativeInFunctionClass}
\frac{\varepsilon}{10^4 N^{3} |\Lambda^{-1/2}|^3} H_{\varepsilon,\delta}' \in \Phi_\Lambda^{\tilde L}
\end{align}
for any $x_0\in \mathbb{R}^N$ and any $\tilde L\geq 4^N (|\Lambda^{1/2}|^N \delta^{-N} +1)+2\varepsilon^{-N}$.
\end{itemize}
\end{proposition}

The following lemma enables us to replace the class of functions $\Phi_\Lambda^{\bar L}$ in the definition of the distance $\mathcal{D}^{\bar L}$ by a class of mollified functions $\phi_\varepsilon$ (with $\phi\in \Phi_\Lambda^{\bar L}$). This is a standard argument in the theory of normal approximation by Stein's method \cite{BhattacharyaRao,Goetze}; however, we provide the proof of our version of the lemma in the appendix, as it keeps an explicit (though not optimal) dependence on the dimension $N$.
\begin{lemma}
\label{SmoothingEstimateLemma}
Given $\phi\in \Phi_\Lambda$, define
\begin{align*}
\phi_\varepsilon(x):=\int_{\mathbb{R}^N} \phi(\sqrt{1-\varepsilon^2}x-\varepsilon z) \mathcal{N}_\Lambda(z) \,dz.
\end{align*}
Introduce the ``smoothed'' distance
\begin{align*}
\mathcal{D}_\varepsilon^{\bar L}(X,\mathcal{N}_\Lambda):=
\sup_{\phi\in \Phi_\Lambda^{\bar L}} \bigg(\mathbb{E}[\phi_\varepsilon(X)]-\int_{\mathbb{R}^N} \phi_\varepsilon(x) \mathcal{N}_\Lambda(x)\,dx \bigg).
\end{align*}
For any $0<\varepsilon\leq \frac{1}{2}$ and any ${\bar L}\geq 2 \cdot 4^N \varepsilon^{-N}$ we then have the estimate
\begin{align}
\label{SmoothingEstimate}
\mathcal{D}^{\bar L}(X,\mathcal{N}_\Lambda)
\leq
20 \sqrt{N} |\Lambda^{1/2}| \varepsilon
+10^3 N^{3/2} \mathcal{D}_\varepsilon^{\bar L}(X,\mathcal{N}_\Lambda)
\end{align}
for all random variables $X$.
\end{lemma}

We now turn to the proof of our result on quantitative normal approximation for a sum $X$ of random variables $X_i$ with ``multilevel local dependence structure''.
\begin{proof}[Proof of Theorem~\ref{NormalApproximation}]
First we observe that by the fact
\begin{align*}
\mathcal{D}(X,\mathcal{N}_\Lambda)
=\lim_{{\bar L}\rightarrow\infty} \mathcal{D}^{\bar L}(X,\mathcal{N}_\Lambda)
\end{align*}
it suffices to establish the bound for $\mathcal{D}^{\bar L}(X,\mathcal{N}_\Lambda)$ for all large enough but finite ${\bar L}<\infty$.

The proof proceeds using Stein's method of normal approximation. Proposition~\ref{PropositionSolutionSteinEquation} provides for any $0<\varepsilon<1$ and for any $\phi\in \Phi_{\Lambda}^{\bar L}$ a function $f_\varepsilon$ that solves the ``smoothed'' Stein's equation
\begin{align}
\label{MollifiedSteinEquationRepeat}
-\nabla \cdot (\Lambda \nabla f_\varepsilon(x)) + x\cdot \nabla f_\varepsilon(x)
= \phi_\varepsilon(x)-\int_{\mathbb{R}^N} \phi_\varepsilon(z) \mathcal{N}_\Lambda(z)\,dz.
\end{align}
The smoothing result of Lemma~\ref{SmoothingEstimateLemma} allows us to control the distance $\mathcal{D}^{\bar L}(X,\mathcal{N}_\Lambda)$ by
\begin{align*}
\mathcal{D}^{\bar L}(X,\mathcal{N}_\Lambda)
\leq
20 \sqrt{N} |\Lambda^{1/2}| \varepsilon
+10^3 N^{3/2} \sup_{\phi \in \Phi_\Lambda^{\bar L}}\bigg|\mathbb{E}[\phi_\varepsilon(X)] - \int_{\mathbb{R}^N} \phi_\varepsilon(z) \mathcal{N}_\Lambda(z) \,dz \bigg|.
\end{align*}
Thus, by the ``mollified'' Stein's equation \eqref{MollifiedSteinEquationRepeat} the estimate
\begin{align}
\label{EstimateOnDL}
\mathcal{D}^{\bar L}(X,\mathcal{N}_\Lambda)
\leq
20 \sqrt{N} |\Lambda^{1/2}| \varepsilon
+10^3 N^{3/2} \sup_{\phi \in \Phi_\Lambda^{\bar L}}\bigg|
\mathbb{E}\Big[(-\nabla \cdot (\Lambda \nabla f_\varepsilon)+x \cdot \nabla f_\varepsilon)(X)\Big]\bigg|
\end{align}
holds true.
Hence, in order to obtain an estimate for $\mathcal{D}^{\bar L}(X,\mathcal{N}_\Lambda)$ it suffices to bound
\begin{align*}
\mathbb{E}\big[(-\nabla \cdot (\Lambda \nabla f_\varepsilon)+x \cdot \nabla f_\varepsilon)(X)\big]
\end{align*}
uniformly in $\phi\in \Phi_\Lambda^{\bar L}$.

In order to derive such an estimate, we may rewrite using the definition of $X$ (see \eqref{DefinitionX})
\begin{align*}
\mathbb{E}\big[(x \cdot \nabla f_\varepsilon)(X)\big]
&=\mathbb{E}\big[X\cdot \nabla f_\varepsilon(X)\big]
=
\sum_{i\in I} \mathbb{E}\big[X_i \cdot \nabla f_\varepsilon(X)\big].
\end{align*}
%Let $Y_i$ and $Z_i$ denote independent copies of the family $X_i$.
As by definition of $\chi_{ij}$ the quantities $X_i$ and $X-\sum_{j\in I} \chi_{ij} X_j$ are stochastically independent, we obtain by adding and subtracting $\mathbb{E}[X_i \cdot \nabla f_\varepsilon (X-\sum_{j\in I} \chi_{ij} X_j)]$ and using the fact that $\mathbb{E}[X_i]=0$
\begin{align*}
&\mathbb{E}\big[(x \cdot \nabla f_\varepsilon)(X)\big]
\\&
=
\sum_{i\in I} \mathbb{E}[X_i] \cdot \mathbb{E}\bigg[\nabla f_\varepsilon \Big(X-\sum_{j\in I} \chi_{ij} X_j\Big)\bigg]
\\&~~~
+\sum_{i\in I} \mathbb{E}\bigg[X_i \cdot \Big(\nabla f_\varepsilon(X)-\nabla f_\varepsilon \Big(X-\sum_{j\in I} \chi_{ij} X_j\Big)\Big)\bigg]
\\&
=
\sum_{i\in I} \mathbb{E}\bigg[X_i \cdot \Big(\nabla f_\varepsilon(X)-\nabla f_\varepsilon \Big(X-\sum_{j\in I} \chi_{ij} X_j\Big)\Big)\bigg].
\end{align*}
Next, we infer by adding zero in order to obtain an expression reminiscent of Taylor expansion
\begin{align*}
&\mathbb{E}\big[(x \cdot \nabla f_\varepsilon)(X)\big]
\\
&=
\sum_{i\in I} \mathbb{E}\bigg[X_i \otimes \Big(\sum_{j\in I} \chi_{ij} X_j\Big) : \nabla^2 f_\varepsilon (X)\bigg]
\\&~~~
+\sum_{i\in I}
\mathbb{E}\bigg[X_i \cdot \bigg(\nabla f_\varepsilon(X)-\nabla f_\varepsilon\Big(X-\sum_{j\in I} \chi_{ij} X_j\Big)
-\Big(\sum_{j\in I} \chi_{ij} X_j\Big) \cdot \nabla^2 f_\varepsilon (X)\bigg)\bigg].
\end{align*}
Adding again zero, we deduce
\begin{align*}
&\mathbb{E}\big[(x \cdot \nabla f_\varepsilon)(X)\big]
\\
&=
\sum_{i\in I} \sum_{j\in I} \chi_{ij} \mathbb{E}[X_i\otimes X_j] : \mathbb{E}[(\nabla^2 f_\varepsilon) (X)]
\\&~~~
+\sum_{i\in I} \sum_{j\in I} \chi_{ij} \mathbb{E}\bigg[\Big(X_i \otimes X_j-\mathbb{E}[X_i \otimes X_j]\Big) : \nabla^2 f_\varepsilon (X)\bigg]
\\&~~~
+\sum_{i\in I}
\mathbb{E}\bigg[X_i \cdot \bigg(\nabla f_\varepsilon(X)-\nabla f_\varepsilon\Big(X-\sum_{j\in I} \chi_{ij} X_j\Big)
-\Big(\sum_{j\in I} \chi_{ij} X_j\Big) \cdot \nabla^2 f_\varepsilon (X)\bigg)\bigg].
\end{align*}
We now observe that the double sum $\sum_{i\in I} \sum_{j\in I} \chi_{ij} \mathbb{E}[X_i\otimes X_j]$ in the first term on the right-hand side is equal to $\Lambda$. Splitting the second term on the right-hand side and using the symmetry with respect to $i$ and $j$, we obtain
\begin{align*}
&\mathbb{E}\big[(x \cdot \nabla f_\varepsilon)(X)\big]-\mathbb{E}[(\nabla \cdot \Lambda \nabla f_\varepsilon)(X)]
\\
&=\sum_{i\in I} \sum_{j\in I:m(j)=m(i)} \chi_{ij} \mathbb{E}\bigg[\Big(X_i \otimes X_j-\mathbb{E}[X_i \otimes X_j]\Big) : \nabla^2 f_\varepsilon (X)\bigg]
\\&~~~
+2\sum_{i\in I} \sum_{j\in I:m(j)<m(i)} \chi_{ij} \mathbb{E}\bigg[\Big(X_i \otimes X_j-\mathbb{E}[X_i \otimes X_j]\Big) : \nabla^2 f_\varepsilon (X)\bigg]
\\&~~~
+\sum_{i\in I}
\mathbb{E}\bigg[X_i \cdot \bigg(\nabla f_\varepsilon(X)
-\nabla f_\varepsilon\Big(X-\sum_{j\in I} \chi_{ij} X_j\Big)
-\Big(\sum_{j\in I} \chi_{ij} X_j\Big) \cdot \nabla^2 f_\varepsilon (X)\bigg)\bigg].
\end{align*}
Using the fact that for $m(i)\geq m(j)$ by definition of $\chi_{ijk}$ and $i^m(j)$ the quantities $X_i \otimes X_j$ and $X-\sum_{k\in I}\chi_{ii^{m(i)}(j)k} X_k$ are stochastically independent (recall that $\chi_{ijk}=1$ implies $\chi_{ii^{m(i)}(j)k}=1$ for $m(j)< m(i)$) and making use of the fact that $\chi_{ij}=\chi_{ij}\chi_{ii^{m(i)}(j)}$, we infer
\begin{align*}
&\mathbb{E}\big[(x \cdot \nabla f_\varepsilon)(X)\big]-\mathbb{E}[(\nabla \cdot \Lambda \nabla f_\varepsilon)(X)]
\\
&=\sum_{i\in I} \sum_{j\in I:m(j)=m(i)} \chi_{ij} \mathbb{E}\bigg[\Big(X_i \otimes X_j-\mathbb{E}[X_i \otimes X_j]\Big) 
\\&~~~~~~~~~~~~~~~~~~~~~~~~~~~~~~~~~~~~~~~~~~
: \bigg(\nabla^2 f_\varepsilon (X)
-\nabla^2 f_\varepsilon \Big(X-\sum_{k\in I}\chi_{ijk} X_k\Big)\bigg)\bigg]
\\&~~~
+2\sum_{i\in I} \sum_{j\in I:m(j)<m(i)} \chi_{ij} \chi_{ii^{m(i)}(j)} \mathbb{E}\bigg[\Big(X_i \otimes X_j-\mathbb{E}[X_i \otimes X_j]\Big)
\\&~~~~~~~~~~~~~~~~~~~~~~~~~~~~~~~~~~~~~~~~~~
: \bigg(\nabla^2 f_\varepsilon (X)-\nabla^2 f_\varepsilon \Big(X-\sum_{k\in I}\chi_{ii^{m(i)}(j)k} X_k\Big)\bigg)\bigg]
\\&~~~
+\sum_{i\in I}
\mathbb{E}\bigg[X_i \cdot \bigg(\nabla f_\varepsilon(X)-\nabla f_\varepsilon\Big(X-\sum_{j\in I} \chi_{ij} X_j\Big)
-\Big(\sum_{j\in I} \chi_{ij} X_j\Big) \cdot \nabla^2 f_\varepsilon (X)\bigg)\bigg].
\end{align*}
We now split the sum in the second term on the right-hand side by introducing the additional variable $l:=i^{m(i)}(j)$. The reason for introducing this additional splitting is that the sum $\sum_{j\in I:m(j)<m(i),i^{m(i)}(j)=l} \chi_{ij} \big(X_i \otimes X_j-\mathbb{E}[X_i \otimes X_j]\big)$ is subject to a better estimate than obtained by a standard triangle inequality, as one may exploit the stochastic independence of many terms in the sum. We deduce
\begin{align*}
&\mathbb{E}\big[(x \cdot \nabla f_\varepsilon)(X)\big]-\mathbb{E}[(\nabla \cdot \Lambda \nabla f_\varepsilon)(X)]
\\
&=\sum_{i\in I} \sum_{j\in I:m(j)=m(i)} \chi_{ij} \mathbb{E}\bigg[\Big(X_i \otimes X_j-\mathbb{E}[X_i \otimes X_j]\Big) 
\\&~~~~~~~~~~~~~~~~~~~~~~~~~~~~~~~~~~~~~~~~~~
: \bigg(\nabla^2 f_\varepsilon (X)
-\nabla^2 f_\varepsilon \Big(X-\sum_{k\in I}\chi_{ijk} X_k\Big)\bigg)\bigg]
\\&~~~
+2\sum_{i\in I} \sum_{l\in I:m(l)=m(i)} \chi_{il}
\mathbb{E}\Bigg[\sum_{j\in I:m(j)<m(i),i^{m(i)}(j)=l} \chi_{ij} \Big(X_i \otimes X_j-\mathbb{E}[X_i \otimes X_j]\Big)
\\&~~~~~~~~~~~~~~~~~~~~~~~~~~~~~~~~~~~~~~~~~~~~~~~~~~~~~
: \bigg(\nabla^2 f_\varepsilon (X)-\nabla^2 f_\varepsilon \Big(X-\sum_{k\in I}\chi_{ilk} X_k\Big)\bigg)\Bigg]
\\&~~~
+\sum_{i\in I}
\mathbb{E}\bigg[X_i \cdot \bigg(\nabla f_\varepsilon(X)-\nabla f_\varepsilon\Big(X-\sum_{j\in I} \chi_{ij} X_j\Big)
-\Big(\sum_{j\in I} \chi_{ij} X_j\Big) \cdot \nabla^2 f_\varepsilon (X)\bigg)\bigg].
\end{align*}
We intend to use the Taylor formula respectively the definition of the oscillation $\osc$ to bound each of the three terms on the right-hand side. For example, the terms in the first sum may either be estimated as
\begin{align*}
&\bigg|\Big(X_i \otimes X_j-\mathbb{E}[X_i \otimes X_j]\Big)
: \bigg(\nabla^2 f_\varepsilon (X)
-\nabla^2 f_\varepsilon \Big(X-\sum_{k\in I}\chi_{ijk} X_k\Big)\bigg)\bigg|
\\&
\leq \big|X_i \otimes X_j-\mathbb{E}[X_i \otimes X_j]\big| \cdot \Big|\sum_{k\in I}\chi_{ijk} X_k\Big| \cdot \sup_{|z|\leq |\sum_{k\in I}\chi_{ijk} X_k|} |\nabla^3 f_\varepsilon(X+z)|
\end{align*}
or as
\begin{align*}
&\bigg|\Big(X_i \otimes X_j-\mathbb{E}[X_i \otimes X_j]\Big)
: \bigg(\nabla^2 f_\varepsilon (X)
-\nabla^2 f_\varepsilon \Big(X-\sum_{k\in I}\chi_{ijk} X_k\Big)\bigg)\bigg|
\\&
\leq \big|X_i \otimes X_j-\mathbb{E}[X_i \otimes X_j]\big| \cdot \big(\osc_{|\sum_{k\in I}\chi_{ijk} X_k|} \nabla^2 f_\varepsilon\big)(X).
\end{align*}
Using the abbreviations $Z_{ij}$, $Z_i$, $Y_{il}$, and $W_{ij}$ from \eqref{NormalApproximationAbbreviations1}
and distinguishing the cases $|Z_{ij}|>\bar Z_{ij}$ and  $|Z_{ij}|\leq \bar Z_{ij}$ (and $|Z_i|>\bar Z_i$ and $|Z_i|\leq \bar Z_i$, and so forth, for fixed but arbitrary constants $\bar X_i$, $\bar Z_{ij}$, $\bar Z_i$, $\bar Y_{il}$, $\bar W_{ij}$) as well as (in the latter cases) the cases $m(i)\le \ell$ and $m(i)>\ell$, we obtain by treating the other sums analogously
\begin{align}
\label{Bound123}
&\bigg|-\mathbb{E}[(\nabla \cdot \Lambda \nabla f_\varepsilon)(X)]+\mathbb{E}\big[(x \cdot \nabla f_\varepsilon)(X)\big]\bigg|
\\
&\leq
\nonumber
\sum_{i\in I,m(i)\leq\ell} \sum_{j\in I:m(j)=m(i)} \chi_{ij} \mathbb{E}\Big[\bar W_{ij} \bar Z_{ij} \sup_{|z|\leq \bar Z_{ij}} |\nabla^3 f_\varepsilon (X+z)|\Big]
\\&~~~
\nonumber
+\sum_{i\in I,m(i)>\ell} \sum_{j\in I:m(j)=m(i)} \chi_{ij} \mathbb{E}\Big[\bar W_{ij}\, (\osc_{\bar Z_{ij}} \nabla^2 f_\varepsilon) (X)\Big]
\\&~~~
\nonumber
+\sum_{i\in I} \sum_{j\in I:m(j)=m(i)} \chi_{ij} \mathbb{E}\Big[|W_{ij}| |Z_{ij}| (\chi_{|Z_{ij}|>\bar Z_{ij}}+\chi_{|W_{ij}|> \bar W_{ij}}) \sup_{z\in \mathbb{R}^N} |\nabla^3 f_\varepsilon(z)|\Big]
\\&~~~
\nonumber
+2\sum_{i\in I,m(i)\leq\ell} \, \sum_{l\in I:m(l)=m(i)} \chi_{il}
\mathbb{E}\Big[\bar Y_{il}\,
\bar Z_{il} \sup_{|z|\leq \bar Z_{il}} |\nabla^3 f_\varepsilon| (X+z)\Big]
\\&~~~
\nonumber
+2\sum_{i\in I,m(i)>\ell} \, \sum_{l\in I:m(l)=m(i)} \chi_{il}
\mathbb{E}\Big[\bar Y_{il}\,
(\osc_{\bar Z_{il}} \nabla^2 f_\varepsilon) (X)\Big]
\\&~~~
\nonumber
+2\sum_{i\in I} \sum_{l\in I:m(l)=m(i)} \chi_{il}
\mathbb{E}\Big[|Y_{il}|\,
|Z_{il}|(\chi_{|Z_{il}|>\bar Z_{il}}+\chi_{|Y_{il}|>\bar Y_{il}}) \sup_{z\in \mathbb{R}^N} |\nabla^3 f_\varepsilon(z)|\Big]
\\&~~~
\nonumber
+\sum_{i\in I,m(i)\leq\ell}
\mathbb{E}\Big[ \bar X_i
\bar Z_{i}^2 \sup_{|z|\leq \bar Z_{i}} |\nabla^3 f_\varepsilon| (X+z)\Big]
\\&~~~
\nonumber
+\sum_{i\in I,m(i)>\ell}
\mathbb{E}\Big[\bar X_i
\bar Z_{i} (\osc_{\bar Z_{i}} \nabla^2 f_\varepsilon) (X)\Big]
\\&~~~
\nonumber
+\sum_{i\in I}
\mathbb{E}\Big[|X_i|
|Z_{i}|^2 (\chi_{|Z_i|>\bar Z_i}+\chi_{|X_i|>\bar X_i}) \sup_{z\in \mathbb{R}^N} |\nabla^3 f_\varepsilon(z)|\Big].
\end{align}
Reordering terms and using the definitions of $H_\delta^\varepsilon$ and $H_{\varepsilon,\delta}'$ in Proposition~\ref{PropositionSolutionSteinEquation} as well as the bound \eqref{BoundThirdDerivativeLinfty}, we deduce
\begin{align*}
&\bigg|-\mathbb{E}[(\nabla \cdot \Lambda \nabla f_\varepsilon)(X)]+\mathbb{E}\big[(x \cdot \nabla f_\varepsilon)(X)\big]\bigg|
\\
&\leq
\sum_{i\in I,m(i)\leq\ell} \sum_{j\in I:m(j)=m(i)} \chi_{ij} \bar W_{ij} \bar Z_{ij} \mathbb{E}\big[H_{\varepsilon,\bar Z_{ij}}'(X)\big]
\\&~~~
+2\sum_{i\in I,m(i)\leq\ell} \, \sum_{l\in I:m(l)=m(i)} \chi_{il}
\bar Y_{il}\, \bar Z_{il} \mathbb{E}\big[H_{\varepsilon,\bar Z_{il}}' (X)\big]
\\&~~~
+\sum_{i\in I,m(i)\leq\ell}
\bar X_i
\bar Z_{i}^2 \mathbb{E} \big[H_{\varepsilon, \bar Z_{i}}'(X)\big]
\\&~~~
+\sum_{i\in I,m(i)>\ell} \sum_{j\in I:m(j)=m(i)} \chi_{ij} \bar W_{ij} \mathbb{E}\big[H^\varepsilon_{\bar Z_{ij}}(X)\big]
\\&~~~
+2\sum_{i\in I,m(i)>\ell} \sum_{l\in I:m(l)=m(i)} \chi_{il} \bar Y_{il}
\mathbb{E}\big[
H^\varepsilon_{\bar Z_{il}} (X)\big]
\\&~~~
+\sum_{i\in I,m(i)>\ell}\bar X_i \bar Z_{i} 
\mathbb{E}\big[H^\varepsilon_{\bar Z_{i}} (X)\big]
\\&~~~
+15 \Lambda^{-1} \varepsilon^{-N}\sum_{i\in I} \Bigg(
\sum_{j\in I:m(j)=m(i)} \chi_{ij} \mathbb{E}\Big[|W_{ij}| |Z_{ij}| (\chi_{|Z_{ij}|>\bar Z_{ij}}+\chi_{|W_{ij}|> \bar W_{ij}})
\\&~~~~~~~~~~~~~~~~~~~~~~~~~~~~~~~
+2\sum_{l\in I:m(l)=m(i)} \chi_{il}
\mathbb{E}\Big[|Y_{il}|\,
|Z_{il}|(\chi_{|Z_{il}|>\bar Z_{il}}+\chi_{|Y_{il}|>\bar Y_{il}}) \Big]
\\&~~~~~~~~~~~~~~~~~~~~~~~~~~~~~~~
+\mathbb{E}\Big[|X_i| |Z_{i}|^2 (\chi_{|Z_i|>\bar Z_i}+\chi_{|X_i|>\bar X_i})\Big]
\Bigg).
\end{align*}
For $L\geq 4^N (|\Lambda^{1/2}|^N \delta^{-N} +1)$, by Proposition~\ref{PropositionSolutionSteinEquation} we may control
\begin{align*}
&\mathbb{E}[H_\delta^\varepsilon(X)]
\\&
=
\mathbb{E}[H_\delta^\varepsilon(X)]
-\int_{\mathbb{R}^N} H_\delta^\varepsilon(z)\mathcal{N}_\Lambda(z) \,dz
+\int_{\mathbb{R}^N} H_\delta^\varepsilon(z)\mathcal{N}_\Lambda(z) \,dz
\\&
\stackrel{\eqref{BoundMollifiedSecondDerivativeOscInFunctionClass}}{\leq} 2\cdot 10^4 N^{5/2} |\Lambda^{-1}| |\log \varepsilon| \mathcal{D}^L(X,\mathcal{N}_\Lambda)
+\int_{\mathbb{R}^N} H_\delta^\varepsilon(z)\mathcal{N}_\Lambda(z) \,dz
\\&
\stackrel{\eqref{BoundMollifiedSecondDerivativeOscLimitDistribution}}{\leq} 2\cdot 10^4 N^{5/2} |\Lambda^{-1}| |\log \varepsilon| \mathcal{D}^L(X,\mathcal{N}_\Lambda)
+10^2 N^{3/2} |\Lambda^{-1}| \, |\log \varepsilon| \, \delta
\\&
\leq C N^{5/2} |\Lambda^{-1}| |\log \varepsilon| \mathcal{D}^L(X,\mathcal{N}_\Lambda)
+C N^{3/2} |\Lambda^{-1}| \, |\log \varepsilon| \, \delta.
\end{align*}
Similarly, for $L\geq  4^N (|\Lambda^{1/2}|^N \delta^{-N} +1) + \varepsilon^{-N}$ we have by Proposition~\ref{PropositionSolutionSteinEquation}
\begin{align*}
&\mathbb{E}[H_{\varepsilon,\delta}'(X)]
\\&
=
\mathbb{E}[H_{\varepsilon,\delta}'(X)]
-\int_{\mathbb{R}^N} H_{\varepsilon,\delta}'(z)\mathcal{N}_\Lambda(z) \,dz
+\int_{\mathbb{R}^N} H_{\varepsilon,\delta}'(z)\mathcal{N}_\Lambda(z) \,dz
\\&
\leq
\frac{10^4 N^3 |\Lambda^{-1/2}|^3}{\varepsilon} \mathcal{D}^L(X,\mathcal{N}_\Lambda)
+10^2 N^3 |\Lambda^{-1/2}|^2 \Big(|\log \varepsilon|+\frac{|\Lambda^{-1/2}|\delta}{\varepsilon}\Big)
\\&
\leq
\frac{C N^3 |\Lambda^{-1/2}|^3}{\varepsilon} \mathcal{D}^L(X,\mathcal{N}_\Lambda)
+C N^3 |\Lambda^{-1/2}|^2 \Big(|\log \varepsilon|+\frac{|\Lambda^{-1/2}|\delta}{\varepsilon}\Big).
\end{align*}
Using these two estimates in the preceding bound and collecting terms, we obtain
\begin{align*}
&\bigg|-\mathbb{E}[(\nabla \cdot \Lambda \nabla f_\varepsilon)(X)]+\mathbb{E}\big[(x \cdot \nabla f_\varepsilon)(X)\big]\bigg|
\\
&\leq
\mathcal{D}^L(X,\mathcal{N}_\Lambda)
\frac{CN^3 |\Lambda^{-1/2}|^3}{\varepsilon}
\sum_{i\in I,m(i)\leq\ell} \bigg(\sum_{j\in I:m(j)=m(i)} \chi_{ij} \bar W_{ij} \bar Z_{ij}
\\&~~~~~~~~~~~~~~~~~~~~~~~~~~~~~~~~~~~~~~~~~~~~~~~~~~~~~~~
+2 \sum_{l\in I:m(l)=m(i)} \chi_{il} \bar Y_{il}\, \bar Z_{il}
+\bar X_i \bar Z_{i}^2\bigg)
\\&~~~
+\frac{CN^3 |\Lambda^{-1/2}|^3}{\varepsilon}
\sum_{i\in I,m(i)\leq\ell} \bigg(\sum_{j\in I:m(j)=m(i)} \chi_{ij} \bar W_{ij} \bar Z_{ij}^2 
\\&~~~~~~~~~~~~~~~~~~~~~~~~~~~~~~~~~~~~~~~~~~~~~~~~~~~~~~~
+2 \sum_{l\in I:m(l)=m(i)} \chi_{il} \bar Y_{il}\, \bar Z_{il}^2
+\bar X_i \bar Z_{i}^3\bigg)
\\&~~~
+CN^3 |\Lambda^{-1/2}|^2 |\log \varepsilon|
\sum_{i\in I,m(i)\leq\ell} \bigg(\sum_{j\in I:m(j)=m(i)} \chi_{ij} \bar W_{ij} \bar Z_{ij}
\\&~~~~~~~~~~~~~~~~~~~~~~~~~~~~~~~~~~~~~~~~~~~~~~~~~~~~~~~
+2 \sum_{l\in I:m(l)=m(i)} \chi_{il} \bar Y_{il}\, \bar Z_{il}
+\bar X_i \bar Z_{i}^2\bigg)
\\&~~~
+\mathcal{D}^L(X,\mathcal{N}_\Lambda)
\cdot CN^{5/2} |\Lambda^{-1}| |\log \varepsilon|
\sum_{i\in I,m(i)>\ell}
\bigg(\sum_{j\in I:m(j)=m(i)} \chi_{ij} \bar W_{ij} 
\\&~~~~~~~~~~~~~~~~~~~~~~~~~~~~~~~~~~~~~~~~~~~~~~~~~~~~~~~~~~~~~~
+2 \sum_{l\in I:m(l)=m(i)} \chi_{il} \bar Y_{il}
+\bar X_i \bar Z_{i} 
\bigg)
\\&~~~
+C N^{3/2} |\Lambda^{-1}| |\log \varepsilon |
\sum_{i\in I,m(i)>\ell}
\bigg(\sum_{j\in I:m(j)=m(i)} \chi_{ij} \bar W_{ij} \bar Z_{ij}
\\&~~~~~~~~~~~~~~~~~~~~~~~~~~~~~~~~~~~~~~~~~~~~~~~~~~~~~~~
+2\sum_{l\in I:m(l)=m(i)} \chi_{il} \bar Y_{il} \bar Z_{il}
+\bar X_i \bar Z_{i}^2 
\bigg)
\\&~~~
+15 |\Lambda^{-1}| \varepsilon^{-N}\sum_{i\in I} \Bigg(
\sum_{j\in I:m(j)=m(i)} \chi_{ij} \mathbb{E}\Big[|W_{ij}| |Z_{ij}| (\chi_{|Z_{ij}|>\bar Z_{ij}}+\chi_{|W_{ij}|> \bar W_{ij}})
\\&~~~~~~~~~~~~~~~~~~~~~~~~~~~~~~~
+2\sum_{l\in I:m(l)=m(i)} \chi_{il}
\mathbb{E}\Big[|Y_{il}|\,
|Z_{il}|(\chi_{|Z_{il}|>\bar Z_{il}}+\chi_{|Y_{il}|>\bar Y_{il}}) \Big]
\\&~~~~~~~~~~~~~~~~~~~~~~~~~~~~~~~
+\mathbb{E}\Big[|X_i| |Z_{i}|^2 (\chi_{|Z_i|>\bar Z_i}+\chi_{|X_i|>\bar X_i})\Big]
\Bigg).
\end{align*}
Plugging in this bound into \eqref{EstimateOnDL} and using the abbreviations \eqref{NormalApproximationAbbreviations2} for a suitable choice of the constant $C$, we obtain
\begin{align*}
&\mathcal{D}^{\bar L}(X,\mathcal{N}_\Lambda)
\\
&\leq
20\sqrt{N} |\Lambda^{1/2}| \varepsilon
+\frac{1}{2} \mathcal{R}_{low} + \frac{1}{2} \mathcal{R}_{alllevel} + \frac{1}{2} \mathcal{R}_{tail}
\\&~~~
+\frac{1}{2} \mathcal{D}^{\bar L}(X,\mathcal{N}_\Lambda)
\frac{CN^{9/2} |\Lambda^{-1/2}|^3}{\varepsilon}
\sum_{i\in I,m(i)\leq\ell} \bigg(\sum_{j\in I:m(j)=m(i)} \chi_{ij} \bar W_{ij} \bar Z_{ij}
\\&~~~~~~~~~~~~~~~~~~~~~~~~~~~~~~~~~~~~~~~~~~~~~~~~~~~~~~~
+2 \sum_{l\in I:m(l)=m(i)} \chi_{il} \bar Y_{il}\, \bar Z_{il}
+\bar X_i \bar Z_{i}^2\bigg)
\\&~~~
+\frac{1}{2} \mathcal{D}^{\bar L}(X,\mathcal{N}_\Lambda)
\cdot CN^{4} |\Lambda^{-1}| |\log \varepsilon|
\sum_{i\in I,m(i)>\ell}
\bigg(\sum_{j\in I:m(j)=m(i)} \chi_{ij} \bar W_{ij} 
\\&~~~~~~~~~~~~~~~~~~~~~~~~~~~~~~~~~~~~~~~~~~~~~~~~~~~~~~~~~~~~~~
+2 \sum_{l\in I:m(l)=m(i)} \chi_{il} \bar Y_{il}
+\bar X_i \bar Z_{i} 
\bigg).
\end{align*}
Defining the $C$ in the condition \eqref{ConditionForNormalApproximation} to be precisely the $C$ in the previous estimate, we see that the condition \eqref{ConditionForNormalApproximation} allows to absorb the last two terms on the right-hand side in the preceding estimate, thereby proving \eqref{NormalApproximationEstimate}.
%If the condition \eqref{ConditionForNormalApproximation} is satisfied, for $\frac{1}{C}$ chosen small enough an absorption argument yields the desired estimate \eqref{NormalApproximationEstimate}.
\end{proof}

\section{Normal approximation for multilevel local dependence structures}

We now proceed to the proof of our normal approximation result Theorem~\ref{TheoremNormalApproximationMultilevelLocalDependence}. The proof is essentially a reduction to the more abstract normal approximation result provided by Theorem~\ref{NormalApproximation}.
\begin{proof}[Proof of Theorem~\ref{TheoremNormalApproximationMultilevelLocalDependence}]
{\bf Step 1: Proof of the estimate \eqref{NormalApproximationMultilevel}.}
We first derive the normal approximation result in the case of nondegenerate $\Var X$, i.\,e.\ estimate \eqref{NormalApproximationMultilevel}. We will restrict ourselves to the case of $L$-periodic random fields $a$ (i.\,e.\ the case of assumption (A')), as the proof in the non-periodic case is entirely analogous.

Let $X_y^m$ be the random variables from Definition~\ref{ConditionRandomVariable}. To derive our result on normal approximation, we apply Theorem~\ref{NormalApproximation} in the following way: We define the index set $I$ to consist of all pairs $i=(m,y)$ with $0\leq m\leq 1+\log_2 L$ and $y\in 2^m \mathbb{Z}^d \cap [0,L)^d$, and set $X_i:=X_y^m$ as well as $m(i):=m$ for $i=(m,y)$. Furthermore, let us introduce the notation $y_i:=y$ for $i=(m,y)$. Finally, for $j\in I$ and $n>m(j)$ we set $i^n(j):=(n,\tilde y)$, where $\tilde y$ is given by $2^n \lfloor \frac{y_j}{2^n} \rfloor$ (where $\lfloor\cdot\rfloor$ denotes the component-wise floor).

We equip this collection of random variables $X_i$ with the following dependence structure: We set
\begin{align*}
\chi_{ij}=
\begin{cases}
1&\text{if }\dist_\infty^\per(y_i+K \log L[-2^{m(i)},2^{m(i)}]^d,y_j+K \log L[-2^{m(j)},2^{m(j)}]^d)
\\&~~~~~~~~~~~~~~~~~~~~~~~~~~~~~~
\leq 2\cdot 2^{\max\{m(i),m(j)\}} K\log L ,
\\
0&\text{otherwise},
\end{cases}
\end{align*}
and
\begin{align*}
\chi_{ijk}=
\begin{cases}
1&\text{if }\dist_\infty^\per(y_i+K \log L[-2^{m(i)},2^{m(i)}]^d,y_k+K \log L[-2^{m(k)},2^{m(k)}]^d)
\\&~~~~~~~~~~~~~~~~~~~~~~~~~~~~~~
\leq 2\cdot 2^{\max\{m(i),m(j),m(k)\}} K\log L ,
\\
1&\text{if }\dist_\infty^\per(y_j+K \log L[-2^{m(j)},2^{m(j)}]^d,y_k+K \log L[-2^{m(k)},2^{m(k)}]^d)
\\&~~~~~~~~~~~~~~~~~~~~~~~~~~~~~~
\leq 2\cdot 2^{\max\{m(i),m(j),m(k)\}} K\log L ,
\\
0&\text{otherwise},
\end{cases}
\end{align*}
where we recall that $\dist_\infty^\per(U,V)$ denotes the periodicity-adjusted maximum-norm distance, i.\,e.\ $\dist_\infty^\per(U,V):=\inf_{x\in U,y\in V} \inf_{k\in \mathbb{Z}^d} |x-y-L k|$.

We readily verify by Definition~\ref{ConditionRandomVariable} that $\chi_{ij}=0$ indeed entails independence of $X_i$ and the collection of all $X_j$ with $\chi_{ij}=0$ (recall that by Definition~\ref{ConditionRandomVariable} each $X_i$ only depends on $a|_{y_i+K \log L[-2^{m(i)},2^{m(i)}]^d}$ and that $a$ has finite range of dependence $1$). Similarly, the pair $(X_i,X_j)$ is independent from the family of all $X_k$ with $\chi_{ijk}=0$.
It is furthermore easy to verify that $\chi_{ij}$ and $\chi_{ijk}$ satisfy the further conditions on a dependence structure as defined in Section~\ref{SectionNormalApproximationMain}, namely the symmetry $\chi_{ij}=\chi_{ji}$ and $\chi_{ijk}=\chi_{jik}$ as well as the conditions $\chi_{i i^n(j)}\geq \chi_{ij}$ and $\chi_{i i^n(j)k}\geq \chi_{ijk}$ for $n>m(j)$ (the latter two properties follow from $y_j+K\log L [-2^{m(j)},2^{m(j)}] \subset y_{i^n(j)}+K\log L [-2^n,2^n]$, which one deduces directly from $\dist_\infty(y_j,y_{i^n(j)})\leq 2^n$, which in turn follows from our definition of $y_{i^n(j)}$).

With these choices, we have
\begin{align*}
X=\sum_{i\in I}X_i.
\end{align*}
Using the notation of Theorem~\ref{NormalApproximation} it is now the next goal to bound $Z_{ij}$, $Z_i$, $Y_{il}$, and $W_{ij}$. By \eqref{BoundMultilevelDependenceStructure}, we trivially have for
\begin{align*}
W_{ij}=X_i \otimes X_j-\mathbb{E}[X_i \otimes X_j]
\end{align*}
the bound
\begin{subequations}
\begin{align}
\label{BoundW}
||W_{ij}||_{\exp^{\gamma/2}} \leq C(\gamma,N) B^2 L^{-2d}.
\end{align}
Rewriting (with $p(m)$ denoting the smallest exponent with $2^{p(m)}>4K\log L \,2^m$, but at most $p(m)=\lfloor\log_2 L\rfloor+1$)
\begin{align*}
Z_{i}&=\sum_{j\in I:\chi_{ij}=1} X_j
=\sum_{m=0}^{1+\log_2 L} \sum_{j \in I:\chi_{ij}=1,m(j)=m} X_j
\\&
=\sum_{m=0}^{m(i)} \sum_{\substack{y\in 2^m \mathbb{Z}^d\cap[0,L)^d: \dist^\per_\infty(y+K \log L\cdot 2^{m}[-1,1]^d,y_i+K \log L \cdot 2^{m(i)}[-1,1]^d)\\ \leq 2\cdot 2^{m(i)}
K \log L}} X_y^m
\\&~~~
+\sum_{m=m(i)+1}^{1+\log_2 L} \sum_{\substack{y\in 2^m \mathbb{Z}^d\cap[0,L)^d: \dist^\per_\infty(y+K \log L\cdot 2^{m}[-1,1]^d,y_i+K \log L \cdot 2^{m(i)}[-1,1]^d)\\ \leq 2\cdot 2^{m} 
K \log L}} X_y^m
\\&
=\sum_{m=0}^{m(i)} \sum_{z\in 2^m \mathbb{Z}^d \cap [0,2^{p(m)})^d} \sum_{\substack{y\in 2^{p(m)} \mathbb{Z}^d\cap[0,L)^d:y+z\in [0,L)^d,
\\
\dist^\per_\infty(y+z+K \log L\cdot 2^{m}[-1,1]^d,y_i+K \log L \cdot 2^{m(i)}[-1,1]^d)\\ \leq 2\cdot 2^{m(i)}
K \log L}} X_{y+z}^m
\\&~~~
+\sum_{m=m(i)+1}^{1+\log_2 L} \sum_{\substack{y\in 2^m \mathbb{Z}^d\cap[0,L)^d: \dist^\per_\infty(y+K \log L\cdot 2^{m}[-1,1]^d,y_i+K \log L \cdot 2^{m(i)}[-1,1]^d)\\ \leq 2\cdot 2^{m}
K \log L}} X_y^m,
\end{align*}
we see that the sum on each level $m\leq m(i)$ may be written as the sum of $\leq C(d)(K \log L)^d$ sums of $\leq C(d) (2^{m(i)} K \log L / 2^{p(m)})^d = C(d) (2^{m(i)-m})^d$ independent random variables (to the latter sums we may apply a concentration estimate), while the sum on the levels $m>m(i)$ consists only of $\leq C(d)(K \log L)^d$ terms.
By Lemma~\ref{ConcentrationStretchedExponential} and \eqref{BoundMultilevelDependenceStructure} we deduce for $\gamma_1:=\gamma/(\gamma+1)$
\begin{align*}
||Z_i||_{\exp^{\gamma_1}} &\leq \sum_{m=0}^{m(i)} C(d,\gamma,N)(K \log L)^d \cdot C(\gamma) \big((2^{m(i)-m})^d\big)^{1/2} \cdot BL^{-d}
\\&~~~
+\sum_{m=m(i)+1}^{1+\log_2 L} C(d) (K \log L)^d B L^{-d}
\end{align*}
and therefore
\begin{align}
\label{BoundZ}
||Z_i||_{\exp^{\gamma_1}}
\leq C(d,\gamma,N) B (K \log L)^{d} (2^{m(i)})^{d/2} L^{-d} + C(d) B K^d (\log L)^{d+1} L^{-d}.
\end{align}
In a very similar way (note that for our choice of $\chi_{ijk}$ and $\chi_{ij}$ we have $\chi_{ijk}\leq \chi_{ik}+\chi_{jk}$), we can estimate
\begin{align*}
Z_{ij}=\sum_{k\in I:\chi_{ijk}=1} X_k
\end{align*}
as
\begin{align}
\label{BoundZ2}
||Z_{ij}||_{\exp^{\gamma_1}} \leq C(d,\gamma,N) B (K \log L)^{d} (2^{\max\{m(i),m(j)\}})^{d/2} L^{-d}
+ C(d) B K^d (\log L)^{d+1} L^{-d}.
\end{align}
It remains to bound
\begin{align*}
Y_{il}=\sum_{j\in I:m(j)<m(i),i^{m(i)}(j)=l,\chi_{ij}=1} \Big(X_i \otimes X_j-\mathbb{E}[X_i \otimes X_j]\Big).
\end{align*}
Note that
\begin{align*}
Y_{il}= X_i \otimes \hat Y_{il} - \mathbb{E}[X_i\otimes \hat Y_{il}],
\end{align*}
where
\begin{align*}
\hat Y_{il}:=\sum_{j\in I:m(j)<m(i),i^{m(i)}(j)=l,\chi_{ij}=1} X_j
=\sum_{m=0}^{m(i)-1} \sum_{j\in I:m(j)=m,i^{m(i)}(j)=l,\chi_{ij}=1} X_j.
\end{align*}
For each level $m$, the inner sum may again be written as a sum of $C(d)(K \log L)^d$ sums of $\lesssim \frac{(2^{m(i)-m})^d}{(K \log L)^d}$ independent random variables (recall that $y_{i^{m}(j)}=2^m \lfloor \frac{y_j}{2^m} \rfloor$). We therefore get by the concentration estimate in Lemma~\ref{ConcentrationStretchedExponential} and \eqref{BoundMultilevelDependenceStructure} for $\gamma_1:=\gamma/(\gamma+1)$
\begin{align*}
||\hat Y_{il}||_{\exp^{\gamma_1}} &\leq \sum_{m=0}^{m(i)-1} C(d) (K \log L)^d C(\gamma,N) \sqrt{(2^{m(i)-m})^d} \cdot B L^{-d}
\\&
\leq C(d,\gamma,N) B (K\log L)^{d} (2^{m(i)})^{d/2} L^{-d}.
\end{align*}
As a consequence, by \eqref{BoundMultilevelDependenceStructure} and Lemma~\ref{CalculusStretchedExponential}a we obtain for $\gamma_2=1/(1/\gamma+1/\gamma_1)=\gamma/(\gamma+2)$
\begin{align}
\label{BoundY}
||Y_{il}||_{\exp^{\gamma_2}} \leq C(d,\gamma,N) B^2 (K\log L)^{d} (2^{m(i)})^{d/2} L^{-2d}.
\end{align}
\end{subequations}
Choosing now the constants $\bar X_i$, $\bar Z_{ij}$, $\bar Z_i$, $\bar Y_{il}$, $\bar W_{il}$, for some $S\geq 1$ as
\begin{subequations}
\label{ChoicesBarQuantities}
\begin{align}
\bar X_i &:= C(d,\gamma,N) B (S \log L)^{1/\gamma} L^{-d},
\\
\bar W_{ij} &:= C(d,\gamma,N) B^2 (S \log L)^{2/\gamma} L^{-2d} ,
\\
\bar Z_{i} &:= C(d,\gamma,N) B (S \log L)^{1/\gamma_1} (K \log L)^{d+1} (2^{m(i)})^{d/2} L^{-d},
\\
\bar Z_{ij} &:= C(d,\gamma,N) B (S \log L)^{1/\gamma_1} (K \log L)^{d+1} (2^{\max\{m(i),m(j)\}})^{d/2} L^{-d},
\\
\bar Y_{il} &:= C(d,\gamma,N) B^2 (S \log L)^{1/\gamma_2} (K\log L)^{d} (2^{m(i)})^{d/2} L^{-2d},
\end{align}
\end{subequations}
we obtain that the condition \eqref{ConditionForNormalApproximation} is certainly satisfied if
\begin{align*}
&\sum_{i\in I,m(i)\leq\ell} \frac{N^{9/2} |\Lambda^{-1/2}|^3}{\varepsilon} \bigg(\sum_{\substack{j\in I:m(j)=m(i),\\ \chi_{ij}=1}} B^3 (S \log L)^{1/\gamma_2+1/\gamma_1} (K \log L)^{2d+1} (2^{m(i)})^{d} L^{-3d}
\\&~~~~~~~~~~~~~~~~~~~~~~~~~~~~~~~~~~
+B^3 (S \log L)^{1/\gamma+2/\gamma_1} (K \log L)^{2d+2} (2^{m(i)})^d L^{-3d}\bigg)
\\&~~~
+\sum_{i\in I,m(i)>\ell} N^{4} |\Lambda^{-1}| |\log \varepsilon|
\bigg(\sum_{\substack{j\in I:m(j)=m(i),\\ \chi_{ij}=1}} B^2 (S \log L)^{1/\gamma_2} (K \log L)^{d} (2^{m(i)})^{d/2} L^{-2d}
\\&~~~~~~~~~~~~~~~~~~~
+B^2 (S \log L)^{1/\gamma+1/\gamma_1} (K \log L)^{d+1} (2^{m(i)})^{d/2} L^{-2d}\bigg)
\\&~~~
\leq \frac{1}{C(d,\gamma,N)}.
\end{align*}
Note that the sum $\sum_{i\in I,m(i)=m}$ consists of the order of $(L/2^m)^{-d}$ terms and the sum $\sum_{j\in I:m(j)=m(i),\chi_{ij}=1}$ consists of the order of $C(d)(K \log L)^d$ terms. Thus, the condition \eqref{ConditionForNormalApproximation} is definitely satisfied if
\begin{align*}
&\sum_{m=0}^{\ell}C(d,\gamma,N) B^3 S^{1/\gamma_2+1/\gamma_1} \frac{K^{3d+2} (\log L)^{3d+1+1/\gamma_2+1/\gamma_1} |\Lambda^{-1/2}|^3 L^{-2d}}{\varepsilon}
\\&~~~
+ \sum_{m=\ell}^{1+\log_2 L} C(d,\gamma,N) B^2 S^{1/\gamma_2} K^{2d} |\Lambda^{-1}| |\log \varepsilon| (\log L)^{2d+1/\gamma_2} (2^m)^{-d/2} L^{-d}
\\&~~~
\leq \frac{1}{C}.
\end{align*}
This condition is satisfied for the choice
\begin{align}
\label{ChoiceEpsilon}
\varepsilon := C(d,\gamma,N) B^3 S^{1/\gamma_2+1/\gamma_1} K^{3d+2} (\log L)^{1+3d+1+1/\gamma_2+1/\gamma_1} |\Lambda^{-1/2}|^3 L^{-2d}
\end{align}
and $\ell$ as the smallest nonnegative integer with
\begin{align}
\label{ChoiceEll}
(2^\ell)^{d/2} \geq
C(d,\gamma,N) B^2 S^{1/\gamma_2} K^{2d} |\Lambda^{-1}| |\log (B^3 |\Lambda^{-1/2}|^3 L^{-2d})| (\log L)^{2d+1/\gamma_2} L^{-d}.
\end{align}
Note that the choice of $\varepsilon$ entails $\varepsilon\geq B^{3} |\Lambda^{-1/2}|^3 L^{-2d}$, which in turn implies by the bound $|\Lambda|\leq C(d,\gamma,N,K) B^2 L^{-d} |\log L|^{C(\gamma)}$ from Lemma~\ref{MultilevelVariableStretchedExponentialBound} that the lower bound
\begin{align}
\label{LowerBoundEps}
\varepsilon\geq c(d,\gamma,N,K) L^{-d/2} |\log L|^{C(d,\gamma)}
\end{align}
holds. Note furthermore that in case $\varepsilon\geq \frac{1}{2}$ with $\varepsilon$ as chosen in \eqref{ChoiceEpsilon} the assertion of our theorem \eqref{NormalApproximationMultilevel} becomes trivial: Indeed, for any random variable $X$ with $\mathbb{E}[X]=0$ we have the bound $\mathcal{D}(X,\mathcal{N}_\Lambda)\leq C (|\Var X|^{1/2}+ |\Lambda|^{1/2})$, which follows from
\begin{align*}
&|\mathbb{E}[\phi(X)]-\phi(0)|
\\
&\leq \mathbb{E}[\osc_{|X|} \phi(0)]
\\&
\leq \sum_{k=0}^\infty \dashint_{\{z\in \mathbb{R}^N:|z|\leq N|\Lambda^{1/2}|\}} \osc_{(k+1+N)|\Lambda^{1/2}|} \phi(z)\,dz ~\mathbb{P}[|\Lambda^{1/2}| k\leq |X|\leq |\Lambda^{1/2}| (k+1)]
\\&
\stackrel{\eqref{DefinitionClassPhi}}{\leq}  \sum_{k=0}^\infty C(N) |\Lambda^{1/2}| (k+1+N) ~\mathbb{P}[|\Lambda^{1/2}| k\leq |X|\leq |\Lambda^{1/2}| (k+1)]
\\&
\leq C |\Var X|^{1/2} + C |\Lambda^{1/2}|
\end{align*}
for any $\phi\in \Phi_\Lambda$ and the corresponding estimate for a Gaussian $\mathcal{N}_\Lambda$ in place of $X$.
We may therefore assume that the condition $\varepsilon\leq \frac{1}{2}$ -- which is required for the application of Theorem~\ref{NormalApproximation} -- is satisfied.

Let us now estimate the terms $\mathcal{R}_{lowlevel}$, $\mathcal{R}_{alllevel}$, and $\mathcal{R}_{tail}$ in Theorem~\ref{NormalApproximation}. We have by Lemma~\ref{CalculusStretchedExponential}b and our choices \eqref{ChoicesBarQuantities} as well as our bounds \eqref{BoundW}, \eqref{BoundZ}, \eqref{BoundZ2}, and \eqref{BoundY} (note that the inner sums in the next two lines contain at most $C(d)(K\log L)^d$ summands each, while the outer sum consists of $\leq L^d (2^{m})^{-d}$ summands of level $m(i)=m$)
\begin{align*}
\mathcal{R}_{tail}&=C |\Lambda^{-1}| N^{3/2} \varepsilon^{-N}
\sum_{i\in I} \Bigg(
\sum_{j\in I:m(j)=m(i),\chi_{ij}=1} \mathbb{E}\Big[|W_{ij}| |Z_{ij}| (\chi_{|Z_{ij}|>\bar Z_{ij}}+\chi_{|W_{ij}|> \bar W_{ij}})\Big]
\\&~~~~~~~~~~~~~~~~~~~~~~~~~~
+\sum_{l\in I:m(l)=m(i), \chi_{il}=1}
\mathbb{E}\Big[|Y_{il}|\,
|Z_{il}|(\chi_{|Z_{il}|>\bar Z_{il}}+\chi_{|Y_{il}|>\bar Y_{il}}) \Big]
\\&~~~~~~~~~~~~~~~~~~~~~~~~~~
+\mathbb{E}\Big[|X_i| |Z_{i}|^2 (\chi_{|Z_i|>\bar Z_i}+\chi_{|X_i|>\bar X_i})\Big]
\Bigg)
\\&
\leq
C |\Lambda^{-1}| N^{3/2} \varepsilon^{-N}
C(d) L^d (K \log L)^d
\\&~~~~~~~~~~~~\times
\bigg(
\max_{i,j} \mathbb{E}[\chi_{|Z_{ij}|>\bar Z_{ij}}+\chi_{|W_{ij}|>\bar W_{ij}}]^{1/2}
\max_{i,j}
\mathbb{E}[|W_{ij}|^2|Z_{ij}|^2]^{1/2}
\\&~~~~~~~~~~~~~~~~~~~
+\max_{i,j} \mathbb{E}[\chi_{|Z_{ij}|>\bar Z_{ij}}+\chi_{|Y_{ij}|>\bar Y_{ij}}]^{1/2}
\mathbb{E}[|Y_{ij}|^2 |Z_{ij}|^2]^{1/2}
\\&~~~~~~~~~~~~~~~~~~~
+\max_{i} \mathbb{E}[\chi_{|Z_i|>\bar Z_i}+\chi_{|X_{i}|>\bar X_{i}}]^{1/2}
\mathbb{E}[|X_i|^2 |Z_i|^4]^{1/2}
\bigg)
\\&
\leq
C |\Lambda^{-1}| N^{3/2} \varepsilon^{-N}
C(d) L^d (K \log L)^d
\times L^{-c(N,\gamma)S/2} (K\log L)^{C(d,\gamma)} B^3 L^{-3d}
\\&
\stackrel{\eqref{LowerBoundEps}}{\leq}
C(d,\gamma,K,N) |\Lambda^{-1}| B^3 L^{-5d}
\end{align*}
for $S$ chosen large enough (depending only on $d$, $\gamma$, and $N$).

Using also \eqref{ChoiceEpsilon}, we obtain first in case $\ell\geq 1$ (note that $S$ has now been chosen as a constant depending only on $d$, $\gamma$, and $N$; furthermore, note that for each fixed $i$ the inner sum in the next line runs over at most $C(d)(K \log L)^d$ elements)
\begin{align*}
\mathcal{R}_{lowlevel}&=
\frac{CN^{9/2} |\Lambda^{-1/2}|^3}{\varepsilon}\sum_{i\in I,m(i)\leq\ell} \bigg(\sum_{j\in I:m(j)=m(i), \chi_{ij}=1} \big(\bar W_{ij} \bar Z_{ij}^2 
+2 \bar Y_{ij}\, \bar Z_{ij}^2 \big)
+\bar X_i \bar Z_{i}^3\bigg)
\\&
\leq \frac{C(d,\gamma,K,N) |\Lambda^{-1/2}|^3}{\varepsilon} \sum_{m=0}^{\ell} \Big(\frac{L}{2^m}\Big)^d \cdot (K \log L)^d \cdot B^4 (\log L)^{C(d,\gamma)} (2^m)^{3d/2} L^{-4d}
\\&
\stackrel{\eqref{ChoiceEpsilon}}{\leq} C(d,\gamma,K,N) B \sum_{m=0}^{\ell} \Big(\frac{L}{2^m}\Big)^d \cdot (K \log L)^d \cdot (\log L)^{C(d,\gamma)} (2^m)^{3d/2} L^{-2d}
\\&
\leq C(d,\gamma,K,N) B \cdot (\log L)^{C(d,\gamma)} (2^\ell)^{d/2} L^{-d}
\\&
\stackrel{\eqref{ChoiceEll}}{\leq}
C(d,\gamma,K,N) B^3 \cdot |\log (B^3 |\Lambda^{-1/2}|^3 L^{-2d})|  (\log L)^{C(d,\gamma)} |\Lambda^{-1}| L^{-2d}.
\end{align*}
In the case $\ell=0$, the last estimate (and thus also the overall estimate) also holds true by $|\Lambda^{-1}|\geq |\Lambda|^{-1}\geq c B^{-2} L^d |\log L|^{-d}$ (the second inequality being a consequence of Lemma~\ref{MultilevelVariableStretchedExponentialBound}).

Using again \eqref{ChoiceEpsilon}, we deduce
\begin{align*}
\mathcal{R}_{alllevel}&=
CN^{9/2} |\Lambda^{-1/2}|^2 
|\log \varepsilon|
\sum_{i\in I} \bigg(\sum_{j\in I:m(j)=m(i),\chi_{ij}=1} \big(\bar W_{ij} \bar Z_{ij}
+2 \bar Y_{ij}\, \bar Z_{ij}\big)
+\bar X_i \bar Z_{i}^2\bigg)
\\&
\leq C(d,\gamma,K,N) B^3 |\Lambda^{-1/2}|^2 |\log (B^3 |\Lambda^{-1/2}|^3 L^{-2d})|
\\&~~~~~~~~\times
\sum_{m=0}^{1+\log_2 L}  \Big(\frac{L}{2^m}\Big)^d \cdot (K \log L)^d \cdot (\log L)^{C(d,\gamma)}(2^m)^{d} L^{-3d}
\\&
\leq C(d,\gamma,K,N) B^3 |\Lambda^{-1/2}|^2 |\log (B^3 |\Lambda^{-1/2}|^3 L^{-2d})| (\log L)^{C(d,\gamma)} L^{-2d}.
\end{align*}
As a consequence, we deduce from Theorem~\ref{NormalApproximation} (using again \eqref{ChoiceEpsilon})
\begin{align*}
&\mathcal{D}(X-\mathbb{E}[X],\mathcal{N}_\Lambda)
\\&~~~~~
\leq
C(d,\gamma,N,K) B^3 |\log (B^3 |\Lambda^{-1/2}|^3 L^{-2d})| \cdot (\log L)^{C(d,\gamma)} |\Lambda^{1/2}| |\Lambda^{-1/2}|^3 L^{-2d}.
\end{align*}
Using the bound $|\Lambda^{-1/2}|\geq |\Lambda|^{-1/2}\geq (C(d,\gamma,N,K)B^{2} L^{-d} (\log L)^d)^{-1/2}$ (the last inequality being a consequence of Lemma~\ref{MultilevelVariableStretchedExponentialBound}), we infer the first estimate of our theorem.

{\bf Step 2: Proof of the estimate \eqref{NormalApproximationMultilevelDegenerate}.}
To obtain the second estimate in our theorem which provides a better bound for degenerate covariance matrices $\Var X$, we repeat the preceding proof, however now adding $Q=L^d$ additional independent multivariate Gaussian random variables $G_1,\ldots,G_Q$ with zero expectation $\mathbb{E}[G_q]=0$ and variance $\Var G_q=\frac{1}{Q}(\Lambda-\Var X)$. This yields by an argument analogous to the above one (exploiting that the new additional random variables are independent from all others and using the fact that $\Var G_q = \frac{1}{Q} (\Lambda-\Var X) \leq L^{-2d}$)
\begin{align}
\nonumber
&\mathcal{D}(X+G-\mathbb{E}[X],\mathcal{N}_\Lambda)
\\&
\label{EstimateOnSmoothedX}
\leq
C(d,\gamma,N,K) B^3 (\log L)^{C(d,\gamma)} \big(L^{-d} |\Lambda^{1/2}| |\Lambda^{-1/2}|^3\big) L^{-d},
\end{align}
where $G:=\sum_{q=1}^Q G_q$. Note that $G$ has Gaussian moments with
\begin{align}
\label{GGaussianMoments}
||G||_{\exp^2}\leq C(N) \sqrt{|\Lambda-\Var X|}.
\end{align}
It now only remains to eliminate $G$ from the estimate \eqref{EstimateOnSmoothedX}.
Let $\phi\in \Phi_\Lambda$. Fixing $\bar G\geq |\Lambda-\Var X|^{1/2}$, we may rewrite for $\phi_\varepsilon$ as defined in \eqref{DefinitionPhiVarepsilon}
\begin{align*}
&\mathbb{E}[\phi_\varepsilon(X-\mathbb{E}[X])]-\int_{\mathbb{R}^N} \phi_\varepsilon(z) \mathcal{N}_\Lambda(z) \,dz
\\&
\leq \mathbb{E}[\phi_\varepsilon(X+G-\mathbb{E}[X])]-\int_{\mathbb{R}^N} \phi_\varepsilon(z) \mathcal{N}_\Lambda(z) \,dz
\\&~~~
+\mathbb{E}[|\phi_\varepsilon(X+G-\mathbb{E}[X])-\phi_\varepsilon(X-\mathbb{E}[X])|]
\\&
\leq \mathcal{D}^{\tilde L}(X+G-\mathbb{E}[X],\mathcal{N}_\Lambda)
+\mathbb{E}[\osc_{\bar G} \phi_\varepsilon(X+G-\mathbb{E}[X])]
\\&~~~
+\sup_z |\nabla \phi_\varepsilon(z)| ~ \mathbb{E}[|G|\chi_{|G|\geq \bar G}].
\end{align*}
An application of Lemma~\ref{PropertiesOfFunctionClass} c) and d) to the function
\begin{align*}
\hat h(x):=(\mathcal{N}_{\bar G^2 \Id} \ast \osc_{(2\sqrt{N}+1)\bar G} \phi_\varepsilon) (x)
\end{align*}
yields $\frac{1}{2\sqrt{N}+2}\cdot \frac{1}{40N}\hat h\in \Phi_\Lambda^{\tilde L}$ for $\tilde L$ large enough as well as $\osc_{\bar G} \phi_\varepsilon(x) \leq 2\hat h(x)$ and thus
\begin{align*}
&\mathbb{E}[\phi_\varepsilon(X-\mathbb{E}[X])]-\int_{\mathbb{R}^N} \phi_\varepsilon(z) \mathcal{N}_\Lambda(z) \,dz
\\&
\leq \mathcal{D}^{\tilde L}(X+G-\mathbb{E}[X],\mathcal{N}_\Lambda)
+\mathbb{E}[2\hat h(X+G-\mathbb{E}[X])]
+\sup_z |\nabla \phi_\varepsilon(z)| ~ \mathbb{E}[|G|\chi_{|G|\geq \bar G}]
\\&
\leq C(N) \mathcal{D}^{\tilde L}(X+G-\mathbb{E}[X],\mathcal{N}_\Lambda)
+\int_{\mathbb{R}^N} 2\hat h(z) \mathcal{N}_\Lambda(z) \,dz
+\sup_z |\nabla \phi_\varepsilon(z)| ~ \mathbb{E}[|G|\chi_{|G|\geq \bar G}]
\\&
\leq C(N) \mathcal{D}^{\tilde L}(X+G-\mathbb{E}[X],\mathcal{N}_\Lambda)
+C(N) \bar G
+\sup_z |\nabla \phi_\varepsilon(z)| ~ \mathbb{E}[|G|\chi_{|G|\geq \bar G}],
\end{align*}
where in the last step we have used the bound
\begin{align*}
&\int_{\mathbb{R}^N} \hat h(z) \mathcal{N}_\Lambda(z) \,dz
\\&
\leq \int_{\mathbb{R}^N} \int_{\mathbb{R}^N} \int_{\mathbb{R}^N} \mathcal{N}_{\bar G^2 \Id}(z-x) \mathcal{N}_{\varepsilon^2 \Lambda}(y) \osc_{\sqrt{1-\varepsilon^2}(2\sqrt{N}+1)\bar G} \phi(\sqrt{1-\varepsilon^2} x-y) \mathcal{N}_\Lambda(z) \,dx \,dy \,dz
\\&
\stackrel{\eqref{DefinitionClassPhi}}{\leq} (2\sqrt{N}+1)\bar G.
\end{align*}
Choosing $\bar G:=S |\log L| |(\Lambda-\Var X)^{1/2}|$ (with some constant $S$ to be chosen below) and $\varepsilon:=L^{-d}$ and using the estimate
\begin{align*}
|\nabla \phi_\varepsilon(z)|
&\leq \int_{\mathbb{R}^N} |\nabla \phi|(\sqrt{1-\varepsilon^2}z-y)
\mathcal{N}_{\varepsilon^2 \Lambda}(y) \,dy
\\&
\leq \varepsilon^{-N} \int_{\mathbb{R}^N} |\nabla \phi|(\sqrt{1-\varepsilon^2}z-y)
\mathcal{N}_{\Lambda}(y) \,dy
\stackrel{\eqref{DefinitionClassPhi}}{\leq} \varepsilon^{-N} \leq L^{C(d,N)}
\end{align*}
as well as Lemma~\ref{CalculusStretchedExponential}b and \eqref{GGaussianMoments},
we deduce for all $\phi\in \Phi_\Lambda$ that
\begin{align*}
&\mathbb{E}[\phi_\varepsilon(X-\mathbb{E}[X])]-\int_{\mathbb{R}^N} \phi_\varepsilon(z) \mathcal{N}_\Lambda(z) \,dz
\\&
\leq C(N) \mathcal{D}^{\tilde L}(X+G-\mathbb{E}[X],\mathcal{N}_\Lambda)
+C(N) S |\log L| |(\Lambda-\Var X)^{1/2}|
\\&~~~
+C(N) L^{C(d,N)} \cdot |\Lambda^{1/2}| \cdot \exp(-c S \log L).
\end{align*}
Lemma~\ref{MultilevelVariableStretchedExponentialBound} and our assumption $|\Lambda-\Var X|\leq L^{-d}$ imply the upper bound $|\Lambda|\leq C(d,\gamma,N,K)B^2 L^{-d} |\log L|^d$.
As a consequence, choosing the constant $S$ large enough and using the notation from Lemma~\ref{SmoothingEstimateLemma} we obtain
\begin{align*}
\mathcal{D}_\varepsilon^{\tilde L}(X-\mathbb{E}[X],\mathcal{N}_\Lambda)
&\leq
C(N) \mathcal{D}^{\tilde L}(X+G-\mathbb{E}[X],\mathcal{N}_\Lambda)
+C(d,N) |\log L| |(\Lambda-\Var X)^{1/2}|
\\&~~~~
+C(d,\gamma,N,K) B |\log L|^C L^{-d}.
\end{align*}
Using Lemma~\ref{SmoothingEstimateLemma}, we conclude that
\begin{align*}
&\mathcal{D}^{\tilde L}(X-\mathbb{E}[X],\mathcal{N}_\Lambda)
\\
&\leq
C(N) |\Lambda^{1/2}| \varepsilon
+C(N) \mathcal{D}^{\tilde L}_\varepsilon(X-\mathbb{E}[X],\mathcal{N}_\Lambda)
\\&
\leq
C(d,\gamma,N,K) B (\log L)^{d/2} L^{-d/2} \cdot L^{-d}
+C(N) \mathcal{D}^{\tilde L}(X+G-\mathbb{E}[X],\mathcal{N}_\Lambda)
\\&~~~~
+C(d,N) |\log L| |(\Lambda-\Var X)^{1/2}|
+C(d,\gamma,N,K) B |\log L|^C L^{-d}
\\&
\stackrel{\eqref{EstimateOnSmoothedX}}{\leq}
C(d,\gamma,N,K) |\log L|^{C(d,\gamma)} (B+B^3 L^{-d}|\Lambda^{1/2}||\Lambda^{-1/2}|^3)L^{-d}
\\&~~~~
+C(d,N) |\log L| |(\Lambda-\Var X)^{1/2}|.
\end{align*}
This yields the estimate \eqref{NormalApproximationMultilevelDegenerate} upon noticing that $B\leq C B^3 |\log L|^C L^{-d}|\Lambda^{1/2}||\Lambda^{-1/2}|^3$ (by Lemma~\ref{MultilevelVariableStretchedExponentialBound} and our assumption $|\Lambda-\Var X|\leq L^{-d}$).
\end{proof}

We next prove our normal approximation result for integral functionals of random fields which may be approximated well by random fields with finite dependency range. The proof is a simple reduction to the statement of Theorem~\ref{TheoremNormalApproximationMultilevelLocalDependence}.
\begin{proof}[Proof of Theorem~\ref{TheoremRandomField}]
{\bf Step 1: Reduction to random fields supported on $[0,L]^d$.}
We will show that we may restrict ourselves to random fields $v$ and $v_r$ that vanish identically outside of a ball of the form $B_{L/2}$, from which point on the reduction to the case of random field supported on the cube $[0,L]^d$ is straightforward. To this aim, let $\tau>0$ and $l\geq K$ and introduce the change of variables
\begin{align*}
\tilde x:= \varphi(x) := L \tau\int_0^{L^{-1} |x|} \frac{1+s^{2l}}{1+s^{2+2l}} \,ds ~ \frac{x}{|x|}.
\end{align*}
Note that this change of variables maps $\mathbb{R}^d$ to $B_{c_l L\tau}$ for some constant $c_l$ depending only on $l$. Furthermore, observe that we have
\begin{align*}
\det \nabla \varphi(x)= \tau^d \frac{1+L^{-2l}|x|^{2l}}{1+L^{-2-2l}|x|^{2+2l}} |\varphi(x)|^{d-1} |x|^{-(d-1)}.
\end{align*}

We define $\tilde a(\tilde x):=a(x)$, $\tilde v(\tilde x):=v(x)$, and $\tilde v_r(\tilde x):=v_r(x)$ for any $\tilde x =\varphi(x) \in B_{c_l L\tau}$, and extend those random fields by zero outside of the ball $B_{c_l L\tau}$.
Similarly, we define
\begin{align*}
\tilde \xi(\tilde x):=\det \nabla (\varphi^{-1})(\tilde x) \xi(x)=\frac{\xi(x)}{\det \nabla \phi(x)}
\end{align*}
and extend $\tilde \xi(\tilde x)$ to a $W^{K,\infty}$ function supported in $B_{c_l L\tau}$. Note that by our assumptions on $\xi$, such an extension exists and we may derive an estimate of the form
\begin{align*}
|\nabla^k \tilde \xi| \leq C L^{-k}
\end{align*}
for all $0\leq k\leq K$.
%\begin{align*}
%|\nabla^m \tilde \xi|(\tilde x)
%\leq
%\det \nabla \varphi^{-1}(\varphi(x)) |\nabla^k | 
%+
%\end{align*}

As the map $\varphi$ is Lipschitz (with a Lipschitz constant independent of $L$), for $\tau>0$ small enough (depending only on the spatial dimension $d$ and $l$) it preserves the finite range of dependence property of the random field $a$. Furthermore, it maps $\mathbb{R}^d$ to the ball $B_{c_l L\tau}$ and we have
\begin{align*}
F=\int_{\mathbb{R}^d} v \xi \,dx = \int_{\mathbb{R}^d} \tilde v \tilde \xi \,d\tilde x.
\end{align*}
To establish the normal approximation result for $\int_{\mathbb{R}^d} v \xi \,dx$, it therefore suffices to establish the corresponding result for $\tilde v$ and $\tilde \xi$. We have already verified that $\tilde \xi$ satisfies the assumptions of the theorem (up to a constant factor). Thus, it only remains to establish the properties of the random fields $\tilde v_r$.

For $\tau>0$ small enough (depending only on $d$ and $l$), the map $\varphi$ is $1$-Lipschitz; hence, the random field $\tilde v_r$ inherits the $r$-local dependence on $\tilde a$ from the $r$-local dependence of the random field $v_r$ on $a$. Given any $\tilde \psi\in W^{K,\infty}$, we also have the estimate
\begin{align*}
\bigg|\int_{\mathbb{R}^d} (\tilde v-\tilde v_r) \tilde \psi \,d\tilde x\bigg|
&= \bigg|\int_{\mathbb{R}^d} (v-v_r)(x) \tilde \psi(\varphi(x)) \det \nabla \varphi(x) \,dx\bigg|
\\&
\leq \mathcal{C} r^{-d} \int_{\mathbb{R}^d} \sup_{y\in B_r(x)} \sum_{k=0}^K r^{k} |\nabla^k (\tilde \psi(\varphi(\cdot)) \det \nabla \varphi(\cdot))|(y) \,dx
\\&
\leq C \mathcal{C} r^{-d} \int_{\mathbb{R}^d} \sup_{y\in B_r(x)} \sum_{k=0}^K r^{k} |(\nabla^k \tilde \psi)(\varphi(y)) \det \nabla \varphi(y)| \,dx
\end{align*}
where in the last step we have used $r\leq L$. Using the $1$-Lipschitz property of $\varphi$, this entails the desired estimate
\begin{align*}
\bigg|\int_{\mathbb{R}^d} (\tilde v-\tilde v_r) \tilde \psi \,d\tilde x\bigg|
\leq C \mathcal{C} r^{-d} \int_{\mathbb{R}^d} \sup_{\tilde y\in B_r(\tilde x)} \sum_{k=0}^K r^{k} |(\nabla^k \tilde \psi)(\tilde y)| \,d\tilde x.
\end{align*}
This finishes the reduction to the case of random fields $v$ and $v_r$ supported on the cube $[0,L]^d$.

%We also readily check that $|\tilde \xi(\tilde x)|= |\xi(x)| |\det \nabla \varphi^{-1}(\tilde x)| \leq C (L+|x|)^{-d-1} \cdot C (1+|x|)^{d+1} \leq C$ and $|\nabla^l \tilde \xi|\leq C(l) L^{-l}$. For the remainder of the proof, after rescaling and translation we may therefore restrict ourselves to random fields $v$ and $v_r$ that vanish identically outside of the cube $[0,L]^d$.

{\bf Step 2: Proof for random field supported in $[0,L]^d$.}
In order to establish our result for random field $v$ and $v_r$ which are supported on the cube $[0,L]^d$, we shall reduce it to the normal approximation result of Theorem~\ref{TheoremNormalApproximationMultilevelLocalDependence}. 
For each $0\leq m\leq \log_2 L+1$, introduce a partition of unity $\eta_y$, $y\in 2^m \mathbb{Z}^d\cap [0,L)^d$, with $\supp \eta_y\subset y+[0,2^m]$ and $|\nabla^l \eta_y|\leq C (2^m)^{-k}$ for $0\leq k\leq K$. We may then rewrite
\begin{align*}
X&=L^{-d} \int_{\mathbb{R}^d} v_1 \xi \,dx+\sum_{m=0}^{\log_2 L-1} L^{-d} \int_{\mathbb{R}^d} (v_{2^{m+1}}-v_{2^m}) \xi \,dx
\\&~~~
+L^{-d} \int_{\mathbb{R}^d} (v-v_{2^{\lfloor\log_2 L\rfloor}}) \xi \,dx
\\&
=\sum_{y\in \mathbb{Z}^d} \underbrace{L^{-d} \int_{\mathbb{R}^d} v_1 \xi \eta_y^0 \,dx}_{=:X_y^0}
+\sum_{m=0}^{\log_2 L-1} \sum_{y\in 2^{m+1} \mathbb{Z}^d} \underbrace{L^{-d} \int_{\mathbb{R}^d}  (v_{2^{m+1}}-v_{2^m}) \xi \eta_y^{m+1} \,dx}_{=:X_y^{m+1}}
\\&~~~
+\underbrace{L^{-d} \int_{\mathbb{R}^d} (v-v_{2^{\lfloor\log_2 L\rfloor}}) \xi \,dx}_{=:X_0^{\log L+1}}.
\end{align*}
The $X_y^m$ give rise to a multilevel local dependence structure in the sense of Definition~\ref{ConditionRandomVariable}. In particular, we easily check that by the bound on $v$ and the assumption \eqref{AssumptionApproximationByFiniteRange} we have $||X_y^m||_{\exp^\gamma}\leq CL^{-d}$ for all $m$ and all $y\in 2^m \mathbb{Z}^d\cap [0,L)^d$ (as we have $|\nabla^l (\xi \eta_y^m)|\leq C (2^m)^{-l}$). Furthermore, the localized dependence of the $X_y^m$ on $a$ required by Definition~\ref{ConditionRandomVariable} follows directly from our definition of $X_y^m$ and our assumption on $v_r$ (if one chooses $K$ in Definition~\ref{ConditionRandomVariable} large enough). The statement of our theorem is then a direct consequence of the normal approximation result from Theorem~\ref{TheoremNormalApproximationMultilevelLocalDependence}.
\end{proof}

\section{Proof of the Result on Moderate Deviations}

We now provide the proof of our moderate deviations result for sums of random variables with multilevel local dependence structure.
\begin{proof}[Proof of Theorem~\ref{TheoremModerateDeviations}]
To simplify notation, we only consider the case $L=2^{\hat k}$ for some $\hat k\in \mathbb{N}$; the proof for the general case is similar. Again, without loss of generality we may assume $\mathbb{E}[X_y^k]=0$ for all $k$ and all $y$.
\\
{\bf Step 1} (Decomposition of $X$).
\\
We first decompose $X$ into groups of terms $\mathcal{G}_i$ and ``remainder terms'' $\mathcal{R}$. The groups $\mathcal{G}_i$ are stochastically independent from each other; each group heuristically consists of the random variables summed over a cube of diameter $\ell-2\cdot 2^{p(m)}$, where $1\ll \ell\ll L$ is an intermediate length scale that we are going to choose and where $2^{p(m)}\ll \ell$ is (half of) the size of the gaps between the groups (which we introduce in order to achieve independence of the groups). The fact that the groups $\mathcal{G}_i$ sum up all random variables over an intermediate length scale $\ell$ allows us to apply normal approximation to the groups $\mathcal{G}_i$.
The rest $\mathcal{R}$ basically corresponds to the random variables ``between the groups'' (to achieve independence of the groups) and the random variables with long-range dependencies. We shall prove that these terms are small in a suitable sense.

More precisely, we introduce an intermediate scale $\ell=2^{\tilde k}$ (with $\tilde k\in \mathbb{N}$ to be chosen below) and define $m_0:=\lfloor \log_2 \frac{\ell}{4K \log L}\rfloor$. This enables us to rewrite the random variable $X$ as
\begin{align*}
X&=
\sum_{m=0}^{1+\log_2 L} \sum_{i\in 2^m \mathbb{Z}^d \cap [0,L)^d} X_i^m
\\&=
\sum_{m=0}^{m_0} \sum_{i\in \ell \mathbb{Z}^d \cap [0,L)^d} \sum_{j\in 2^m \mathbb{Z}^d \cap [0,\ell)^d} X_{i+j}^m
+\sum_{m=m_0+1}^{1+\log_2 L} \underbrace{\sum_{i\in 2^m \mathbb{Z}^d \cap [0,L)^d} X_{i}^m}_{=:\mathcal{R}^{m}}
\end{align*}
which gives (defining $p(m)$ to be the smallest integer with $2^{p(m)}\geq 2^{m+2} K \log L$)
\begin{align*}
X&
=\sum_{m=0}^{m_0} \sum_{i\in \ell \mathbb{Z}^d \cap [0,L)^d} \sum_{j\in 2^m \mathbb{Z}^d \cap [2^{p(m)},\ell-2^{p(m)})^d} X_{i+j}^m
\\&~~~
+\sum_{m=0}^{m_0} \underbrace{\sum_{i\in \ell \mathbb{Z}^d \cap [0,L)^d} \sum_{j\in 2^m \mathbb{Z}^d \cap [0,\ell)^d \setminus [2^{p(m)},\ell-2^{p(m)})^d} X_{i+j}^m}_{=:\mathcal{R}^{m}}
\\&~~~
+\sum_{m=m_0+1}^{1+\log_2 L} \underbrace{\sum_{i\in 2^m \mathbb{Z}^d \cap [0,L)^d} X_i^m}_{=:\mathcal{R}^{m}}.
\end{align*}
Exchanging the order of summation in the first term and defining
\begin{align}
\label{DefinitionMainGroups}
\mathcal{G}_i:=\sum_{m=0}^{m_0}  \sum_{j\in 2^m \mathbb{Z}^d \cap [2^{p(m)},\ell-2^{p(m)})^d} X_{i+j}^m,
\end{align}
we obtain
\begin{align*}
X&
=\sum_{i\in \ell \mathbb{Z}^d \cap [0,L)^d} \mathcal{G}_i
+\sum_{m=0}^{m_0} \mathcal{R}^{m}
+\sum_{m=m_0+1}^{1+\log_2 L} \mathcal{R}^{m}.
\end{align*}
{\bf Step 2} (Estimate on the terms $\mathcal{R}^m$ for $m>m_0$).\\
The remaining terms on level $m$ for $m_0+1\leq m\leq 1+\log_2 L$
\begin{align*}
\mathcal{R}^{m}:=\sum_{i\in 2^m \mathbb{Z}^d \cap [0,L)^d} X_i^m
\end{align*}
(observe that the sum consists of $(\frac{L}{2^m})^d$ terms) may be grouped into $\sim (K \log L)^d$ groups, each only consisting of $\sim \frac{(\frac{L}{2^m})^d}{(K \log L)^d}$ independent random variables: Choosing $p(m)$ as before to be the smallest integer with $2^{p(m)}\geq 2^{m+2} K \log L$ (but choosing $p(m)=\log_2 L$ if this integer were larger than $\log_2 L$), we have
\begin{align*}
\mathcal{R}^{m}=\sum_{j\in 2^m \mathbb{Z}^d \cap [0,2^{p(m)})^d} \underbrace{\sum_{i\in 2^{p(m)} \mathbb{Z}^d \cap [0,L)^d} X_{i+j}^m}_{=:\mathcal{R}^{m,j}}.
\end{align*}
By Definition~\ref{ConditionRandomVariable}, the random variables in each $\mathcal{R}^{m,j}$ are now independent and we deduce from Lemma~\ref{ConcentrationStretchedExponential} (with $\tilde \gamma:=\gamma/(\gamma+1)$)
\begin{align*}
\big|\big|\mathcal{R}^{m,j}\big|\big|_{\exp^{\tilde \gamma}}
\leq C(\gamma) \sqrt{\Big(\frac{L}{2^{p(m)}}\Big)^d} \max_i ||X_i^m||_{\exp^\gamma}
\end{align*}
and as a consequence (using the bounds from Definition~\ref{ConditionRandomVariable})
\begin{align*}
&\bigg|\bigg|\sum_{m=m_0+1}^{1+\log_2 L}\mathcal{R}^{m}\bigg|\bigg|_{\exp^{\tilde \gamma}}
=\bigg|\bigg|\sum_{m=m_0+1}^{1+\log_2 L} \sum_{j\in 2^m \mathbb{Z}^d \cap [0,2^{p(m)})^d} \mathcal{R}^{m,j}\bigg|\bigg|_{\exp^{\tilde \gamma}}
\\
&\leq \sum_{m=m_0+1}^{1+\log_2 L}
\sum_{j\in 2^m \mathbb{Z}^d \cap [0,2^{p(m)})^d} C(\gamma) \sqrt{\Big(\frac{L}{2^{p(m)}}\Big)^d} \max_i ||X_i^m||_{\exp^\gamma}
\\&
\leq
\sum_{m=m_0+1}^{1+\log_2 L}
\Big(\frac{2^{p(m)}}{2^m}\Big)^d C(\gamma) \sqrt{\Big(\frac{L}{2^{p(m)}}\Big)^d} \max_i ||X_i^m||_{\exp^\gamma}
\\&
\leq
\sum_{m=m_0+1}^{1+\log_2 L}
 C(d) (K \log L)^d C(\gamma) \sqrt{\Big(\frac{L}{2^m K \log L}\Big)^d} B L^{-d}
\\&
\leq
 C(d) C(\gamma) (K \log L)^{d/2} (2^{m_0})^{-d/2} B L^{-d/2}
\end{align*}
which yields by the choice of $m_0$
\begin{align}
\label{EstimateRmBig}
\bigg|\bigg|\sum_{m=m_0+1}^{\log_2 L}\mathcal{R}^{m}\bigg|\bigg|_{\exp^{\tilde \gamma}}
\leq
 C(d,\gamma) B (K \log L)^d \ell^{-d/2} L^{-d/2}.
\end{align}
{\bf Step 3} (Estimate on the terms $\mathcal{R}^m$ for $m\leq m_0$).\\
The remaining terms on level $m$ for $0\leq m\leq m_0$
\begin{align*}
\mathcal{R}^{m}=\sum_{i\in \ell \mathbb{Z}^d \cap [0,L)^d} \sum_{j\in 2^m \mathbb{Z}^d \cap [0,\ell)^d \setminus [2^{p(m)},\ell-2^{p(m)})^d} X_{i+j}^m
\end{align*}
may be grouped as follows into a sum of sums of independent random variables: We have
\begin{align*}
\mathcal{R}^{m}
&=\sum_{i\in \ell \mathbb{Z}^d \cap [0,L)^d} \sum_{j\in 2^{p(m)}\mathbb{Z}^d\cap [0,\ell)^d \setminus [2^{p(m)},\ell-2^{p(m)})^d} \sum_{k\in 2^m \mathbb{Z}^d \cap [0,2^{p(m)})^d} X_{i+j+k}^m
\\&
= \sum_{k\in 2^m \mathbb{Z}^d \cap [0,2^{p(m)})^d} \underbrace{\sum_{i\in \ell \mathbb{Z}^d \cap [0,L)^d} \sum_{j\in 2^{p(m)}\mathbb{Z}^d\cap [0,\ell)^d \setminus [2^{p(m)},\ell-2^{p(m)})^d} X_{i+j+k}^m}_{=:\mathcal{R}^{m,k}}.
\end{align*}
Note that $\mathcal{R}^{m,k}$ is a sum of at most $(\frac{L}{\ell})^d \cdot \frac{C(d)\ell^{d-1}2^{p(m)}}{(2^{p(m)})^d}$ independent random variables. An application of Lemma~\ref{ConcentrationStretchedExponential} yields
\begin{align*}
\big|\big|\mathcal{R}^{m,k}\big|\big|_{\exp^{\tilde \gamma}}
\leq C(\gamma) \sqrt{\Big(\frac{L}{\ell}\Big)^d \cdot \frac{C(d)\ell^{d-1}2^{p(m)}}{(2^{p(m)})^d}} \max_i ||X_i^m||_{\exp^\gamma}
\end{align*}
which entails
\begin{align*}
&\bigg|\bigg|\sum_{m=0}^{m_0}\mathcal{R}^{m}\bigg|\bigg|_{\exp^{\tilde \gamma}}
=\bigg|\bigg|\sum_{m=0}^{m_0}\sum_{k\in 2^m \mathbb{Z}^d \cap [0,2^{p(m)})^d} \mathcal{R}^{m,k}\bigg|\bigg|_{\exp^{\tilde \gamma}}
\\&
\leq
\sum_{m=0}^{m_0} \sum_{k\in 2^m \mathbb{Z}^d \cap [0,2^{p(m)})^d} C(\gamma) \sqrt{\Big(\frac{L}{\ell}\Big)^d \cdot \frac{C(d)\ell^{d-1}2^{p(m)}}{(2^{p(m)})^d}} \max_i ||X_i^m||_{\exp^\gamma}
\\&
\leq
\sum_{m=0}^{m_0} \Big(\frac{2^{p(m)}}{2^m}\Big)^d C(\gamma) \sqrt{\Big(\frac{L}{\ell}\Big)^d \cdot \frac{C(d)\ell^{d-1}2^{p(m)}}{(2^{p(m)})^d}} \max_i ||X_i^m||_{\exp^\gamma}
\\&
\leq
\sum_{m=0}^{m_0} C(d) (K \log L)^d C(\gamma) \sqrt{\Big(\frac{L}{\ell}\Big)^d \cdot \frac{C(d)\ell^{d-1}}{(2^m)^{d-1}(K \log L)^{d-1}}} B L^{-d}
\\&
\leq
\sum_{m=0}^{m_0} C(d) C(\gamma) (K \log L)^{(d+1)/2} \ell^{-1/2} (2^m)^{-(d-1)/2} B L^{-d/2}
\end{align*}
and as a consequence
\begin{align}
\label{EstimateRmSmall}
\bigg|\bigg|\sum_{m=0}^{m_0}\mathcal{R}^{m}\bigg|\bigg|_{\exp^{\tilde \gamma}}
\leq C(d,\gamma) B (K \log L)^{(d+3)/2} \ell^{-1/2} L^{-d/2}.
\end{align}
{\bf Step 4 (Estimate on the ``bulk contributions'' $\mathcal{G}_i$)}.
\\
The formula \eqref{DefinitionMainGroups} may be rewritten as (recalling that $p(m)$ is the smallest integer with $2^{p(m)}\geq 2^{m+2} K \log L$)
\begin{align*}
\mathcal{G}_i=\sum_{m=0}^{m_0} \sum_{k\in 2^m \mathbb{Z}^d \cap [0,2^{p(m)})^d}
\sum_{j\in 2^{p(m)}\mathbb{Z}^d\cap [2^{p(m)},\ell-2^{p(m)})^d} X_{i+j+k}^m,
\end{align*}
which yields upon applying Lemma~\ref{ConcentrationStretchedExponential} to the inner sum
\begin{align}
\nonumber
||\mathcal{G}_i||_{\exp^{\tilde \gamma}}
&\leq \sum_{m=0}^{m_0} \sum_{k\in 2^m \mathbb{Z}^d \cap [0,2^{p(m)})^d} C(d,\gamma)
\bigg(\frac{\ell}{2^{p(m)}}\bigg)^{d/2} B L^{-d}
\\&
\nonumber
\leq \sum_{m=0}^{m_0} C(d,\gamma) B
\bigg(\frac{\ell}{2^m K\log L}\bigg)^{d/2} (K \log L)^d L^{-d}
\\&
\label{EstimateGroupTails}
\leq C(d,\gamma) B (K \log L)^{d/2} \ell^{d/2} L^{-d}.
\end{align}

Next, to each of the groups
\begin{align*}
\mathcal{G}_i:=\sum_{m=0}^{m_0}  \sum_{j\in 2^m \mathbb{Z}^d \cap [2^{p(m)},\ell-2^{p(m)})^d} X_{i+j}^m,
\end{align*}
we apply Theorem~\ref{TheoremNormalApproximationMultilevelLocalDependence} with $L$ replaced by $\ell$; note that we may rescale our variables $X_{i+j}^m$ in the group $\mathcal{G}_i$ by a factor of $(\frac{L}{\ell})^d$, as we have the bound $||X_{i+j}^m||_{\exp^\gamma}\leq B L^{-d}$ while to apply Theorem~\ref{TheoremNormalApproximationMultilevelLocalDependence} to the group $\mathcal{G}_i$ we only need the estimate $||X_{i+j}^m||_{\exp^\gamma}\leq B \ell^{-d}$. We then obtain for any positive matrix $\Lambda_i\geq \Var \mathcal{G}_i$ with $|\Lambda_i-\Var \mathcal{G}_i|\leq \ell^{d} L^{-2d}$
\begin{align*}
&\mathcal{D}\Big[\Big(\frac{L}{\ell}\Big)^d \mathcal{G}_i,\mathcal{N}_{(L/\ell)^{2d}\Lambda_i}\Big]
\\&~~~~~~
\leq
C(d,\gamma,N,K) B^3 (\log \ell)^{C(d,\gamma)} \ell^{-d} \Big(\frac{L}{\ell}\Big)^d |\Lambda_i^{1/2}| \Big(\frac{L}{\ell}\Big)^{-3d} |\Lambda_i^{-1/2}|^3 \ell^{-d}
\\&~~~~~~~~~~
+C(d,N) (\log \ell)^{C(d,\gamma)} \Big(\frac{L}{\ell}\Big)^d |\Lambda_i-\Var \mathcal{G}_i|^{1/2}.
\end{align*}
Rescaling and using the fact that the $1$-Wasserstein distance $\mathcal{W}_1$ is bounded by our distance $\mathcal{D}$, we deduce
\begin{align*}
\mathcal{W}_1[\mathcal{G}_i,\mathcal{N}_{\Lambda_i}] \leq&
C(d,\gamma,N,K) B^3 (\log L)^{C(d,\gamma)} |\Lambda_i^{1/2}| |\Lambda_i^{-1/2}|^3 \ell^{d} L^{-3d}
\\&~~~
\nonumber
+C(d,N) (\log L)^{C(d,\gamma)} |\Lambda_i-\Var \mathcal{G}_i|^{1/2}.
\end{align*}
We choose $\ell$ with $L^{1/2}\leq \ell \leq 2L^{1/2}$ and set
\begin{align}
\label{ChoiceLambdai}
\Lambda_i:=\Var \mathcal{G}_i+\ell^{d-1/4} L^{-2d} \Id_{N\times N}.
\end{align}
Note that this choice entails
\begin{align*}
|\Lambda_i^{-1/2}| \leq \ell^{-d/2+1/8} L^{d}.
\end{align*}
As a consequence of these estimates, the choice of $\ell$, and the estimate on $|\Lambda_i^{1/2}|$ deduced from \eqref{EstimateGroupTails}, we obtain
\begin{align}
\label{GroupNormalApproximation}
\mathcal{W}_1[\mathcal{G}_i,\mathcal{N}_{\Lambda_i}] \leq&
C(d,\gamma,N,K) B^4 (\log L)^{C(d,\gamma)} \ell^{d/2-1/8} L^{-d}.
\end{align}

We are now in position to apply Lemma~\ref{ConcentrationCloseToGaussian} to the sum $\sum_{i\in \ell\mathbb{Z}^d\cap [0,L)^d} \mathcal{G}_i$ (recall that the $\mathcal{G}_i$ form a collection of independent random variables). Note that in our setting we have $M=(L/\ell)^d$. By our estimates \eqref{GroupNormalApproximation} and \eqref{EstimateGroupTails}, in our application of Lemma~\ref{ConcentrationCloseToGaussian} we may choose $b:=C(d,\gamma,K) B^4 (\log L)^{d/2} \ell^{d/2} L^{-d}$ and any $\tau\leq \frac{1}{2}$ with
\begin{align*}
\tau\geq C(d,\gamma,N,K) (\log L)^{C(d,\gamma)} \ell^{-1/8}.
\end{align*}
With this choice, Lemma~\ref{ConcentrationCloseToGaussian} yields
\begin{align}
\label{SumOfGroups}
\sum_{i\in \ell \mathbb{Z}^d \cap [0,L)^d} \mathcal{G}_i \stackrel{d}{=} Y + Z
\end{align}
where $Y$ is a multivariate Gaussian random variable with covariance matrix
\begin{align}
\label{TildeLambda}
\tilde \Lambda:=\sum_{i\in \ell \mathbb{Z}^d \cap [0,L)^d} \Lambda_i
\end{align}
and $Z$ satisfies the estimate
\begin{align*}
\mathbb{P}[|Z|\geq r]
\leq
3N \exp\bigg(-\frac{r^2}{C(d,\gamma,K,N) B^8 \tau |\log \tau|^{1/\tilde \gamma} |\log L|^{C(d,\gamma)}L^{-d}} \bigg)
\end{align*}
for any
\begin{align*}
r\leq &\sqrt{\tau |\log \tau|^{1/\tilde \gamma}} c L^{-d/2}
\\&~~~~~~~~
\times
\min\bigg\{\frac{\sqrt{(L/\ell)^d \tau |\log \tau|^{1/\gamma}}}{(2\log (2(L/\ell)^d))^{1/\tilde \gamma}},\big(\tau |\log \tau|^{1/\tilde \gamma} (L/\ell)^d\big)^{\tilde \gamma/(4+2\tilde \gamma)}
\bigg\}.
\end{align*}
A crude estimate on $|\log \tau|$ of the form $|\log \tau| \leq C(d,\gamma,N,K) (\log L)^{C(d,\gamma)}$
%-- for which we use $|\Lambda_i^{1/2}| |\Lambda_i^{-1/2}|\geq 1$ and $|\Lambda_i^{1/2}|\leq C(d,\gamma,N,B,K) \ell^{d/2}L^{-d}$, the latter of which is a direct consequence of the derived moment bounds for $\mathcal{G}_i$ --
yields using also $(L/\ell)^d \tau\geq c$ (which holds by the lower bound on $\tau$ and the choice $\ell\sim L^{1/2}$)
%$|\Lambda_i^{-1/2}| \leq c (\log L)^{-C} L^{d}$ (which we may assume without loss of generality, see below, and which entails $M\tau\leq 1$)
\begin{align*}
\mathbb{P}[|Z|\geq r]
\leq
3N \exp\bigg(-\frac{r^2}{\tau \cdot C(d,\gamma,K,N) B^8 (\log L)^{C(d,\gamma)} L^{-d}} \bigg)
\end{align*}
for any
\begin{align*}
r\leq L^{-d/2} \cdot \Big(\frac{L}{\ell}\Big)^{d \tilde \gamma/(4+2\tilde \gamma)} \cdot c(d,\gamma,N,K) (\log L)^{-C(d,\gamma)} \tau^{1/2+\tilde \gamma/(4+2\tilde \gamma)}.
\end{align*}
We now set $\tau:=c(d,\gamma,N,K) L^{-\min\{d\tilde\gamma/2(4+2\tilde \gamma),1/32\}}$. The previous estimate then yields
\begin{align*}
\mathbb{P}[|Z|\geq r]
\leq
3N \exp\bigg(-\frac{r^2}{L^{-\beta_1} \cdot C(d,\gamma,K,N) B^8 (\log L)^{C(d,\gamma)} L^{-d}} \bigg)
\end{align*}
for some $\beta_1>0$ as long as
\begin{align*}
r\leq L^{-d/2} \cdot c(d,\gamma,N,K) (\log L)^{-C(d,\gamma)}.
\end{align*}
Choosing $r=\frac{1}{2}L^{-d/2} L^{-\beta_1/4}$ and using $L\geq C$ (note that upon changing the constants, our theorem is trivially true for $L\leq C$), we get
\begin{align}
\label{EstimateZ}
\mathbb{P}\bigg[|Z|\geq \frac{1}{2}L^{-d/2} L^{-\beta_1/4} \bigg]
\leq
C \exp\bigg(-\frac{L^{\beta_1/2}}{C(d,\gamma,K,N)  B^8 (\log L)^{C(d,\gamma)} } \bigg).
\end{align}
{\bf Step 5 (Conclusion)}.
\\
By \eqref{SumOfGroups}, we know that the law of the sum of the main groups
\begin{align*}
\sum_{i\in \ell \mathbb{Z}^d \cap [0,L)^d} \mathcal{G}_i
\end{align*}
is equal to the law of $Y+Z$, where $Y$ is a Gaussian random variable with covariance matrix $\tilde \Lambda=\sum_{i\in \ell \mathbb{Z}^d\cap [0,L)^d} \tilde \Lambda_i$ and where $Z$ satisfies the smallness estimate \eqref{EstimateZ}. By \eqref{TildeLambda}, \eqref{ChoiceLambdai}, and \eqref{EstimateRmSmall} as well as \eqref{EstimateRmBig} and \eqref{EstimateGroupTails} and the choice $\ell\sim L^{1/2}$, we see that
\begin{align*}
|\tilde \Lambda-\Var X|
&
\leq \bigg|\tilde \Lambda- 
\sum_{i\in \ell \mathbb{Z}^d \cap [0,L)^d} \Var \mathcal{G}_i\bigg|
+\bigg|\Var X-\Var
\sum_{i\in \ell \mathbb{Z}^d \cap [0,L)^d} \mathcal{G}_i\bigg|
\\&
\leq C L^{-d-1/8}+C B^2 L^{-d-1/4} (\log L)^C.
\end{align*}
The estimates \eqref{EstimateRmSmall} and \eqref{EstimateRmBig} imply using Lemma~\ref{CalculusStretchedExponential}b that
\begin{align*}
\mathbb{P}\Bigg[\Bigg|\sum_{m=0}^{\log_2 L+1} \mathcal{R}^m\Bigg| \geq r\Bigg]
\leq C \exp\bigg(-\bigg(\frac{r}{C(d,\gamma,N,K) B (\log L)^{C(d,\gamma)} \ell^{-1/2} L^{-d/2}}\bigg)^{\tilde \gamma}\bigg)
\end{align*}
and therefore for our choice $\ell:=L^{1/2}$
\begin{align*}
\mathbb{P}\Bigg[\Bigg|\sum_{m=0}^{\log_2 L+1} \mathcal{R}^m\Bigg| \geq \frac{1}{2} L^{-1/8}\cdot L^{-d/2} \Bigg]
\leq C \exp\bigg(-\frac{L^{\tilde \gamma/8}}{C(d,\gamma,N,K)B}\bigg).
\end{align*}
Combining this estimate with \eqref{EstimateZ} and
\begin{align*}
X-\sum_{m=0}^{1+\log_2 L} \mathcal{R}^m \stackrel{d}{=} Y+Z,
\end{align*}
we see that there exists $\beta=\beta(d,\gamma)>0$ and $\Lambda=\tilde \Lambda\in \mathbb{R}^{N\times N}_{\operatorname{sym}}$ with $\Lambda>0$ and
\begin{align*}
|\Lambda-\Var X|\leq C B^2 L^{-1/8} L^{-d}
\end{align*}
such that for any measurable $A\subset \mathbb{R}^N$ we have
\begin{align*}
\mathbb{P}\big[X\in A\big]
\leq \int_{\{x\in \mathbb{R}^N:\dist(x,A)\leq L^{-\beta} L^{-d/2}\}} \mathcal{N}_{\Lambda}(x) \,dx + C \exp\Big(-\frac{c}{B^8} L^{2\beta}\Big).
\end{align*}
\end{proof}

In the previous proofs, we have made use of the following elementary lemma.
\begin{lemma}
\label{MultilevelVariableStretchedExponentialBound}
Let $d\geq 1$ and $L\geq 2$. Consider a random field $a$ on $\mathbb{R}^d$ subject to the assumption of finite range of dependence (A) or an $L$-periodic random field subject to the assumption of finite range of dependence (A').
Let $X=X(a)$ be a random variable that is a sum of random variables with multilevel local dependence in the sense of Definition~\ref{ConditionRandomVariable}. Then for $\tilde \gamma:=\gamma/(\gamma+1)$ the concentration estimate
\begin{align*}
||X-\mathbb{E}[X]||_{\exp^{\tilde \gamma}} \leq C(d,\gamma,K) B |\log L|^{d/2} L^{-d/2}
\end{align*}
holds true.
\end{lemma}
\begin{proof}
Defining $p(m)$ to be the smallest integer with $2^{p(m)}\geq 2^{m+2} K |\log L|$ (but defining $p(m)=\log_2 L$ if this integer were larger than $\log_2 L$), we rewrite
\begin{align*}
X=\sum_{m=0}^{\log_2 L+1} \sum_{k\in 2^m \mathbb{Z}^d \cap [0,2^{p(m)})^d}
\sum_{j\in 2^{p(m)}\mathbb{Z}^d\cap [0,L)^d,j+k\in [0,L)^d} X_{j+k}^m.
\end{align*}
By Definition~\ref{ConditionRandomVariable}, the inner sum is now a sum of independent random variables. Applying Lemma~\ref{ConcentrationStretchedExponential} to this sum, we obtain
\begin{align*}
&||X-\mathbb{E}[X]||_{\exp^{\tilde \gamma}}
\\&
\leq \sum_{m=0}^{\log_2 L+1} \sum_{k\in 2^m \mathbb{Z}^d \cap [0,2^{p(m)})^d}
\bigg|\bigg|\sum_{j\in 2^{p(m)}\mathbb{Z}^d\cap [0,L)^d,j+k\in [0,L)^d} (X_{j+k}^m-\mathbb{E}[X_{j+k}^m])\bigg|\bigg|_{\exp^{\tilde \gamma}}
\\&
\leq \sum_{m=0}^{\log_2 L+1} \sum_{k\in 2^m \mathbb{Z}^d \cap [0,2^{p(m)})^d} C(\gamma) \bigg(\frac{L^d}{(2^{p(m)})^d}\bigg)^{1/2} \max_i ||X_{i}^m-\mathbb{E}[X_{i}^m]||_{\exp^{\gamma}}
\\&
\leq \sum_{m=0}^{\log_2 L+1} \frac{(2^{p(m)})^d}{(2^m)^d} \cdot C(d,\gamma) \bigg(\frac{L^d}{(2^{p(m)})^d}\bigg)^{1/2} B L^{-d}
\\&
\leq \sum_{m=0}^{\log_2 L+1} \frac{(2^m)^d (K\log L)^d}{(2^m)^d} \cdot C(d,\gamma) \bigg(\frac{L^d}{(2^m)^d (K\log L)^d}\bigg)^{1/2} B L^{-d}
\\&
\leq C(d,\gamma,K) B (\log L)^{d/2} L^{-d/2}.
\end{align*}
\end{proof}

%\section{Proof of the main result}
%
%\begin{proof}[Proof of Theorem~\ref{TheoremRandomField}]
%We choose a function $\eta$ with $0\leq \eta\leq 1$, $\supp \eta\subset B_d(0)$, $|\nabla^n \eta|\leq C(n)$ for all $n\in \mathbb{N}$, and $\sum_{k\in \mathbb{Z}^d} \eta(x-k)=1$ for all $x\in \mathbb{R}^d$.
%We next define $m_L:=\lfloor \log_2 L \rfloor$ and $\tilde L:=2^{m_L}$ and write
%\begin{align*}
%v = (v-v_{\tilde L}) + v_1 + \sum_{m=2}^{m_L} (v_{2^m}-v_{2^{m-1}}).
%\end{align*}
%Next, define
%\begin{align*}
%X_{0}^{m_L} := \int_{\mathbb{R}^d} (v-v_{\tilde L}) \psi \,dx,
%\end{align*}
%and
%\begin{align*}
%X_{k}^{0} := \int_{\mathbb{R}^d} v_1 \eta_k \psi \,dx
%\quad\quad\text{for any }k\in \mathbb{Z}^d
%\end{align*}
%and
%\begin{align*}
%X_{i}^{m} := \int_{\mathbb{R}^d} (v_{2^m}-v_{2^{m-1}}) \eta() \psi \,dx
%\end{align*}
%
%
%In a second step, we simply group the together, yielding .
%
%\end{proof}

\appendix

\section{Auxiliary results for Stein's method in the multivariate setting}

In this appendix, we provide the proof of the bounds on the solutions to the ``smoothed'' Stein's equation stated in Proposition~\ref{PropositionSolutionSteinEquation} and the estimate on the distance $\mathcal{D}^{\bar L}$ in terms of the smoothed distance $\mathcal{D}_{\varepsilon}^{\bar L}$ stated in Lemma~\ref{SmoothingEstimateLemma}.

\begin{proof}[Proof of Proposition~\ref{PropositionSolutionSteinEquation}]
{\bf Proof of a).}
Following the argument of \cite{Goetze} (see also \cite{ChenGoldsteinShao}), we consider the function
\begin{align}
\label{Definitionfepsilon}
f_\varepsilon(x)=\frac{1}{2} \int_{\varepsilon^2}^1 \Big(\int_{\mathbb{R}^N}\phi(\sqrt{1-s}x-\sqrt{s}z)\mathcal{N}_\Lambda(z) \,dz-\int_{\mathbb{R}^N} \mathcal{N}_\Lambda(z)\phi(z)\,dz\Big) \frac{1}{1-s} \,ds
\end{align}
which in the case of smooth $\phi$ with compactly supported derivative $\nabla \phi$ satisfies
\begin{align*}
&-(\nabla \cdot \Lambda \nabla f_\varepsilon)(x)
+(x\cdot \nabla f_\varepsilon)(x)
\\&
=-\frac{1}{2} \int_{\varepsilon^2}^1 \int_{\mathbb{R}^N} \Lambda: (\nabla^2 \phi)(\sqrt{1-s}x-\sqrt{s}z)\mathcal{N}_\Lambda(z) \,dz \,ds
\\&~~~
+ \frac{1}{2} \int_{\varepsilon^2}^1 \int_{\mathbb{R}^N} x\cdot (\nabla \phi)(\sqrt{1-s}x-\sqrt{s}z) \mathcal{N}_\Lambda(z) \frac{1}{\sqrt{1-s}} \,dz \,ds
\\&
=-\frac{1}{2} \int_{\varepsilon^2}^1 \int_{\mathbb{R}^N} \Lambda: (\nabla^2 \phi)(\sqrt{1-s}x-\sqrt{s}z)\mathcal{N}_\Lambda(z) \,dz \,ds
\\&~~~
- \frac{1}{2} \int_{\varepsilon^2}^1 \int_{\mathbb{R}^N} z\cdot (\nabla \phi)(\sqrt{1-s}x-\sqrt{s}z) \mathcal{N}_\Lambda(z) \frac{1}{\sqrt{s}} \,dz \,ds
\\&~~~
+\int_{\varepsilon^2}^1 -\frac{d}{ds} \int_{\mathbb{R}^N} \phi(\sqrt{1-s}x-\sqrt{s}z) \mathcal{N}_\Lambda(z) \,dz \,ds
\\&
=-\frac{1}{2} \int_{\varepsilon^2}^1 \int_{\mathbb{R}^N} \phi(\sqrt{1-s}x-\sqrt{s}z) \frac{1}{s} \Big( \big(\nabla \cdot (\Lambda \nabla \mathcal{N}_\Lambda)\big)(z)+\nabla \cdot (z \mathcal{N}_\Lambda(z)) \Big) \,dz \,ds
\\&~~~
+\int_{\mathbb{R}^N} \phi(\sqrt{1-\varepsilon^2}x-\varepsilon z) \mathcal{N}_\Lambda(z) \,dz-\int_{\mathbb{R}^N} \phi(z) \mathcal{N}_\Lambda(z) \,dz
\\&
=\int_{\mathbb{R}^N} \phi(\sqrt{1-\varepsilon^2}x-\varepsilon z) \mathcal{N}_\Lambda(z) \,dz
-\int_{\mathbb{R}^N} \phi(z) \mathcal{N}_\Lambda(z) \,dz.
\end{align*}
For general $\phi\in \Phi_\Lambda^{\bar L}$, this equation follows by approximation. Hence, \eqref{MollifiedSteinEquation} follows from the additional computation
\begin{align*}
&\int_{\mathbb{R}^N} \phi_\varepsilon(x) \mathcal{N}_\Lambda(x) \,dx
\\
&=\int_{\mathbb{R}^N} \int_{\mathbb{R}^N} \phi(\sqrt{1-\varepsilon^2}x-\varepsilon z) \mathcal{N}_\Lambda(x) \mathcal{N}_\Lambda(z) \,dz \,dx
\\&
\stackrel{\eqref{MultiplyGaussians}}{=}
\int_{\mathbb{R}^N} \int_{\mathbb{R}^N} \phi(\sqrt{1-\varepsilon^2}x-\varepsilon z) \mathcal{N}_\Lambda(\sqrt{1-\varepsilon^2}x-\varepsilon z) \mathcal{N}_\Lambda(\sqrt{1-\varepsilon^2}z+\varepsilon x) \,dz \,dx
\\&
=
\int_{\mathbb{R}^N} \int_{\mathbb{R}^N} \phi(\tilde x) \mathcal{N}_\Lambda(\tilde x) \mathcal{N}_\Lambda(\tilde z)  \,d\tilde z \,d\tilde x
\\&
=
\int_{\mathbb{R}^N} \phi(x) \mathcal{N}_\Lambda(x) \,dx
\end{align*}
where we have used the fact that $(\tilde x,\tilde z)=(\sqrt{1-\varepsilon^2}x-\varepsilon z,\sqrt{1-\varepsilon^2}z+\varepsilon x)$ is an orthogonal linear transformation. This finishes the proof of Proposition~\ref{PropositionSolutionSteinEquation}a.

{\bf Proof of b).}
In order to establish \eqref{BoundThirdDerivativeLinfty}, we derive from \eqref{Definitionfepsilon} the following representation for the third spatial derivative of $f_\varepsilon$:
\begin{align}
\label{ThirdDerivative}
\nabla^3 f_\varepsilon (x) = \frac{1}{2} \int_{\varepsilon^2}^1 \frac{\sqrt{1-s}}{s^{3/2}} \int_{\mathbb{R}^N} \phi(\sqrt{1-s}x-\sqrt{s}z) \nabla^3 \mathcal{N}_\Lambda (z) \,dz \,ds.
\end{align}
This entails for any $\tau\in (0,1]$
\begin{align*}
&|\nabla^3 f_\varepsilon(x)|
\\
&\leq
\frac{1}{2}\int_{\varepsilon^2}^1 \frac{\sqrt{1-s}}{s} \int_{\mathbb{R}^N} |\nabla \phi|(\sqrt{1-s}x-\sqrt{s}z) |\nabla^2 \mathcal{N}_\Lambda| (z) \,dz\,ds
\\&
\leq
\frac{1}{2} \int_{\varepsilon^2}^1 \frac{\sqrt{1-s}}{s} \int_{\mathbb{R}^N} |\nabla \phi|(\sqrt{1-s}x-\sqrt{s}z) \frac{|\nabla^2 \mathcal{N}_\Lambda(z)|}{\mathcal{N}_{(1+\tau)\Lambda}(z)} \mathcal{N}_{(1+\tau)\Lambda} (z) \,dz\,ds
\\&
\stackrel{\eqref{BoundGaussianSecondDerivative}}{\leq}
\frac{3}{2} (1+\tau)^{(N+2)/2} \tau^{-1} |\Lambda^{-1}|
\\&~~~~~~~~~~
\times
\int_{\varepsilon^2}^1 \frac{\sqrt{1-s}}{s} \int_{\mathbb{R}^N} |\nabla \phi|(\sqrt{1-s}x-\sqrt{s}z) \mathcal{N}_{(1+\tau)\Lambda} (z) \,dz\,ds.
\end{align*}
Combining this estimate with the bound (valid for any $s\in (0,1)$)
\begin{align}
\label{EstimateFirstDerivative}
&\int_{\mathbb{R}^N} |\nabla \phi|(\sqrt{1-s}x-\sqrt{s}z) \mathcal{N}_{(1+\tau)\Lambda} (z) \,dz
\\&
\nonumber
=
\int_{\mathbb{R}^N} |\nabla \phi|(\tilde z) \mathcal{N}_{(1+\tau)\Lambda} \Big(\frac{1}{\sqrt{s}}(\sqrt{1-s}x-\tilde z)\Big) \sqrt{s}^{-N} \,d\tilde z
\\&
\nonumber
\leq
\int_{\mathbb{R}^N} |\nabla \phi|(\tilde z) \mathcal{N}_{(1+\tau)\Lambda} \big(\sqrt{1-s} x-\tilde z\big)  s^{-N/2} \,d\tilde z
\\&
\nonumber
\stackrel{\eqref{ConvolutionGaussian}}{=}
\int_{\mathbb{R}^N} \int_{\mathbb{R}^N}|\nabla \phi|(\tilde z) \mathcal{N}_{\Lambda} (\tilde z - \sqrt{1-s} x + w) \mathcal{N}_{\tau\Lambda}(w)  s^{-N/2} \,d\tilde z \,dw
\\&
\nonumber
\stackrel{\eqref{ClassPhiDifferentialBound}}{\leq}
s^{-N/2} \int_{\mathbb{R}^N}  \mathcal{N}_{\tau\Lambda}(w) \,dw
\\&
\nonumber
= s^{-N/2},
\end{align}
we infer choosing $\tau:=2/(N+2)$
\begin{align*}
|\nabla^3 f_\varepsilon(x)|
&\leq \frac{3}{2} (1+\tau)^{(N+2)/2} \tau^{-1} |\Lambda^{-1}|
\int_{\varepsilon^2}^1 s^{-(N+2)/2} \,ds
\\&
\leq \frac{3}{2}\cdot e \cdot \frac{N+2}{2} \cdot |\Lambda^{-1}| \cdot \frac{2}{N} (\varepsilon^{-N}-1).
\end{align*}
This proves \eqref{BoundThirdDerivativeLinfty}.

{\bf Proof of c).}
We now turn to the proof of Proposition~\ref{PropositionSolutionSteinEquation}c. The bound \eqref{UpperBoundByH} is immediate from Lemma~\ref{PropertiesOfFunctionClass}c.

Computing the second spatial derivative of $f_\varepsilon$ as defined by \eqref{Definitionfepsilon}, we infer
\begin{align}
\label{SecondDerivative}
\nabla^2 f_\varepsilon (x) = \frac{1}{2} \int_{\varepsilon^2}^1 \frac{1}{s} \int_{\mathbb{R}^N} \phi(\sqrt{1-s}x-\sqrt{s}z) \nabla^2 \mathcal{N}_\Lambda (z) \,dz \,ds.
\end{align}
In order to derive the estimates \eqref{BoundMollifiedSecondDerivativeOscLimitDistribution} and \eqref{BoundMollifiedSecondDerivativeOscInFunctionClass}, we estimate for $r>0$ and $\tau:=2/(N+2)$
\begin{align*}
&\int_{\mathbb{R}^N} (\osc_{r} \nabla^2 f_\varepsilon) (x) \mathcal{N}_\Lambda(x-x_0) \,dx
\\&
\leq \frac{1}{2} \int_{\mathbb{R}^N} \int_{\varepsilon^2}^1 \frac{1}{s} \int_{\mathbb{R}^N} (\osc_{\sqrt{1-s}r} \phi)(\sqrt{1-s}x-\sqrt{s}z) |\nabla^2 \mathcal{N}_\Lambda| (z)  \,dz  \,ds ~ \mathcal{N}_{\Lambda}(x-x_0)\,dx
\\&
\stackrel{\eqref{BoundGaussianSecondDerivative}}{\leq}
\frac{1}{2} \int_{\varepsilon^2}^1 \frac{1}{s} \int_{\mathbb{R}^N} \int_{\mathbb{R}^N} (\osc_{\sqrt{1-s}r} \phi)(\sqrt{1-s}x-\sqrt{s}z) 
\\&~~~~~~~~~~~~~~~~~~~~~~~~~~~~~~
\times
3(1+\tau)^{(N+2)/2}\tau^{-1} |\Lambda^{-1}| \mathcal{N}_{(1+\tau)\Lambda} (z)
\mathcal{N}_{\Lambda}(x-x_0) \,dz \,dx \,ds
\\&
\leq
\frac{3e(N+2)}{4} |\Lambda^{-1}| \int_{\varepsilon^2}^1 \frac{1}{s} \int_{\mathbb{R}^N} \int_{\mathbb{R}^N} (\osc_{\sqrt{1-s}r} \phi)(\sqrt{1-s}x-\sqrt{s}z)
\\&~~~~~~~~~~~~~~~~~~~~~~~~~~~~~~~~~~~~~~~~~~~~~~~
\times
(\mathcal{N}_{\tau\Lambda} \ast \mathcal{N}_{\Lambda}) (z) \mathcal{N}_{\Lambda}(x-x_0) \,dz \,dx \,ds
\\&
=
\frac{3e(N+2)}{4} |\Lambda^{-1}| \int_{\varepsilon^2}^1 \frac{1}{s} \int_{\mathbb{R}^N} \int_{\mathbb{R}^N} (\mathcal{N}_{s\tau\Lambda} \ast \osc_{\sqrt{1-s}r} \phi)(\sqrt{1-s}(x+x_0)-\sqrt{s}z)
\\&~~~~~~~~~~~~~~~~~~~~~~~~~~~~~~~~~~~~~~~~~~~~~~~
\times
\mathcal{N}_{\Lambda} (z) \mathcal{N}_{\Lambda}(x) \,dz \,dx \,ds.
\end{align*}
Invoking the change of variables $(\tilde x,\tilde z):=(\sqrt{1-s}x-\sqrt{s}z,\sqrt{1-s}z+\sqrt{s}x)$ (note that this is a linear orthogonal transformation) as well as the multiplication property \eqref{MultiplyGaussians}, we deduce
\begin{align*}
&\int_{\mathbb{R}^N}  (\osc_{r} \nabla^2 f_\varepsilon) (x) \mathcal{N}_\Lambda(x-x_0) \,dx
\\&
\leq
\frac{3e(N+2)}{4} |\Lambda^{-1}| \int_{\varepsilon^2}^1 \frac{1}{s} \int_{\mathbb{R}^N} \int_{\mathbb{R}^N} (\mathcal{N}_{s\tau\Lambda} \ast \osc_{\sqrt{1-s}r} \phi)(\tilde x+\sqrt{1-s}x_0)
\\&~~~~~~~~~~~~~~~~~~~~~~~~~~~~~~~~~~~~~~~~~~~~~~~~~~\times
\mathcal{N}_{\Lambda} (\tilde x) \mathcal{N}_{\Lambda}(\tilde z) \,d\tilde x \,d\tilde z \,ds
\\&
=
\frac{3e(N+2)}{4} |\Lambda^{-1}| \int_{\varepsilon^2}^1 \frac{1}{s} \int_{\mathbb{R}^N} \int_{\mathbb{R}^N} \int_{\mathbb{R}^N} (\osc_{\sqrt{1-s}r} \phi)(\tilde x+\sqrt{1-s}x_0-w)
\\&~~~~~~~~~~~~~~~~~~~~~~~~~~~~~~~~~~~~~~~~~~~~~~~~~~~~~\times
\mathcal{N}_{\Lambda} (\tilde x)
\mathcal{N}_{s\tau\Lambda}(w) \mathcal{N}_{\Lambda}(\tilde z) \,dw \,d\tilde x \,d\tilde z \,ds
\\&
\stackrel{\eqref{DefinitionClassPhi}}{\leq}
\frac{9(N+2)}{4} |\Lambda^{-1}| \int_{\varepsilon^2}^1 \frac{1}{s} \int_{\mathbb{R}^N}\int_{\mathbb{R}^N} \sqrt{1-s}r \mathcal{N}_{s\tau\Lambda}(w) \mathcal{N}_{\Lambda}(\tilde z) \,dw \,d\tilde z \,ds
\end{align*}
and hence
\begin{align}
\label{EstimateOscSecondDerivative}
\int_{\mathbb{R}^N}  (\osc_{r} \nabla^2 f_\varepsilon) (x) \mathcal{N}_\Lambda(x-x_0) \,dx
\leq
16N |\Lambda^{-1}| |\log \varepsilon| r
\end{align}
for any $r>0$ and any $x_0\in \mathbb{R}^N$. As a consequence, we get
\begin{align*}
&\int_{\mathbb{R}^N} 2(\mathcal{N}_{\delta^2 \Id_N} \ast \osc_{K\delta} \nabla^2 f_\varepsilon) (x) \mathcal{N}_\Lambda(x) \,dx
\\&
=2\int_{\mathbb{R}^N} \int_{\mathbb{R}^N} (\osc_{K\delta} \nabla^2 f_\varepsilon) (x-w) \mathcal{N}_{\delta^2 \Id_N}(w) \mathcal{N}_\Lambda(x) \,dx\,dw
\\&
\leq 2\int_{\mathbb{R}^N} 16N |\Lambda^{-1}| |\log \varepsilon| K \delta \mathcal{N}_{\delta^2 \Id_N}(w) \,dx,
\end{align*}
which (by our choice of $K$) immediately proves \eqref{BoundMollifiedSecondDerivativeOscLimitDistribution}.

Furthermore, \eqref{EstimateOscSecondDerivative} entails that
\begin{align*}
h(x):=\frac{1}{32 N |\Lambda^{-1}| |\log \varepsilon| K} (\osc_{K\delta} \nabla^2 f_\varepsilon) (x)
\end{align*}
satisfies the estimates
\begin{align*}
\int_{\mathbb{R}^N}  |h(x)| \mathcal{N}_\Lambda(x-x_0) \,dx
\leq \delta
\end{align*}
for any $x_0\in \mathbb{R}^N$ and (by the inequality $(\osc_r \osc_{K\delta} f)(x) \leq \osc_{r+K\delta} f(x)$)
\begin{align*}
&\int_{\mathbb{R}^N}  \osc_r h(x) \mathcal{N}_\Lambda(x-x_0) \,dx
\\
&\leq
\frac{1}{32 N |\Lambda^{-1}| |\log \varepsilon| K}
\int_{\mathbb{R}^N}  (\osc_{r+K\delta} \nabla^2 f_\varepsilon)(x) \mathcal{N}_\Lambda(x-x_0) \,dx
\\&
\leq \frac{1}{2 K} (r+K\delta)
\\&
\leq r
\end{align*}
for any $x_0\in \mathbb{R}^N$ and any $r\geq \delta$.
In conclusion, Lemma~\ref{PropertiesOfFunctionClass}d implies
\begin{align*}
\frac{1}{40N} \times
\frac{1}{32 N |\Lambda^{-1}| |\log \varepsilon| K} (\mathcal{N}_{\delta^2\Id_N} \ast \osc_{K\delta} \nabla^2 f_\varepsilon) \in \Phi_\Lambda^{\tilde L}
\end{align*}
for any $\tilde L\geq 4^N (|\Lambda^{1/2}|^N \delta^{-N} +1)$, which proves \eqref{BoundMollifiedSecondDerivativeOscInFunctionClass}.

{\bf Proof of d), Part 1.}
We now establish Proposition~\ref{PropositionSolutionSteinEquation}d. The estimate \eqref{UpperBoundByH2} is immediate by Lemma~\ref{PropertiesOfFunctionClass}c and the inequality $\sup_{|y|\leq \delta} |\nabla^3 f_\varepsilon(x+y)|\leq |\nabla^3 f_\varepsilon(x)|+\osc_\delta \nabla^3 f_\varepsilon(x)$.

Before proceeding, let us first prove the following auxiliary result: The third derivative of $f_\varepsilon$ satisfies the estimate
\begin{align}
\label{EstimateThirdDerivativeLimitDistribution}
\int_{\mathbb{R}^N}  |\nabla^3 f_\varepsilon| (x) \mathcal{N}_\Lambda(x) \,dx
\leq
16 N |\Lambda^{-1}| |\log \varepsilon|,
\end{align}
and has the property
\begin{align}
\label{BoundThirdDerivativeInFunctionClass}
\frac{\varepsilon}{10 N^{3/2} |\Lambda^{-1/2}|^3} |\nabla^3 f_\varepsilon| \in \Phi_\Lambda^{\tilde L}
\end{align}
for any $\tilde L\geq 10 \varepsilon^{-N}$.

The estimate \eqref{EstimateThirdDerivativeLimitDistribution}
is simply a consequence of \eqref{EstimateOscSecondDerivative} in the limit $r\rightarrow 0$.
To show \eqref{BoundThirdDerivativeInFunctionClass} we first establish a uniform bound for $\nabla^4 f_\varepsilon$. This is done in an analogous way to the proof of assertion b) of our proposition: Using \eqref{BoundGaussianThirdDerivative} instead of \eqref{BoundGaussianSecondDerivative}, we deduce
\begin{align*}
&|\nabla^4 f_\varepsilon(x)|
\\
&\leq
\frac{1}{2}\int_{\varepsilon^2}^1 \frac{1-s}{s^{3/2}} \int_{\mathbb{R}^N} |\nabla \phi|(\sqrt{1-s}x-\sqrt{s}z) |\nabla^3 \mathcal{N}_\Lambda| (z) \,dz\,ds
\\&
\stackrel{\eqref{BoundGaussianThirdDerivative}}{\leq}
\frac{5}{2} (1+\tau)^{(N+3)/2} \tau^{-3/2} |\Lambda^{-1/2}|^3
\\&~~~~~~~~~~
\times
\int_{\varepsilon^2}^1 \frac{1}{s^{3/2}} \int_{\mathbb{R}^N} |\nabla \phi|(\sqrt{1-s}x-\sqrt{s}z) \mathcal{N}_{(1+\tau)\Lambda} (z) \,dz\,ds
\end{align*}
and choosing $\tau:=2/(N+3)$ and applying \eqref{EstimateFirstDerivative}, we get
\begin{align}
\nonumber
|\nabla^4 f_\varepsilon(x)|
&\leq \frac{5}{2} \cdot e \cdot \frac{(N+3)^{3/2}}{2^{3/2}} |\Lambda^{-1/2}|^3 \int_{\varepsilon^2}^1 \frac{1}{s^{3/2}} \cdot s^{-N/2} \,ds
\\&
\nonumber
\leq \frac{5}{2} \cdot e \cdot \frac{(N+3)^{3/2}}{2^{3/2}} |\Lambda^{-1/2}|^3\cdot \frac{2}{N+1} \cdot \frac{1}{(\varepsilon^2)^{(N+1)/2}}
\\&
\label{LipschitzBoundThirdDerivative}
\leq 20 \sqrt{N+3} |\Lambda^{-1/2}|^3 \varepsilon^{-N-1}.
\end{align}
In order to show \eqref{BoundThirdDerivativeInFunctionClass}, for any $r>0$ we
estimate starting with \eqref{ThirdDerivative} and choosing $\tau:=2/(N+3)$
%derive from \eqref{Definitionfepsilon} the following representation for the third spatial derivative of $f_\varepsilon$:
%Given $r>0$ and choosing $\tau:=2/(N+3)$, we estimate
\begin{align*}
&\int_{\mathbb{R}^N} (\osc_{r} \nabla^3 f_\varepsilon) (x) \mathcal{N}_\Lambda(x-x_0) \,dx
\\&
\leq \frac{1}{2} \int_{\varepsilon^2}^1 \frac{\sqrt{1-s}}{s^{3/2}} \int_{\mathbb{R}^N} \int_{\mathbb{R}^N} (\osc_{\sqrt{1-s}r} \phi)(\sqrt{1-s}x-\sqrt{s}z) |\nabla^3 \mathcal{N}_\Lambda| (z) \mathcal{N}_{\Lambda}(x-x_0) \,dz \,dx \,ds
\\&
\stackrel{\eqref{BoundGaussianThirdDerivative}}{\leq}
\frac{1}{2} \int_{\varepsilon^2}^1 \frac{1}{s^{3/2}} \int_{\mathbb{R}^N} \int_{\mathbb{R}^N} (\osc_{\sqrt{1-s}r} \phi)(\sqrt{1-s}x-\sqrt{s}z) 
\\&~~~~~~~~~~~~~~~~~~~~~~~~~~
\times
5 (1+\tau)^{(N+3)/2} \tau^{-3/2} |\Lambda^{-1/2}|^3 \mathcal{N}_{(1+\tau)\Lambda} (z)
\mathcal{N}_{\Lambda}(x-x_0) \,dz \,dx \,ds
\\&
\leq
\frac{5e}{2} \cdot \frac{(N+3)^{3/2}}{2^{3/2}} |\Lambda^{-1/2}|^3 \int_{\varepsilon^2}^1 \frac{1}{s^{3/2}} \int_{\mathbb{R}^N} \int_{\mathbb{R}^N} (\osc_{\sqrt{1-s}r} \phi)(\sqrt{1-s}x-\sqrt{s}z)
\\&~~~~~~~~~~~~~~~~~~~~~~~~~~~~~~~~~~~~~~~~~~~~~~~~~~~~~~~~~~~~
\times
(\mathcal{N}_{\tau\Lambda} \ast \mathcal{N}_{\Lambda}) (z) \mathcal{N}_{\Lambda}(x-x_0) \,dz \,dx \,ds
\\&
=
\frac{5e(N+3)^{3/2}}{2 \cdot \sqrt{2}^3} |\Lambda^{-1/2}|^3 \int_{\varepsilon^2}^1 \frac{1}{s^{3/2}} \int_{\mathbb{R}^N} \int_{\mathbb{R}^N} (\mathcal{N}_{s\tau\Lambda} \ast \osc_{\sqrt{1-s}r} \phi)(\sqrt{1-s}(x+x_0)-\sqrt{s}z)
\\&~~~~~~~~~~~~~~~~~~~~~~~~~~~~~~~~~~~~~~~~~~~~~~~~~~~~~~~~~~~
\times
\mathcal{N}_{\Lambda} (z) \mathcal{N}_{\Lambda}(x) \,dz \,dx \,ds.
\end{align*}
Invoking the change of variables $(\tilde x,\tilde z):=(\sqrt{1-s}x-\sqrt{s}z,\sqrt{1-s}z+\sqrt{s}x)$ (note that this is a linear orthogonal transformation) as well as the multiplication property \eqref{MultiplyGaussians}, we deduce
\begin{align*}
&\int_{\mathbb{R}^N}  (\osc_{r} \nabla^3 f_\varepsilon) (x) \mathcal{N}_\Lambda(x-x_0) \,dx
\\&
\leq
\frac{5(N+3)^{3/2}}{2} |\Lambda^{-1/2}|^3 \int_{\varepsilon^2}^1 \frac{1}{s^{3/2}} \int_{\mathbb{R}^N} \int_{\mathbb{R}^N} (\mathcal{N}_{s\tau\Lambda} \ast \osc_{\sqrt{1-s}r} \phi)(\tilde x+\sqrt{1-s}x_0)
\\&~~~~~~~~~~~~~~~~~~~~~~~~~~~~~~~~~~~~~~~~~~~~~~~~~~~~~~~~~~~\times
\mathcal{N}_{\Lambda} (\tilde x) \mathcal{N}_{\Lambda}(\tilde z) \,d\tilde x \,d\tilde z \,ds
\\&
=
\frac{5(N+3)^{3/2}}{2} |\Lambda^{-1/2}|^3 \int_{\varepsilon^2}^1 \frac{1}{s^{3/2}} \int_{\mathbb{R}^N} \int_{\mathbb{R}^N} \int_{\mathbb{R}^N} (\osc_{\sqrt{1-s}r} \phi)(\tilde x+\sqrt{1-s}x_0-w)
\\&~~~~~~~~~~~~~~~~~~~~~~~~~~~~~~~~~~~~~~~~~~~~~~~~~~~~~~~~~~\times
\mathcal{N}_{\Lambda} (\tilde x)
\mathcal{N}_{s\tau\Lambda}(w) \mathcal{N}_{\Lambda}(\tilde z) \,dw \,d\tilde x \,d\tilde z \,ds
\\&
\stackrel{\eqref{DefinitionClassPhi}}{\leq}
\frac{5(N+3)^{3/2}}{2} |\Lambda^{-1/2}|^3 \int_{\varepsilon^2}^1 \frac{1}{s^{3/2}} \int_{\mathbb{R}^N}\int_{\mathbb{R}^N} \sqrt{1-s}r \mathcal{N}_{s\tau\Lambda}(w) \mathcal{N}_{\Lambda}(\tilde z) \,dw \,d\tilde z \,ds
\end{align*}
and hence
\begin{align}
\label{EstimateOscThirdDerivative}
\int_{\mathbb{R}^N}  (\osc_{r} \nabla^3 f_\varepsilon) (x) \mathcal{N}_\Lambda(x-x_0) \,dx
\leq
10 N^{3/2} |\Lambda^{-1/2}|^3 \frac{r}{\varepsilon}
\end{align}
for any $r>0$ and any $x_0\in \mathbb{R}^N$. Combining this estimate with \eqref{LipschitzBoundThirdDerivative}, we infer \eqref{BoundThirdDerivativeInFunctionClass}. 

{\bf Proof of d), Part 2.} We now turn to the last part of the proof of our proposition.
As a consequence of \eqref{EstimateOscThirdDerivative}, we get
\begin{align*}
&\int_{\mathbb{R}^N} 2(\mathcal{N}_{\delta^2 \Id_N} \ast \osc_{K\delta} \nabla^3 f_\varepsilon) (x) \mathcal{N}_\Lambda(x) \,dx
\\&
=2\int_{\mathbb{R}^N} \int_{\mathbb{R}^N} (\osc_{K\delta} \nabla^3 f_\varepsilon) (x-w) \mathcal{N}_{\delta^2 \Id_N}(w) \mathcal{N}_\Lambda(x) \,dx\,dw
\\&
\leq 2\int_{\mathbb{R}^N} 10 N^{3/2} |\Lambda^{-1/2}|^3 \frac{K\delta}{\varepsilon} \mathcal{N}_{\delta^2 \Id_N}(w) \,dx,
\end{align*}
which gives
\begin{align*}
&\int_{\mathbb{R}^N} 2(\mathcal{N}_{\delta^2 \Id_N} \ast \osc_{K\delta} \nabla^3 f_\varepsilon) (x) \mathcal{N}_{\Lambda}(x) \,dx
\leq 20 N^{3/2} K |\Lambda^{-1/2}|^3 \, \frac{\delta}{\varepsilon}.
\end{align*}
Using $K=2\sqrt{N}+1$ and combining this estimate with \eqref{EstimateThirdDerivativeLimitDistribution}, we infer \eqref{BoundMollifiedThirdDerivativeLimitDistribution}.

Furthermore, we know by \eqref{EstimateOscThirdDerivative} that
\begin{align*}
h(x):=\frac{\varepsilon}{20 N^{3/2} |\Lambda^{-1/2}|^3 K} \osc_{K\delta} \nabla^3 f_\varepsilon(x)
\end{align*}
satisfies the estimate
\begin{align*}
\int_{\mathbb{R}^N}  h(x) \mathcal{N}_\Lambda(x-x_0) \,dx
\leq \delta
\end{align*}
for any $x_0\in \mathbb{R}^N$; by \eqref{EstimateOscThirdDerivative}, it also satisfies the bound
\begin{align*}
&\int_{\mathbb{R}^N}  \osc_r h(x) \mathcal{N}_\Lambda(x-x_0) \,dx
\\
&\leq
\frac{\varepsilon}{20 N^{3/2} |\Lambda^{-1/2}|^3 K}
\int_{\mathbb{R}^N}  (\osc_{r+K\delta} \nabla^3 f_\varepsilon)(x) \mathcal{N}_\Lambda(x-x_0) \,dx
\\&
\leq \frac{1}{2 K} (r+K\delta)
\\&
\leq r
\end{align*}
for any $x_0\in \mathbb{R}^N$ and any $r\geq \delta$.
In conclusion, Lemma~\ref{PropertiesOfFunctionClass}d implies
\begin{align*}
\frac{1}{40N} \times
\frac{\varepsilon}{20 N^{3/2} |\Lambda^{-1/2}|^3 K} (\mathcal{N}_{\delta^2\Id_N} \ast \osc_{K\delta} \nabla^3 f_\varepsilon) \in \Phi_\Lambda^{\tilde L}
\end{align*}
for any $\tilde L\geq 4^N (|\Lambda^{1/2}|^N \delta^{-N} +1)$. In conjunction with \eqref{BoundThirdDerivativeInFunctionClass} and $K=2\sqrt{N}+1$, we infer \eqref{BoundMollifiedThirdDerivativeInFunctionClass}.
\end{proof}

\begin{proof}[Proof of Lemma~\ref{SmoothingEstimateLemma}]
 Without loss of generality (and in order to simplify notation) we may assume
\begin{align*}
|\Lambda|= 1,
\end{align*}
that is the maximal eigenvalue of $\Lambda$ is equal to $1$; the general case then follows upon rescaling $\hat X:=\frac{1}{|\Lambda^{1/2}|}X$, $\hat \phi(x):=\frac{1}{|\Lambda^{1/2}|}\phi(|\Lambda^{1/2}|x)$, $\hat \Lambda:=\Lambda/|\Lambda|$ (note that then $\hat \phi_\varepsilon(x)=\frac{1}{|\Lambda^{1/2}|}\phi_\varepsilon(|\Lambda^{1/2}| x)$).

In order to establish \eqref{SmoothingEstimate}, by the very definition \eqref{DefinitionD} of $\mathcal{D}^{\bar L}$ and upon replacing $\varepsilon$ by $\frac{1}{2}\varepsilon$ it suffices to prove
\begin{align}
\label{SmoothingEstimateSimplified}
\bigg|\mathbb{E}[\phi(X)]-\int_{\mathbb{R}^N} \phi(x)\mathcal{N}_\Lambda(x) \,dx
\bigg|
\leq
10\sqrt{N} \varepsilon
+10^3 N^{3/2} \mathcal{D}_{\varepsilon/2}^{\bar L}(X,\mathcal{N}_\Lambda)
\end{align}
for all $\phi\in \Phi_\Lambda^{\bar L}$ and all $0<\varepsilon\leq 1$.

Introducing the function
\begin{align*}
\tilde \phi_\varepsilon(x):=\phi_\varepsilon(x/\sqrt{1-\varepsilon^2}),
\end{align*}
the error stemming from the smoothing of $\phi$ may be estimated as
\begin{align*}
|\tilde \phi_\varepsilon(x) - \phi(x)|&=|\phi_\varepsilon(x/\sqrt{1-\varepsilon^2}) - \phi(x)|
\\&
\leq
\bigg|\int_{\mathbb{R}^N} \big( \phi(x-\varepsilon z) - \phi(x) \big) \mathcal{N}_\Lambda(z) \,dz\bigg|
\\&
\leq
\int_{\mathbb{R}^N} \osc_{\varepsilon|z|}\phi(x) \mathcal{N}_{\Lambda}(z) \,dz.
%\\&
%\leq
%\int_{\mathbb{R}^N} \osc_{\varepsilon |\Lambda^{1/2}| |z|}\phi(x) \mathcal{N}_{\Id_N}(z) \,dz
%\\&
%\leq
%\int_0^\infty \frac{\mathcal{H}^{N-1}(S^{N-1})}{(2\pi)^{N/2}} s^{N-1} \exp\Big(-\frac{1}{2}s^2\Big) (\osc_{\varepsilon |\Lambda^{1/2}| s} \phi)(x) \,ds.
\end{align*}
As a consequence, we deduce
\begin{align*}
&\Bigg|\mathbb{E}[\phi(X)]-\int_{\mathbb{R}^N} \phi(x)\mathcal{N}_\Lambda(x) \,dx\Bigg|
\\&
\leq
\bigg|\mathbb{E}[\tilde \phi_\varepsilon(X)]
-\int_{\mathbb{R}^N} \tilde \phi_\varepsilon(x)\mathcal{N}_\Lambda(x) \,dx\bigg|
+\mathbb{E}\big[\big|\tilde \phi_\varepsilon-\phi\big|(X)\big]
\\&~~~~
+\bigg|\int_{\mathbb{R}^N} (\tilde \phi_\varepsilon-\phi)(x)\mathcal{N}_\Lambda(x) \,dx\bigg|
\\&
\leq
\bigg|\mathbb{E}[\tilde \phi_\varepsilon(X)]
-\int_{\mathbb{R}^N} \tilde \phi_\varepsilon(x)\mathcal{N}_\Lambda(x) \,dx\bigg|
+\int_{\mathbb{R}^N} \mathbb{E}\big[\osc_{\varepsilon |z|}\phi(X) \big] \mathcal{N}_{\Lambda}(z) \,dz
\\&~~~~
+\int_{\mathbb{R}^N}\int_{\mathbb{R}^N} \osc_{\varepsilon |z|} \phi(x) \mathcal{N}_{\Lambda}(z) \,dz ~ \mathcal{N}_{\Lambda}(x)  \,dx.
\end{align*}
For $K := 2\sqrt{N}$, by Lemma~\ref{PropertiesOfFunctionClass}c we infer
\begin{align*}
&\Bigg|\mathbb{E}[\phi(X)]-\int_{\mathbb{R}^N} \phi(x)\mathcal{N}_\Lambda(x) \,dx\Bigg|
\\&
\leq
\bigg|\mathbb{E}[\tilde \phi_\varepsilon(X)]
-\int_{\mathbb{R}^N} \tilde \phi_\varepsilon(x)\mathcal{N}_\Lambda(x) \,dx\bigg|
\\&~~~~
+2\int_{\mathbb{R}^N} \mathbb{E}\big[ \big(\mathcal{N}_{2\varepsilon^2 \Id_N} \ast \osc_{(|z|+K) \varepsilon}\phi\big)(X) \big] \mathcal{N}_{\Lambda}(z) \,dz
\\&~~~~
+\int_{\mathbb{R}^N}\int_{\mathbb{R}^N} \osc_{\varepsilon |z|} \phi(x) \mathcal{N}_{\Lambda}(z) \,dz ~ \mathcal{N}_{\Lambda}(x)  \,dx.
\end{align*}
Introducing the definition
\begin{align}
\label{Definitionpsiz}
\psi^z(x):=\big(\mathcal{N}_{\varepsilon^2 (2\Id_N-\Lambda)} \ast \osc_{(|z|+K) \varepsilon}\phi \big)(x),
\end{align}
we see that we may rewrite (recall that $\mathcal{N}_{\varepsilon^2\Lambda} \ast \mathcal{N}_{\varepsilon^2(2\Id_N-\Lambda)}=\mathcal{N}_{2\varepsilon^2\Id_N}$)
\begin{align*}
(\tilde \psi^z)_\varepsilon(x):&= (\psi^z)_\varepsilon(x/\sqrt{1-\varepsilon^2}) := \int_{\mathbb{R}^N} \psi^z(x-\varepsilon w) \mathcal{N}_{\Lambda}(w) \,dw
\\&
= \int_{\mathbb{R}^N} \psi^z(x-w) \mathcal{N}_{\varepsilon^2 \Lambda}(w) \,dw
= (\mathcal{N}_{\varepsilon^2 \Lambda} \ast \psi^z)(x)
\\&
=
\big(\mathcal{N}_{2\varepsilon^2 \Id_N} \ast \osc_{(|z|+K) \varepsilon}\phi \big)(x).
\end{align*}
As a consequence, we deduce the bound 
\begin{align*}
&\Bigg|\mathbb{E}[\phi(X)]-\int_{\mathbb{R}^N} \phi(x)\mathcal{N}_\Lambda(x) \,dx\Bigg|
\\&
\leq
\bigg|\mathbb{E}[\tilde \phi_\varepsilon(X)]
-\int_{\mathbb{R}^N} \tilde \phi_\varepsilon(x)\mathcal{N}_\Lambda(x) \,dx\bigg|
\\&~~~~
+2\int_{\mathbb{R}^N}
\bigg(
\mathbb{E}\big[ (\tilde \psi^z)_\varepsilon (X) \big]
-\int_{\mathbb{R}^N} (\tilde \psi^z)_\varepsilon(x) \mathcal{N}_{\Lambda}(x) \,dx
\bigg) \mathcal{N}_{\Lambda}(z) \,dz
\\&~~~~
+2\int_{\mathbb{R}^N}\int_{\mathbb{R}^N} \big(\mathcal{N}_{2\varepsilon^2 \Id_N} \ast \osc_{(|z|+K) \varepsilon}\phi \big) (x) \mathcal{N}_{\Lambda}(x) \,dx
~ \mathcal{N}_{\Lambda}(z) \,dz
\\&~~~~
+\int_{\mathbb{R}^N}\int_{\mathbb{R}^N} \osc_{\varepsilon |z|} \phi(x) \mathcal{N}_{\Lambda}(z) \,dz ~ \mathcal{N}_{\Lambda}(x)  \,dx
\end{align*}
and therefore
\begin{align}
\nonumber
&\Bigg|\mathbb{E}[\phi(X)]-\int_{\mathbb{R}^N} \phi(x)\mathcal{N}_\Lambda(x) \,dx\Bigg|
\\&
\nonumber
\leq
\bigg|\mathbb{E}[\tilde \phi_\varepsilon(X)]
-\int_{\mathbb{R}^N} \tilde \phi_\varepsilon(x)\mathcal{N}_\Lambda(x) \,dx\bigg|
\\&~~~~
\nonumber
+2\int_{\mathbb{R}^N}
\bigg(
\mathbb{E}\big[ (\tilde \psi^z)_\varepsilon (X) \big]
-\int_{\mathbb{R}^N} (\tilde \psi^z)_\varepsilon(x) \mathcal{N}_{\Lambda}(x) \,dx
\bigg) \mathcal{N}_{\Lambda}(z) \,dz
\\&~~~~
\nonumber
+2\int_{\mathbb{R}^N}\int_{\mathbb{R}^N}\int_{\mathbb{R}^N} \osc_{(|z|+K) \varepsilon}\phi (x-w) \mathcal{N}_{2\varepsilon^2 \Id_N}(w) \mathcal{N}_{\Lambda}(x) \,dx \,dw
~ \mathcal{N}_{\Lambda}(z) \,dz
\\&~~~~
\nonumber
+\int_{\mathbb{R}^N}\int_{\mathbb{R}^N} \osc_{\varepsilon |z|} \phi(x) \mathcal{N}_{\Lambda}(z) \,dz ~ \mathcal{N}_{\Lambda}(x)  \,dx
\\&
\nonumber
\stackrel{\eqref{DefinitionClassPhi}}{\leq}
\bigg|\mathbb{E}[\tilde \phi_\varepsilon(X)]
-\int_{\mathbb{R}^N} \tilde \phi_\varepsilon(x)\mathcal{N}_\Lambda(x) \,dx\bigg|
\\&~~~~
\nonumber
+2\int_{\mathbb{R}^N}
\bigg(
\mathbb{E}\big[ (\tilde \psi^z)_\varepsilon (X) \big]
-\int_{\mathbb{R}^N} (\tilde \psi^z)_\varepsilon(x) \mathcal{N}_{\Lambda}(x) \,dx
\bigg) \mathcal{N}_{\Lambda}(z) \,dz
\\&~~~~
\nonumber
+2\int_{\mathbb{R}^N}\int_{\mathbb{R}^N} (|z|+K) \varepsilon \mathcal{N}_{2\varepsilon^2 \Id_N}(w) \,dw
~ \mathcal{N}_{\Lambda}(z) \,dz
\\&~~~~
\nonumber
+\int_{\mathbb{R}^N} \varepsilon |z| \mathcal{N}_{\Lambda}(z) \,dz
\\&
\label{EstimateMollifiedIntermediateStep}
\leq
\bigg|\mathbb{E}[\tilde \phi_\varepsilon(X)]
-\int_{\mathbb{R}^N} \tilde \phi_\varepsilon(x)\mathcal{N}_\Lambda(x) \,dx\bigg|
\\&~~~~\nonumber
+2\int_{\mathbb{R}^N}
\bigg(
\mathbb{E}\big[ (\tilde \psi^z)_\varepsilon (X) \big]
-\int_{\mathbb{R}^N} (\tilde \psi^z)_\varepsilon(x) \mathcal{N}_{\Lambda}(x) \,dx
\bigg) \mathcal{N}_{\Lambda}(z) \,dz
\\&~~~~\nonumber
+\big(2\sqrt{N}+2K\big)\varepsilon
\\&~~~~\nonumber
+ \sqrt{N} \varepsilon,
\end{align}
where in the last step we have used H\"older's inequality and the moment estimate $\int_{\mathbb{R}^N} |z|^2 \mathcal{N}_\Lambda(z) \,dz\leq \operatorname{tr} \Lambda \leq N$.

Given $\psi\in \Phi_\Lambda^{\bar L}$, introducing the notation
\begin{align*}
\hat\psi(x):=\psi(x/\sqrt{1-\varepsilon^2/4})
\end{align*}
and the notation
\begin{align*}
\tilde \psi_\varepsilon(x)
&:=\int_{\mathbb{R}^N} \psi(x-\varepsilon z) \mathcal{N}_\Lambda(z) \,dz =\psi_\varepsilon(x/\sqrt{1-\varepsilon^2})
\end{align*}
we may write
\begin{align}
\nonumber
\tilde \psi_\varepsilon(x)
&=\int_{\mathbb{R}^N} \psi(x-\varepsilon z) \mathcal{N}_\Lambda(z) \,dz
=\int_{\mathbb{R}^N} \hat \psi(\sqrt{1-\varepsilon^2/4}\, x-\sqrt{1-\varepsilon^2/4} \,\varepsilon z) \mathcal{N}_\Lambda(z) \,dz
\\&
\nonumber
=\int_{\mathbb{R}^N} \hat \psi(\sqrt{1-\varepsilon^2/4}\, x-z) \mathcal{N}_{(1-\varepsilon^2/4)\varepsilon^2\Lambda}(z) \,dz
\\&
\nonumber
=\int_{\mathbb{R}^N} (\mathcal{N}_{(3/4-\varepsilon^2/4) \varepsilon^2 \Lambda} \ast \hat \psi)(\sqrt{1-\varepsilon^2/4}\, x-z) \mathcal{N}_{\frac{1}{4}\varepsilon^2\Lambda}(z) \,dz
\\&
\label{RewriteMollification}
=\theta_{\varepsilon/2}(x)
\end{align}
for $\theta_{\varepsilon/2}(x):=\int_{\mathbb{R}^N} \theta(\sqrt{1-\varepsilon^2/4}\, x-\frac{1}{2}\varepsilon z) \mathcal{N}_{\Lambda}(z) \,dz$ with
\begin{align*}
\theta := \mathcal{N}_{(3/4-\varepsilon^2/4)\varepsilon^2 \Lambda} \ast \hat \psi.
\end{align*}
Note that by Lemma~\ref{PropertiesOfFunctionClass} b) and a), from $\psi\in \Phi_\Lambda^{\bar L}$ it follows that $\frac{1}{2} \hat \psi \in \Phi_\Lambda^{\bar L}$ and therefore also $\frac{1}{2} \theta \in \Phi_\Lambda^{\bar L}$. Applying these considerations for $\psi:=\phi$, we deduce from estimate \eqref{EstimateMollifiedIntermediateStep}
\begin{align}
\label{EstimateMollifiedDistanceSecondIntermediate}
&\Bigg|\mathbb{E}[\phi(X)]-\int_{\mathbb{R}^N} \phi(x)\mathcal{N}_\Lambda(x) \,dx\Bigg|
\\&
\nonumber
\leq
2\mathcal{D}^{\bar L}_{\varepsilon/2}(X,\mathcal{N}_\Lambda)
\\&~~~~
\nonumber
+2\int_{\mathbb{R}^N}
\bigg(
\mathbb{E}\big[ (\tilde \psi^z)_\varepsilon (X) \big]
-\int_{\mathbb{R}^N} (\tilde \psi^z)_\varepsilon(x) \mathcal{N}_{\Lambda}(x) \,dx
\bigg) \mathcal{N}_{\Lambda}(z) \,dz
\\&~~~~
\nonumber
+\big(2K + 3\sqrt{N}\big)\varepsilon.
\end{align}
Finally, notice that Lemma~\ref{PropertiesOfFunctionClass}d entails (with $h(x):=\frac{1}{|z|+K+1} \osc_{(|z|+K)\varepsilon} \phi$ and $\delta:=\varepsilon$ as well as $K:=2\sqrt{N}$; note that $\osc_r h(x)\leq \frac{1}{|z|+K+1} \osc_{(|z|+K)\varepsilon+r} \phi(x)$)
\begin{align*}
\frac{1}{40N(|z|+K+1)} (\mathcal{N}_{\varepsilon^2 \Id_N} \ast \osc_{(|z|+K)\varepsilon}\phi) \in \Phi_\Lambda^{\tilde L}
\end{align*}
for any $\tilde L\geq 4^N (|\Lambda^{1/2}|^N \varepsilon^{-N} + 1)$, i.\,e.\ in particular for $\tilde L={\bar L}$.
Using $\psi^z=\mathcal{N}_{\varepsilon^2(\Id_N-\Lambda)} \ast (\mathcal{N}_{\varepsilon^2 \Id_N} \ast \osc_{(|z|+K)\varepsilon}\phi)$ (by definition \eqref{Definitionpsiz} and \eqref{ConvolutionGaussian}) and Lemma~\ref{PropertiesOfFunctionClass}a this yields
\begin{align*}
\frac{1}{40N(|z|+K+1)} \psi^z \in \Phi_\Lambda^{\bar L}.
\end{align*}
Consequently, the function $\hat \psi^z(x):=\psi^z(x/\sqrt{1-\varepsilon^2/4})$ satisfies by Lemma~\ref{PropertiesOfFunctionClass}b
\begin{align*}
\frac{1}{2\times 40N(|z|+K+1)} \hat \psi^z \in \Phi_\Lambda^{\bar L}.
\end{align*}
Applying the considerations around \eqref{RewriteMollification} to $\psi:=\psi^z/(2\times 40N(|z|+K+1))$, we deduce from \eqref{EstimateMollifiedDistanceSecondIntermediate}
\begin{align*}
&\Bigg|\mathbb{E}[\phi(X)]-\int_{\mathbb{R}^N} \phi(x)\mathcal{N}_\Lambda(x) \,dx\Bigg|
\\&
\leq
2\mathcal{D}^{\bar L}_{\varepsilon/2}(X,\mathcal{N}_\Lambda)
\\&~~~~
+2\int_{\mathbb{R}^N} 2\times 40N(|z|+K+1)
\mathcal{D}_{\varepsilon/2}^{\bar L}(X,\mathcal{N}_\Lambda) \mathcal{N}_{\Lambda}(z) \,dz
\\&~~~~
+\big(2K+3\sqrt{N}\big)\varepsilon
\end{align*}
which enables us to conclude
\begin{align*}
&\Bigg|\mathbb{E}[\phi(X)]-\int_{\mathbb{R}^N} \phi(x)\mathcal{N}_\Lambda(x) \,dx\Bigg|
\\&
\leq
2\mathcal{D}_{\varepsilon/2}^{\bar L}(X,\mathcal{N}_\Lambda)
\\&~~~~
+2\mathcal{D}_{\varepsilon/2}^{\bar L}(X,\mathcal{N}_\Lambda) \times 2 \times 40N \times (\sqrt{N} + K +1) 
\\&~~~~
+\big(2K+3\sqrt{N}\big)\varepsilon.
\end{align*}
Recalling that we have assumed $|\Lambda|\leq 1$ and that $K=2\sqrt{N}$, we deduce \eqref{SmoothingEstimateSimplified}.
\end{proof}

\begin{lemma}
\label{PropertiesOfFunctionClass}
The function classes $\Phi_\Lambda^{\bar L}$ and the oscillation operator $\osc$ are subject to the following properties:
\begin{itemize}
\item[a)] Given any $b\in L^1(\mathbb{R}^N)$ with $\int_{\mathbb{R}^N} |b(x)| \,dx\leq 1$ and any $\psi \in \Phi_\Lambda^{\bar L}$, the convolution $b\ast \psi$ satisfies
\begin{align*}
b \ast \psi \in \Phi_\Lambda^{\bar L}.
\end{align*}
\item[b)] Given any $\psi \in \Phi_\Lambda^{\bar L}$ and any $\tau \geq 1$, the rescaled function
\begin{align*}
\psi_\tau(x):=\frac{1}{\tau}\psi(\tau x)
\end{align*}
satisfies $\psi_\tau \in \Phi_\Lambda^{\bar L}$.
\item[c)]
For $K\geq 2\sqrt{N}$, $\delta>0$, $\varepsilon>0$, $r>0$, and any function $\psi \in L^\infty_{loc}(\mathbb{R}^N;V)$ for any normed vector space $V$, the estimates
\begin{align*}
\osc_{\delta} \psi (x)
\leq
2(\mathcal{N}_{\varepsilon^2 \Id_N} \ast \osc_{K\varepsilon+\delta} \psi)(x)
\end{align*}
and
\begin{align*}
\osc_{\delta} \psi (x)
\leq
2(\mathcal{N}_{2\varepsilon^2 \Id_N} \ast \osc_{K\varepsilon+\delta} \psi)(x)
\end{align*}
hold.
\item[d)]
Let $\delta>0$ and let $h\in L^\infty_{loc}(\mathbb{R}^N)$ be any function subject to the properties
\begin{align*}
\int_{\mathbb{R}^N} |h|(x) ~ \mathcal{N}_\Lambda(x-x_0) \,dx \leq \delta
\end{align*}
for any $x_0\in \mathbb{R}^N$ and
\begin{align*}
\int_{\mathbb{R}^N} \osc_r h (x) ~ \mathcal{N}_\Lambda(x-x_0) \,dx \leq r
\end{align*}
for any $r\geq \delta$ and any $x_0\in \mathbb{R}^N$. We then have
\begin{align*}
\frac{1}{40N}
\mathcal{N}_{\delta^2 \Id_N} \ast h
\in \Phi_\Lambda^{\tilde L}
\end{align*}
for any $\tilde L\geq 4^N (|\Lambda^{1/2}|^N \delta^{-N}+1)$.
\end{itemize}
\end{lemma}
\begin{proof}
To establish d), we choose $\tau:=2/N$ and compute for $r\geq \delta$
\begin{align*}
&\int_{\mathbb{R}^N} \osc_r (\mathcal{N}_{\delta^2 \Id_N} \ast h)(x) \mathcal{N}_\Lambda(x-x_0) \,dx
\\&
=\int_{\mathbb{R}^N} \osc_r \bigg(\int_{\mathbb{R}^N} \mathcal{N}_{\delta^2 \Id_N}(w) h(\cdot-w) \,dw \bigg)(x) \mathcal{N}_\Lambda(x-x_0) \,dx
\\&
\leq \int_{\mathbb{R}^N} \int_{\mathbb{R}^N} \mathcal{N}_{\delta^2 \Id_N}(w) \osc_r h(x-w) \,dw ~ \mathcal{N}_\Lambda(x-x_0) \,dx
\\&
= \int_{\mathbb{R}^N} \mathcal{N}_{\delta^2 \Id_N}(w)  \int_{\mathbb{R}^N} \osc_r h(x) \mathcal{N}_\Lambda(x+w-x_0) \,dx \,dw
\\&
\leq \int_{\mathbb{R}^N} \mathcal{N}_{\delta^2 \Id_N}(w) r \,dw
\\&
\leq r.
\end{align*}
For $r\in [\frac{1}{2N}\delta,\delta]$, upon replacing $r$ by $\delta$ we obtain
\begin{align*}
\int_{\mathbb{R}^N} \osc_r (\mathcal{N}_{\delta^2 \Id_N} \ast h)(x) \mathcal{N}_\Lambda(x-x_0) \,dx
\leq \delta \leq 2N r.
\end{align*}
In contrast, for $r\leq \frac{1}{2N} \delta$ we estimate
\begin{align*}
&\int_{\mathbb{R}^N} \osc_r (\mathcal{N}_{\delta^2 \Id_N} \ast h)(x) \mathcal{N}_\Lambda(x-x_0) \,dx
\\&
=\int_{\mathbb{R}^N} \osc_r \bigg(\int_{\mathbb{R}^N} \mathcal{N}_{\delta^2 \Id_N}(\cdot-w) h(w) \,dw \bigg)(x) \mathcal{N}_\Lambda(x-x_0) \,dx
\\&
\leq \int_{\mathbb{R}^N} \int_{\mathbb{R}^N} \osc_r\mathcal{N}_{\delta^2 \Id_N}(x-w) |h(w)| \,dw ~ \mathcal{N}_\Lambda(x-x_0) \,dx
\\&
\stackrel{\eqref{OscBoundGaussian}}{\leq} \int_{\mathbb{R}^N} \int_{\mathbb{R}^N} \frac{r}{\delta} \, 40\sqrt{N} \mathcal{N}_{(1+\tau)\delta^2 \Id_N}(x-w) |h(w)| \,dw ~ \mathcal{N}_\Lambda(x-x_0) \,dx
\\&
= \frac{r}{\delta} \, 40\sqrt{N} \int_{\mathbb{R}^N} \mathcal{N}_{(1+\tau)\delta^2 \Id_N}(w) \int_{\mathbb{R}^N} |h(x-w)| ~ \mathcal{N}_\Lambda(x-x_0) \,dx \,dw
\\&
\leq \frac{r}{\delta} \, 40\sqrt{N} \int_{\mathbb{R}^N} \mathcal{N}_{(1+\tau)\delta^2 \Id_N}(w) \delta \,dw
\\&
\leq 40\sqrt{N} \,r.
\end{align*}
To complete the proof of d), it only remains to show the estimate on the Lipschitz constant of $\mathcal{N}_{\delta^2 \Id_N} \ast h$. To do so, we estimate
\begin{align*}
&|\nabla (\mathcal{N}_{\delta^2 \Id_N} \ast h)(x)|
\\
&\leq
\int_{\mathbb{R}^N} |\nabla \mathcal{N}_{\delta^2 \Id_N}(x-w)| |h(w)| \,dw
\\&
=
\int_{\mathbb{R}^N} |\nabla \mathcal{N}_{\delta^2 \Id_N}(w)| |h(x-w)| \,dw
\\&
\leq
\int_{\mathbb{R}^N} |\delta^{-2} w| \, \mathcal{N}_{\delta^2 \Id_N}(w) \, |h(x-w)| \,dw
\\&
\leq \sqrt{2}^N 
\int_{\mathbb{R}^N} |\delta^{-2} w| \exp\bigg(-\frac{1}{4}\delta^{-2} |w|^2\bigg) \, \mathcal{N}_{2\delta^2 \Id_N}(w) \, |h(x-w)| \,dw
\\&
\leq \sqrt{2}^N \delta^{-1}
\int_{\mathbb{R}^N} \mathcal{N}_{2\delta^2 \Id_N}(w) \, |h(x-w)| \,dw
\\&
\leq \sqrt{2}^N \delta^{-1} \bigg(\frac{\sqrt{2}\delta+|\Lambda^{1/2}|}{\sqrt{2}\delta} \bigg)^N
\int_{\mathbb{R}^N} \mathcal{N}_{2\delta^2 \Id_N+\Lambda}(w) \, |h(x-w)| \,dw
\\&
\leq \sqrt{2}^N \delta^{-1} \bigg(\frac{\sqrt{2}\delta+|\Lambda^{1/2}|}{\sqrt{2}\delta} \bigg)^N
\int_{\mathbb{R}^N} \int_{\mathbb{R}^N} \mathcal{N}_{2\delta\Id_N}(z) \mathcal{N}_{\Lambda}(w-z) \, |h(x-w)| \,dw \,dz
\\&
\leq \sqrt{2}^N \delta^{-1} \bigg(\frac{\sqrt{2}\delta+|\Lambda^{1/2}|}{\sqrt{2}\delta} \bigg)^N
\int_{\mathbb{R}^N} \delta \mathcal{N}_{2\delta\Id_N}(z) \,dz
\\&
\leq \sqrt{2}^N \bigg(\frac{\sqrt{2}\delta+|\Lambda^{1/2}|}{\sqrt{2}\delta} \bigg)^N.
\end{align*}
This establishes assertion d).

Assertion a) is a direct consequence of the estimates
\begin{align*}
|\nabla (b\ast \psi)(x)| = \bigg| \int_{\mathbb{R}^N} b(w) \nabla \psi(x-w) \,dw \bigg|
\leq \int_{\mathbb{R}^N} |b(w)| \bar L \,dw \leq \bar L
\end{align*}
and
\begin{align*}
&\int_{\mathbb{R}^N} \osc_r (b\ast \psi)(x) \mathcal{N}_\Lambda(x-x_0) \,dx
\\&
=\int_{\mathbb{R}^N} \osc_r \bigg(\int_{\mathbb{R}^N} b(w) \psi(\cdot-w) \,dw\bigg)(x) \mathcal{N}_\Lambda(x-x_0) \,dx
\\&
\leq \int_{\mathbb{R}^N} \int_{\mathbb{R}^N} |b(w)| \osc_r \psi(x-w) \,dw ~ \mathcal{N}_\Lambda(x-x_0) \,dx
\\&
=\int_{\mathbb{R}^N} \int_{\mathbb{R}^N}  \osc_r \psi(x) \, \mathcal{N}_\Lambda(x+w-x_0) \,dx ~|b(w)| \,dw
\\&
\leq \int_{\mathbb{R}^N} r |b(w)| \,dw
\\&
\leq r.
\end{align*}

Concerning b), we directly verify $|\nabla \psi_\tau(x)|=|\nabla \psi(\tau x)|\leq \bar L$. On the other hand, we have for any $r>0$ using the convolution property $\mathcal{N}_{\tau^2 \Lambda}=\mathcal{N}_{\Lambda} \ast \mathcal{N}_{(\tau^2-1) \Lambda}$
\begin{align*}
&\int_{\mathbb{R}^N} \osc_r \psi_\tau(x) \mathcal{N}_\Lambda(x-x_0) \,dx
\\&
=\int_{\mathbb{R}^N} \frac{1}{\tau} \osc_{\tau r} \psi(\tau x) \mathcal{N}_\Lambda(x-x_0) \,dx
\\&
=\int_{\mathbb{R}^N} \frac{1}{\tau} \osc_{\tau r} \psi(\tilde x) \mathcal{N}_{\tau^2 \Lambda}(\tilde x-\tau x_0) \,d\tilde x
\\&
=\frac{1}{\tau} \int_{\mathbb{R}^N}\int_{\mathbb{R}^N} \osc_{\tau r} \psi(\tilde x) \mathcal{N}_{(\tau^2-1) \Lambda} (w) \mathcal{N}_\Lambda (\tilde x-w-\tau x_0) \,d\tilde x \,dw
\\&
\leq \frac{1}{\tau} \int_{\mathbb{R}^N} \tau r \, \mathcal{N}_{(\tau^2-1) \Lambda} (w) \,dw
\\&
\leq r.
\end{align*}
This establishes b).

Regarding c), we have by $\osc_\delta\psi(x)\leq \osc_{K\varepsilon+\delta}\psi(x-z)$ for any $|z|\leq K\varepsilon$
\begin{align*}
&\osc_{\delta} \psi (x)
\\&
\leq \frac{1}{\int_{\{|z|\leq K \varepsilon\}} \mathcal{N}_{\varepsilon^2 \Id_N}(z) \,dz} \int_{\mathbb{R}^N} \mathcal{N}_{\varepsilon^2 \Id_N}(z) \osc_{K\varepsilon+\delta}\psi(x-z) \,dz
\\&
=
\frac{1}{\int_{\{|z|\leq K\}} \mathcal{N}_{\Id_N}(z) \,dz} 
(\mathcal{N}_{\varepsilon^2\Id_N} \ast \osc_{K\varepsilon+\delta}\psi)(x)
\end{align*}
which for $K\geq \sqrt{2N}$ entails the desired bounds (to obtain the second one, we replace $\varepsilon$ by $\sqrt{2} \varepsilon$) using
\begin{align}
\label{MassGaussianInBall}
\int_{\{|z|\leq \sqrt{2N}\}} \mathcal{N}_{\Id_N}(z) \,dz
\geq \frac{1}{2}.
\end{align}
The latter estimate follows from $\int_{\mathbb{R}^d} \mathcal{N}_{\Id_N}(z) \,dz=1$ and
\begin{align*}
\int_{\{|z|\geq r\}} \mathcal{N}_{\Id_N}(z) \,dz
\leq
\frac{1}{r^2}
\int_{\{|z|\geq r\}} |z|^2 \mathcal{N}_{\Id_N}(z) \,dz
\leq
\frac{N}{r^2}.
\end{align*}
\end{proof}

In the following lemma, we collect a number of standard properties of Gaussians.
\begin{lemma}
Let $\Lambda\in \mathbb{R}^{N\times N}$  be a symmetric positive definite matrix.

For any symmetric positive definite matrix $\tilde \Lambda\in \mathbb{R}^{N\times N}$, we have
\begin{align}
\label{ConvolutionGaussian}
\mathcal{N}_\Lambda \ast \mathcal{N}_{\tilde \Lambda} = \mathcal{N}_{\Lambda+\tilde\Lambda}.
\end{align}

For any $0<\tau\leq 1$ the second and the third derivative of the Gaussian $\mathcal{N}_\Lambda$ satisfy the bounds
\begin{align}
\label{BoundGaussianSecondDerivative}
\frac{|\nabla^2 \mathcal{N}_\Lambda(z)|}{\mathcal{N}_{(1+\tau)\Lambda}(z)}
&\leq 3 (1+\tau)^{(N+2)/2} \tau^{-1} |\Lambda^{-1}|
\end{align}
and
\begin{align}
\label{BoundGaussianThirdDerivative}
\frac{|\nabla^3 \mathcal{N}_\Lambda(z)|}{\mathcal{N}_{(1+\tau)\Lambda}(z)}
&\leq 5 (1+\tau)^{(N+3)/2} \tau^{-3/2} |\Lambda^{-1/2}|^3
\end{align}
for all $z\in \mathbb{R}^N$.

For any $0\leq s\leq 1$ we have the multiplication property
\begin{align}
\label{MultiplyGaussians}
&\mathcal{N}_{\Lambda}(z)
\mathcal{N}_{\Lambda}(x)
=
\mathcal{N}_{\Lambda}(\sqrt{1-s}z+\sqrt{s}x)
\mathcal{N}_{\Lambda}(\sqrt{1-s}x-\sqrt{s}z)
\end{align}
for all $x,z\in \mathbb{R}^{N}$.

Let $0<\tau\leq 1$ and $r\leq \frac{1}{4} \tau \delta |\Lambda^{-1}|^{-1/2}$. Then the Gaussian $\mathcal{N}_{\delta^2 \Lambda}$ satisfies the estimate
\begin{align}
\label{OscBoundGaussianGeneral}
\osc_r \mathcal{N}_{\delta^2 \Lambda}(z)
\leq \frac{r}{\delta} \times \frac{20}{\tau^{1/2}} (1+\tau)^{N/2} |\Lambda^{-1/2}| \mathcal{N}_{(1+\tau) \delta^2 \Lambda}(z).
\end{align}
In particular, for $\tau:=2/N$ and $r\leq \frac{1}{2N} \delta |\Lambda^{-1}|^{-1/2}$ we have
\begin{align}
\label{OscBoundGaussian}
\osc_r \mathcal{N}_{\delta^2 \Lambda}(z)
\leq \frac{r}{\delta} \times 40 \sqrt{N} |\Lambda^{-1/2}| \mathcal{N}_{(1+\tau) \delta^2 \Lambda}(z).
\end{align}
\end{lemma}
\begin{proof}
The convolution property of Gaussians \eqref{ConvolutionGaussian} is classical.
Equation \eqref{MultiplyGaussians} is immediate from the definition
\begin{align*}
\mathcal{N}_\Lambda(z):=\frac{1}{(2\pi)^{N/2} \sqrt{\det \Lambda}}
\exp\bigg(-\frac{1}{2}\Lambda^{-1} z \cdot z\bigg).
\end{align*}

The estimate \eqref{BoundGaussianSecondDerivative} follows from
\begin{align*}
\nabla^2 \mathcal{N}_\Lambda (z) =
\Big(\Lambda^{-1} + \Lambda^{-1} z \otimes \Lambda^{-1} z \Big) \mathcal{N}_\Lambda (z)
\end{align*}
and
\begin{align}
\label{BoundOnGaussianByWiderGaussian}
\frac{\mathcal{N}_{\Lambda}(z)}{\mathcal{N}_{(1+\tau)\Lambda}(z)}
\leq (1+\tau)^{N/2} \exp\bigg(-\frac{1}{2}\frac{\tau}{1+\tau} \Lambda^{-1} z \cdot z\bigg)
\end{align}
as well as the bound for all $b\geq 0$
\begin{align*}
b^2 \exp\bigg(-\frac{1}{2} \frac{\tau}{1+\tau} b^2 \bigg)
\leq \frac{1+\tau}{\tau} \cdot \frac{2}{e}.
\end{align*}
Similarly, we deduce \eqref{BoundGaussianThirdDerivative} by applying \eqref{BoundOnGaussianByWiderGaussian} to
\begin{align*}
\nabla^3 \mathcal{N}_\Lambda (z) =
\Big(3\big(\Lambda^{-1} \otimes \Lambda^{-1} z\big)_{sym} + \Lambda^{-1} z \otimes \Lambda^{-1} z \otimes \Lambda^{-1} z \Big) \mathcal{N}_\Lambda (z)
\end{align*}
and using the estimate (valid for all $b\geq 0$)
\begin{align*}
(3 b + b^3) \exp\bigg(-\frac{1}{2} \frac{\tau}{1+\tau} b^2 \bigg)
&\leq 3\frac{(1+\tau)^{1/2}}{\tau^{1/2}} \cdot \frac{1}{\sqrt{e}}+\frac{(1+\tau)^{3/2}}{\tau^{3/2}} \cdot \frac{3\cdot \sqrt{3}}{e^{3/2}}
\\&
\leq 5\frac{(1+\tau)^{3/2}}{\tau^{3/2}}.
\end{align*}

Finally, we turn to the proof of \eqref{OscBoundGaussianGeneral}, from which \eqref{OscBoundGaussian} follows as an immediate consequence.
By the estimate
\begin{align*}
|\exp(-a)-\exp(-(a+b))|\leq b \exp(-a)
\end{align*}
(valid for any $a,b\geq 0$) we deduce for $\tau\leq 1$ the bound
\begin{align*}
&\osc_r \mathcal{N}_{\delta^2 \Lambda}(z)
\\&
\leq \frac{1}{(2\pi)^{N/2} \sqrt{\det \delta^2\Lambda}}
\big(2\delta^{-2} |\Lambda^{-1/2} z| |\Lambda^{-1/2}| r + \delta^{-2} |\Lambda^{-1}| r^2\big)
\\&~~~~
\times
\exp\bigg(-\frac{1}{2} \delta^{-2} \Lambda^{-1} z\cdot z + \delta^{-2} |\Lambda^{-1/2} z| |\Lambda^{-1/2}|  r + \frac{1}{2} \delta^{-2} |\Lambda^{-1}| r^2 \bigg)
\\&
\leq \delta^{-2} (1+\tau)^{N/2} \mathcal{N}_{(1+\tau) \delta^2 \Lambda}(z)
~\sup_x\Bigg[
\big(2|\Lambda^{-1/2} x| |\Lambda^{-1/2}|  r + |\Lambda^{-1}| r^2\big)
\\&~~~~~~~~~~~~
\times
\exp\bigg(-\frac{1}{2} \frac{\tau}{1+\tau} \delta^{-2} |\Lambda^{-1/2} x|^2 + \delta^{-2} |\Lambda^{-1/2} x| |\Lambda^{-1/2}|  r + \frac{1}{2} \delta^{-2} |\Lambda^{-1}| r^2 \bigg)\Bigg]
\\&
\leq \delta^{-2} (1+\tau)^{N/2} \mathcal{N}_{(1+\tau) \delta^2 \Lambda}(z)
~\sup_x \Bigg[
\big(2|\Lambda^{-1/2} x| |\Lambda^{-1/2}| r + |\Lambda^{-1}| r^2\big)
\\&~~~~~~~~~~~~~~~~~~~~~~~~~~~~~~
\times
\exp\bigg(-\frac{\tau}{8} \delta^{-2} |\Lambda^{-1/2} x|^2 \bigg) \times \exp \bigg( \frac{1+4\tau^{-1}}{2} \delta^{-2} |\Lambda^{-1}| r^2 \bigg)\Bigg]
\end{align*}
which for $r\leq \frac{1}{4} \tau \delta |\Lambda^{-1}|^{-1/2}$ entails
\begin{align*}
&\osc_r \mathcal{N}_{\delta^2 \Lambda}(z)
\\&
\leq \delta^{-2} (1+\tau)^{N/2} \mathcal{N}_{(1+\tau) \delta^2 \Lambda}(z)  \bigg(2|\Lambda^{-1/2}|r \sup_{b\geq 0} b \exp\Big(-\frac{\tau}{8} \delta^{-2} b^2\Big) + \delta |\Lambda^{-1/2}|r\bigg) \times 2
\\&
\leq \delta^{-2} (1+\tau)^{N/2} \mathcal{N}_{(1+\tau) \delta^2 \Lambda}(z)  \big(8 \delta \tau^{-1/2}|\Lambda^{-1/2}|r + \delta |\Lambda^{-1/2}|r\big) \times 2.
\end{align*}
As a consequence, we deduce \eqref{OscBoundGaussianGeneral}.
\end{proof}

\section{Concentration inequalities for independent random variables}

A concentration estimate for the sum of bounded independent random variables is provided by Bennett's inequality.
\begin{lemma}[Bennett's inequality]
\label{Bennett}
Let $X_1,\ldots,X_M$ be independent random variables with $\mathbb{E}[X_i]=0$ for all $1\leq i\leq M$. Suppose that for some $A>0$ the $X_i$ satisfy the uniform bound $|X_i|\leq A$ almost surely. Defining
\begin{align*}
\sigma^2 := \sum_{i=1}^M \Var X_i,
\end{align*}
and using the notation $h(x):=(1+x) \log (1+x) -x$, we have for any $r\geq 0$ the estimate
\begin{align*}
\mathbb{P}\bigg[\sum_{i=1}^M X_i \geq r\bigg]
\leq
\exp\bigg(-\frac{\sigma^2}{A^2} h\bigg(\frac{Ar}{\sigma^2}\bigg)\bigg).
\end{align*}
In particular, we have
\begin{align*}
\mathbb{P}\bigg[\sum_{i=1}^M X_i \geq r\bigg]
\leq
\exp\bigg(-\min\bigg\{\frac{r^2}{3\sigma^2},\frac{r}{3A}\bigg\} \bigg).
\end{align*}
\end{lemma}
\begin{proof}
For the proof of Bennett's inequality -- the first inequality of the lemma -- see \cite{Bennett}.
The second inequality follows by distinguishing the cases $r\leq \frac{\sigma^2}{A}$, in which Bennett's inequality gives
\begin{align*}
\mathbb{P}\bigg[\sum_{i=1}^M X_i \geq r\bigg]
\leq
\exp\bigg(-\frac{r^2}{3\sigma^2} \bigg),
\end{align*}
and $r\geq \frac{\sigma^2}{A}$, in which case we deduce
\begin{align*}
\mathbb{P}\bigg[\sum_{i=1}^M X_i \geq r\bigg]
\leq
\exp\bigg(-\frac{r}{3A}\bigg).
\end{align*}
\end{proof}
As a corollary of Bennett's inequality, we deduce the following concentration inequality for random variables with stretched exponential moments.
\begin{lemma}
\label{iidSumTailBound}
Let $X_1,\ldots,X_M$ be independent random variables with zero mean and uniformly bounded stretched exponential moments
\begin{align*}
\mathbb{E}\bigg[\exp\bigg(\frac{|X_i|^{\gamma_0}}{b^{\gamma_0}}\bigg)\bigg]\leq 2
\end{align*}
for some $\gamma_0>0$ and some $b>0$. The sum
\begin{align*}
X:=\sum_{i=1}^M X_i
\end{align*}
then satisfies for any $V\geq \Var X$ the estimate
\begin{align*}
\mathbb{P}\bigg[\bigg|\sum_{i=1}^M X_i\bigg|\geq r\bigg]
\leq
3\exp\bigg(-\frac{r^2}{10V} \bigg)
\end{align*}
for any
\begin{align}
\label{ConditionOnr}
r\leq \sqrt{V} \min\bigg\{\frac{\sqrt{V}}{b (2\log (2M))^{1/\gamma_0}},\bigg(\frac{\sqrt{V}}{b}\bigg)^{\gamma_0/(2+\gamma_0)}
\bigg\}.
\end{align}
\end{lemma}
\begin{proof}
Fixing $A>0$, we split each random variable $X_i$ according to
\begin{align*}
X_i=X_i^{bulk}+X_i^{tail}
\end{align*}
with
\begin{align*}
X_i^{bulk}:&=X_i \chi_{|X_i|\leq A},
\\
X_i^{tail}:&=X_i \chi_{|X_i|> A}.
\end{align*}
By Lemma~\ref{Bennett}, we get
\begin{align*}
\mathbb{P}\bigg[\bigg|\sum_{i=1}^M (X_i^{bulk}-\mathbb{E}[X_i^{bulk}])\bigg| \geq r\bigg]
\leq
2\exp\bigg(-\min\bigg\{\frac{r^2}{3\Var X},\frac{r}{3A}\bigg\} \bigg).
\end{align*}
Furthermore, we have the bound
\begin{align*}
\mathbb{P}\bigg[\exists i:X_i^{tail}\neq 0\bigg]
\leq 2M \exp\bigg(-\frac{A^{\gamma_0} }{b^{\gamma_0}}\bigg).
\end{align*}
Choosing $A:=\frac{V}{r}$ for some $V\geq \Var X$, we deduce
\begin{align}
\label{BoundLemmaiidSumTailBoundAlmost}
\mathbb{P}\bigg[\bigg|\sum_{i=1}^M (X_i-\mathbb{E}[X_i^{bulk}])\bigg|\geq r\bigg]
\leq
3\exp\bigg(-\frac{r^2}{3V} \bigg)
\end{align}
as long as $\log (2M)\leq \frac{A^{\gamma_0}}{2b^{\gamma_0}}=\frac{V^{\gamma_0}}{2(rb)^{\gamma_0}}$
and
\begin{align*}
-\frac{V^{\gamma_0}}{2r^{\gamma_0} b^{\gamma_0}}
\leq -\frac{r^2}{2V}.
\end{align*}
Note that the latter condition may be rewritten as $r^{2+\gamma_0} b^{\gamma_0} \leq V^{\gamma_0+1}$ which is satisfied in case \eqref{ConditionOnr}, while the former condition is
\begin{align*}
r\leq  \frac{V}{b (2\log (2M))^{1/\gamma_0}}
\end{align*}
which is also satisfied under our assumption \eqref{ConditionOnr}.

From the same estimates, we also deduce
\begin{align*}
|\mathbb{E}[X_i^{bulk}]| &= |\mathbb{E}[X_i^{tail}]| \leq \mathbb{E}[|X_i^{tail}|^2]^{1/2} \mathbb{E}[\chi_{|X_i|>A}]^{1/2}
\\&
\leq \sqrt{2\Var X}\exp\bigg(-\frac{A^{\gamma_0}}{2b^{\gamma_0}}\bigg)
\\&
\leq \sqrt{2V} \exp\bigg(-\frac{r^2}{2V}\bigg).
\end{align*}
Plugging this estimate into \eqref{BoundLemmaiidSumTailBoundAlmost} and estimating gives the desired result.
\end{proof}

For the sum of independent random variables, each of which is approximately a multivariate Gaussian, the following simple concentration estimate towards a Gaussian holds true.

\begin{lemma}
\label{ConcentrationCloseToGaussian}
Let $X_1,\ldots,X_M$ be independent $\mathbb{R}^N$-valued random variables with zero mean and uniformly bounded stretched exponential moments
\begin{align*}
\mathbb{E}\bigg[\exp\bigg(\frac{|X_i|^{\gamma_0}}{b^{\gamma_0}}\bigg)\bigg]\leq 2
\end{align*}
for some $\gamma_0>0$ and some $b>0$. Let $0<\tau\leq \frac{1}{2}$ and suppose that the probability distribution of each $X_m$ is close to a Gaussian $\mathcal{N}_{\Lambda_m}$ with $||\mathcal{N}_{\Lambda_m}||_{\exp^{\gamma_0}}\leq b$ in the $1$-Wasserstein distance
\begin{align*}
\mathcal{W}_1(X_m,\mathcal{N}_{\Lambda_m})\leq \tau b.
\end{align*}
Then there exists a probability space and random variables $Y$ and $Z$ such that the law of the sum
\begin{align*}
X:=\sum_{m=1}^M X_m
\end{align*}
coincides with the law of the sum
\begin{align*}
Y+Z,
\end{align*}
where $Y$ is a multivariate Gaussian with covariance matrix $\Lambda:=\sum_{m=1}^M \Lambda_m$ and where $Z$ is a random variable subject to the estimate
\begin{align*}
\mathbb{P}[|Z|\geq r]
\leq
3N \exp\bigg(-\frac{r^2}{10 N \tau |\log \tau|^{1/\gamma_0} M b^2} \bigg)
\end{align*}
for any
\begin{align*}
r\leq \sqrt{N \tau |\log \tau|^{1/\gamma_0}} \sqrt{M} b \min\bigg\{\frac{\sqrt{M \tau |\log \tau|^{1/\gamma_0}}}{(2\log (2M))^{1/\gamma_0}},\big(\tau |\log \tau|^{1/\gamma_0} M\big)^{\gamma_0/(4+2\gamma_0)}
\bigg\}.
\end{align*}
\end{lemma}
\begin{proof}
By the stochastic independence of the $X_m$ and the definition of the $1$-Wasserstein distance $\mathcal{W}_1(X_m,\mathcal{N}_{\Lambda_m})$, there exists a probability space and independent triples of random variables $(\tilde X_m,Y_m,Z_m)$ with $\tilde X_m\stackrel{d}{=}X_m$ and
\begin{align*}
\tilde X_m = Y_m+Z_m
\end{align*}
where $Y_m$ is a Gaussian random variable with covariance matrix $\Lambda_m$ and where $Z_m$ satisfies
\begin{align*}
\mathbb{E}[|Z_m|] = \mathcal{W}_1(X_m,\mathcal{N}_{\Lambda_m}).
\end{align*}
By writing $Z_m=\tilde X_m-Y_m$, we deduce using Lemma~\ref{LemmaStretchedExponentialQuasinorm}
\begin{align*}
||Z_m||_{\exp^{\gamma_0}} \leq ||X_m||_{\exp^{\gamma_0}} + ||Y_m||_{\exp^2}
\leq Cb+Cb
\end{align*}
and therefore for $A$ with $z^2/b^2 \leq \exp(z^{\gamma_0}/2b^{\gamma_0})$ for all $z\geq A$ (using again Lemma~\ref{LemmaStretchedExponentialQuasinorm})
\begin{align*}
\mathbb{E}[|Z_m|^2] &\leq A \mathbb{E}[|Z_m|] + \mathbb{E}[|Z_m|^2 \chi_{|Z_m|\geq A}]
\\&
\leq A \mathbb{E}[|Z_m|] + b^2 \exp(-A^{\gamma_0}/2b^{\gamma_0}) \mathbb{E}[\exp(|Z_m|^{\gamma_0}/b^{\gamma_0})]
\\&
\leq A \cdot Cb \tau + 2 b^2 \exp(-A^{\gamma_0}/2b^{\gamma_0}).
\end{align*}
Choosing $A:=C b |\log \tau|^{1/\gamma_0}$ (which satisfies the previous assumption) we get
\begin{align*}
\mathbb{E}[|Z_m|^2]
\leq C \tau b^2 |\log \tau|^{1/\gamma_0}.
\end{align*}

As the sum of independent Gaussian random variables is again a Gaussian random variable, we get that
\begin{align*}
Y=\sum_{m=1}^M Y_m
\end{align*}
is a multivariate Gaussian with covariance matrix $\Lambda=\sum_{m=1}^M \Lambda_m$. On the other hand, applying Lemma~\ref{iidSumTailBound} (note that Lemma~\ref{iidSumTailBound} applies to scalar-valued random variables) to each component of
\begin{align*}
Z:=\sum_{m=1}^M Z_m
\end{align*}
with the choice $V:=\sqrt{N} M\tau b^2 |\log \tau|^{1/\gamma_0}$,
we obtain
\begin{align*}
\mathbb{P}[|Z|\geq r]
\leq
3N \exp\bigg(-\frac{r^2}{10 N \tau |\log \tau|^{1/\gamma_0} M b^2} \bigg)
\end{align*}
for any
\begin{align*}
r\leq \sqrt{N \tau |\log \tau|^{1/\gamma_0}} \sqrt{M} b \min\bigg\{\frac{\sqrt{M \tau |\log \tau|^{1/\gamma_0}}}{(2\log (2M))^{1/\gamma_0}},\big(\tau |\log \tau|^{1/\gamma_0} M\big)^{\gamma_0/(4+2\gamma_0)}
\bigg\}.
\end{align*}
\end{proof}

\section{Calculus for random variables with stretched exponential moments}

\label{AppendixStretchedExponential}
Throughout this paper, we have equipped the space of random variables $X$ with stretched exponential moments in the sense
\begin{align*}
\mathbb{E}\bigg[\exp\bigg(\frac{|X|^\gamma}{C}\bigg)\bigg]\leq 2
\end{align*}
for some $\gamma>0$ and some $C>0$ with the norm
\begin{align*}
||X||_{\exp^\gamma} := \sup_{p\geq 1} \frac{1}{p^{1/\gamma}} \mathbb{E}\big[|X|^p\big]^{1/p}.
\end{align*}
In the setting of exponential or higher moments $\gamma\geq 1$, this norm is equivalent to the Luxemburg norm associated with the convex function $\exp(x^\gamma)-1$. However, it has two advantages over the Luxemburg norm: First, it simplifies calculus when considering the integrability of products of random variables or the concentration properties of independent random variables. Secondly and more importantly, it is also a well-defined norm on the space of random variables $X$ with subexponential stretched exponential moments $\gamma\in (0,1)$.
\begin{lemma}
\label{LemmaStretchedExponentialQuasinorm}
Let $\gamma>0$.
Consider a random variable $X$ on some probability space. Define the quasinorm
\begin{align*}
||X||_{\exp^\gamma,\operatorname{quasi}}:=
\inf\bigg\{s>0:
\mathbb{E}\bigg[\exp\bigg(\frac{|X|^\gamma}{s^\gamma}\bigg)\bigg]\leq 2
\bigg\}.
\end{align*}
Then we have $||X||_{\exp^\gamma,\operatorname{quasi}}<\infty$ if and only if $||X||_{\exp^\gamma}<\infty$ and there exist constants $c(\gamma),C(\gamma)$ such that the estimate
\begin{align*}
c(\gamma) ||X||_{\exp^\gamma}
\leq ||X||_{\exp^\gamma,\operatorname{quasi}}
\leq C(\gamma) ||X||_{\exp^\gamma}
\end{align*}
is satisfied.
\end{lemma}
We omit the (elementary) proof of this lemma and the next lemma; the proofs may be found in the companion article \cite{FischerVarianceReduction}.
\begin{lemma}[Calculus for random variables with stretched exponential moments]
\label{CalculusStretchedExponential}
Let $X$, $Y$ be random variables with stretched exponential moments in the sense $||X||_{\exp^\gamma}<\infty$ and $||Y||_{\exp^\beta}<\infty$ for some $\gamma,\beta>0$.
\begin{itemize}
\item[a)] The product $XY$ has stretched exponential moments with exponent $\alpha$ given by $\frac{1}{\alpha}=\frac{1}{\gamma}+\frac{1}{\beta}$ and satisfies the bound
\begin{align*}
||XY||_{\exp^\alpha}\leq C(\beta,\gamma) ||X||_{\exp^\gamma} ||Y||_{\exp^\beta}.
\end{align*}
\item[b)] There exists constants $c=c(\gamma)>0$, $C=C(\gamma)<\infty$, with the following property: For any $K\geq 0$, we have the estimate
\begin{align*}
\mathbb{P}\big[|X|\geq K||X||_{\exp^\gamma}\big] \leq C\exp(-c K^\gamma).
\end{align*}
\end{itemize}
\end{lemma}
For independent random variables with stretched exponential moments, the following simple concentration estimate holds.
Again, we omit the proof and refer the reader to \cite{FischerVarianceReduction}.
Note that the result is not optimal, see e.\,g.\  \cite{BoucheronLugosiMassart} for a stronger statement.
\begin{lemma}
\label{ConcentrationStretchedExponential}
Let $X_1,\ldots,X_M$ be independent random variables with vanishing expectation and uniformly bounded stretched exponential moments
\begin{align*}
||X_m||_{\exp^{\gamma_0}}\leq b
\end{align*}
for some $\gamma_0>0$ and some $b>0$. Then the sum
\begin{align*}
X:=\sum_{m=1}^M X_m
\end{align*}
has uniformly bounded stretched exponential moments
\begin{align*}
||X||_{\exp^{\tilde \gamma}} \leq C(\gamma_0) \sqrt{M} b
\end{align*}
for $\tilde \gamma := \gamma_0/(\gamma_0+1)$.
\end{lemma}

\bibliographystyle{abbrv}
\bibliography{stochastic_homogenization}

\end{document}